\documentclass[12pt,reqno]{amsart}

\usepackage{xcolor}
\usepackage{amsmath, amsthm, amssymb}
\usepackage{amsfonts}
\usepackage[ansinew]{inputenc}
\usepackage{graphicx}
\usepackage[english]{babel}
\usepackage{hyperref}
\theoremstyle{plain}
\newtheorem{thm}{Theorem}[section]
\newtheorem{cor}[thm]{Corollary}
\newtheorem{lem}[thm]{Lemma}
\newtheorem{prop}[thm]{Proposition}

\newtheorem{conjecture}[thm]{Conjecture}

\theoremstyle{definition}
\newtheorem{defi}[thm]{Definition}

\theoremstyle{remark}
\newtheorem{rem}[thm]{Remark}

\numberwithin{equation}{section}

\newcommand{\average}{{\mathchoice {\kern1ex\vcenter{\hrule height.4pt
width 6pt depth0pt} \kern-11pt} {\kern1ex\vcenter{\hrule
height.4pt width 4.3pt depth0pt} \kern-8pt} {} {} }}
\newcommand{\ave}{\average\int}

\def\Z{\mathbb{Z}}
\def\R{\mathbb{R}}

\newcommand{\ep}{\varepsilon}
\newcommand{\eps}{\varepsilon}

\begin{document}

\title[Stable solutions to the fractional Allen-Cahn equation]{Stable solutions to the fractional Allen-Cahn equation in the nonlocal perimeter regime}

\author{Xavier Cabr\'e}
\author{Eleonora Cinti}
\author{Joaquim Serra}

\address{X. Cabr\'e\textsuperscript{\,1,2},
\textsuperscript{1\,}ICREA, Pg. Lluis Companys 23, 08010 Barcelona, Spain, and \linebreak
\textsuperscript{2\,}Universitat Polit\`ecnica de Catalunya, Departament de Matem\`{a}tiques and IMTech, Diagonal 647,
08028 Barcelona, Spain}
\email{xavier.cabre@upc.edu}

\address{E. Cinti, Universit\`a degli Studi di Bologna, Dipartimento di Matematica, Piazza di Porta San Donato 5,
40126 Bologna, Italy}
\email{eleonora.cinti5@unibo.it}

\address{J. Serra, Eidgen\"ossische Technische Hochschule Z\"urich, R\"amistrasse 101, 8092 Zurich, Switzerland}
\email{joaquim.serra@math.ethz.ch}

\thanks{The authors are supported by grants MTM2017-84214-C2-1-P and RED2018-102650-T funded by MCIN/AEI/10.13039/501100011033 and by ``ERDF A way of making Europe'', and
are part of the Catalan research group 2017 SGR 1392. X.C. is member of the Barcelona
Graduate School of Mathematics. E.C. was supported by the Italian grant FFABR 2017. J.S. was supported by Swiss NSF Ambizione Grant PZ00P2 180042 and by the European Research Council (ERC) under the Grant Agreement No 948029.
}

\begin{abstract}
We study stable solutions to the fractional Allen-Cahn equation \linebreak  $(-\Delta)^{s/2} u = u-u^3$, $|u|<1$ in $\R^n$. For every  $s\in (0,1)$ and dimension $n\geq 2$, 
we establish sharp energy estimates, density estimates, and the convergence of blow-downs to stable nonlocal $s$-minimal cones.
As a consequence, we obtain a new classification result:  if for some pair $(n,s)$, { with $n\ge 3$}, hyperplanes  are the only stable nonlocal $s$-minimal cones in $\R^n\setminus\{0\}$, then every stable solution to the fractional Allen-Cahn equation in $\R^n$ is 1D, namely, its level sets are parallel hyperplanes.

Combining this result with the classification of stable $s$-minimal cones in $\R^3\setminus\{0\}$ for $s\sim 1$ obtained by the authors in a recent paper, we give positive answers to the ``stability conjecture'' in $\R^3$ and to the ``De Giorgi  conjecture'' in $\R^4$ for the fractional Allen-Cahn equation when the order $s\in (0,1)$ of the operator is sufficiently close to $1$.

\end{abstract}

\maketitle

\tableofcontents

\section{Introduction}\label{intro}

\subsection{Two classical models in phase transitions vs.\ minimal surfaces}\label{sub-1.1}

The Peierls-Nabarro model was introduced in the early 1940's in the context of crystal dislocations \cite{Peierls, Nabarro} and also arises in  the study of phase transitions with  line-tension effects \cite{ABS} and boundary vortices in thin magnetic films \cite{K}. This model concerns the energy functional
\[
I_\ep(u)  :=   
\frac \ep  4[u]_{H^{1/2}(\R^n)}^2 +\negmedspace\int_{\R^n} W(u)\,dx, \ \quad [u]_{H^{1/2}(\R^n)}^2 := \iint_{\R^n\times \R^n}\negmedspace \frac{\big|u(x)-u(\bar x)\big|^2}{|x-\bar x|^{n+1}}  \,dx \,d\bar x,
\]
where $u: \R^n \rightarrow (-1,1)$ and $W(u):= 1+\cos(\pi u)$. 

The very related Allen-Cahn functional, introduced later, in 1958, within the context of  the Van Der Walls-Cahn-Hilliard theory for phase transitions in fluids \cite{CahnHil}, is also tightly connected to the Ginzburg-Landau theory of superconductivity. It is defined as
\[
J_\ep(u) :=  \frac{\ep^2}{2} [u]_{H^{1}(\R^n)}^2 +\int_{\R^n} W(u)\,dx,   \qquad   [u]_{H^{1}(\R^n)}^2 := \int_{\R^n} |\nabla u|^2\,dx, 
\]
where $u: \R^n \rightarrow (-1,1)$ and $W(u):= \frac 1 4 (1-u^2)^2$. 

In both models $\ep>0$ is a parameter and $W(u)$ is a so-called {\em double-well potential}, namely, a function with two minima (or wells) at the values $u=-1$ and $u= +1$ which correspond to two ``stable phases''. 

Critical points $u\in C^2(\R^n)$ of $I_\ep$ and $J_\ep$ (more precisely, of their localized versions presented later in Subsection~\ref{subsectionstable}) solve,\footnote{Up to modifying the multiplicative constants in the definitions of $[\,\cdot\,]_{H^s}$ and $(-\Delta)^s$ in order to make them consistent.} respectively, the Peierls-Nabarro and Allen-Cahn equations:
\[
\ep(-\Delta)^{1/2}  u + W'(u) = 0 \qquad \mbox{and} \qquad  \ep^2(-\Delta) u + W'(u) = 0.
\]

A deep link between any of the two models and minimal surfaces is found when investigating the asymptotic behaviour of sequences of minimizers of  $I_\ep$ and $J_\ep$ as $\ep\downarrow 0$. Indeed, as a consequence of $\Gamma$-convergence results of Alberti, Bouchitt\'e, and Seppecher  \cite{ABS} and of Modica and Mortola  \cite{MM}, respectively, the following holds:

\vspace{4pt}
 {\it If $u_{\ep_k}:\R^n \rightarrow \R$ is a sequence of minimizers (in every bounded set) of either   $I_{\ep_k}$ or $J_{\ep_k}$
and $\ep_k\downarrow 0$, then (up to a subsequence)
\begin{equation}\label{assymp}
u_{\ep_k} \, \stackrel{L^1_{\rm loc}}{\longrightarrow} \, \chi_{E}-\chi_{\R^n\setminus E}, \qquad  \mbox{where $E\subset \R^n$ is a minimizer of the perimeter.} 
\end{equation}
}In other words, for every $\lambda\in(-1,1)$ the level sets $\{u_\ep=\lambda\}$ converge,  as $\ep\downarrow 0$ and up to a subsequence, to an area minimizing hypersurface (in particular minimal, i.e., with zero mean curvature).

\subsection{The fractional Allen-Cahn energies}\label{sub-1.2}
The classical functionals $I_\ep$ and $J_\ep$ introduced above belong to the more general family of Allen-Cahn energies
\[
E_{s, \ep} (u)  :=  \frac{\ep^s}{4} [u]_{H^{s/2}(\R^n)}^2 +\int_{\R^n} W(u)\,dx , 
\]
where $s\in (0,2]$ and
\[
\qquad [u]_{H^{s/2}(\R^n)}^2 :=  (2-s)\iint_{\R^n\times \R^n} \frac{\big|u(x)-u(\bar x)\big|^2}{|x-\bar x|^{n+s}} \,dx\,d\bar x.
\]
Note that the normalization factor $2-s$ in the definition  of $[u]_{H^{s/2}(\R^n)}^2$ guarantees that
\[[u]_{H^{s/2}(\R^n)}^2 \rightarrow \int_{\R^n} |\nabla u|^2\,dx\qquad \mbox{as $s\uparrow 2$}\]
 (up to a dimensional multiplicative constant).  
The functionals $E_{s,\ep}$  have been extensively studied in the literature; see among others \cite{C-SM, SV-1, C-Si1, CSi, CC2, S-new} and references therein.
Note that  the classical functionals $I_{\ep}$ and $J_\ep$ correspond to the cases $s=1$ and $s=2$ of $E_{s,\ep}$. 

Critical points $u: \R^n \to (-1,1)$ of $E_{s,\ep}$ (in fact, of the natural localized version of $E_{s,\ep}$ given in Subsection~\ref{subsectionstable})  satisfy the fractional Allen-Cahn equation 
\begin{equation}\label{fract}
\ep^s(-\Delta)^{s/2} u + W'(u) =0, \quad  |u|\le 1,
\end{equation} 
in $\R^n$ (up to a positive multiplicative constant in front of $(-\Delta)^{s/2} u$).
Similarly to the cases $s=1$ and $s=2$,  for every $s\in [1,2]$ the energies  $E_{s, \ep}$  (suitably localized and renormalized) $\Gamma$-converge to the classical perimeter as $\ep\downarrow 0$ ---see \cite{SV-gamma}. 
As a consequence,  if $u_{\ep_k}$ is a sequence of minimizers (in every bounded set) of  $E_{s,\ep_k}$ with $s\in [1,2]$, then   \eqref{assymp} holds (up to a subsequence).
Interestingly, a new qualitative behaviour is found for  $s\in(0,1)$, where the asymptotic analysis of $E_{s, \ep}$ as $\ep \downarrow 0$ leads to the so-called $s$-{\em minimal surfaces} (or {\em nonlocal $s$-minimal surfaces}), introduced by Caffarelli, Roquejoffre, and Savin in \cite{CRS}.  These new geometric objects generalize classical minimal surfaces (which are recovered as a limit case when $s\uparrow 1$) and share with them several structural properties; see \cite{CRS, FV, CSV, CCS} for the precise definition and several results on $s$-minimal surfaces.

\subsection{Bernstein's and De Giorgi's conjectures} \label{sub-1.3}

Let us recall the classical Bernstein's problem: 

{\it If the graph of a function defined in all of $\R^{n-1}$ is a minimal surface in $\R^n$, is the function necessarily linear? }

This question has deep connections with the regularity theory for the (co-dimen-\-sion one) Plateau problem in any dimension and, more precisely, with the optimal dimensional bounds on the singular set of minimal surfaces ---as shown by De Giorgi in \cite{DG-Bern}. Moreover, its study contributed to striking developments in the theory of minimal surfaces. The question was completely answered in the late 1960's in a series of groundbreaking works, which established that: 
\begin{enumerate}
\item[(i)]  Hyperplanes are the only graphs of functions defined on $\R^{n-1}$ which are minimal surfaces in $\R^n$ as long as $n\le 8$ (Bernstein \cite{Bern}, Fleming \cite{Fle},  De Giorgi \cite{DG1, DG-Bern}, Almgren \cite{Alm}, and Simons \cite{Simons}). 
\item[(ii)] There exist entire minimal graphs in $\R^n$ which are not hyperplanes in dimensions $n\ge 9$ (Bombieri, De Giorgi, and Giusti \cite{BDG}).
\end{enumerate}

Bernstein's problem finds a counterpart for certain ``critical points'' of $J_\ep$ in a famous conjecture by De Giorgi (1978):
\begin{conjecture}[\cite{DG}]\label{DGconj}{\it 
Let $u\in C^2(\R^{n})$, $|u|<1$, be a solution of $-\Delta u=u-u^3$ in~$\R^n$ satisfying $\partial_{x_{n}} u >0$. 
Then, if $n\le 8$,  $u$ must be 1D, that is, all its level sets $\{u=\lambda\}$ must be  hyperplanes.}
\end{conjecture}

Conjecture \ref{DGconj} was first proved, about twenty years after it was raised, in dimensions  $n=2$ and $n=3$, by Ghoussoub and Gui \cite{GG} and  Ambrosio and Cabr\'e \cite{AmbrC}, respectively.
Almost ten years later, in the celebrated paper  \cite{Savin}, Savin adressed the conjecture in the dimensions $4\le n\le 8$, and he succeeded in proving it under the additional assumption
\begin{equation}\label{limits}
 \lim_{x_{n}\to \pm \infty} u = \pm 1.
\end{equation}
Short after, Del Pino, Kowalczyk, and Wei \cite{dPKW} established the existence of a counterexample to the conjecture in dimensions  $n\ge 9$.

Since for $s\in [1,2)$ the functionals $J_\ep$ and $E_{s,\ep}$ have the identical asymptotic behaviour \eqref{assymp},  there are no heuristic reasons to prefer $J_\ep$ to $E_{s,\ep}$ when $s\in [1,2)$, in the statement of the De Giorgi conjecture. This motivates
\begin{conjecture}\label{DGconj2}
Conjecture \ref{DGconj}  also holds for  $(-\Delta)^{s/2} u=u-u^3$ when $s\in [1,2)$. 
\end{conjecture}

On the other hand, for nonlocal $s$-minimal surfaces  the analogue of the Bernstein's problem is well understood for $s\in (0,1)$ sufficiently close to $1$. Indeed, it follows, by 
combining the results in \cite{CRS},  \cite{CV},  \cite{SV}, and \cite{FV} that: 
\vspace{6pt}

{
\em For $ n\le 8$, there is a dimensional constant $s_*\in [0,1)$ such that, if $s\in [s_*,1)$, hyperplanes are the only graphs $\{x_n = \phi(x_1, \dots, x_{n-1})\}$ which are $s$-minimal surfaces in $\R^n$ ---for $n=2$ and $n=3$ this is known with $s_*=0$.
}

\vspace{6pt}
The heuristics thus suggest that the De Giorgi conjecture should be true for monotone critical points of $E_{\ep,s}$ also in the range $s\in (s_*,1)$. That is:
\begin{conjecture}\label{DGconj2b}{\it 
Conjecture \ref{DGconj}  also holds for  $(-\Delta)^{s/2} u=u-u^3$ when $s\in (s_*,1)$, where $s_*\in[0,1)$ is a dimensional constant. 
}
\end{conjecture}

The articles by Cabr\'e and Sol\`a-Morales \cite{C-SM}, Cabr\'e and  Sire \cite{CSi}, and  Sire and Valdinoci \cite{Si-V} proved Conjectures \ref{DGconj2}  and \ref{DGconj2b} in dimension $n=2$, for $s$ in the whole range $(0,2)$ ---that is $s_*=0$.
Later, Cabr\'e and Cinti \cite{CC1,CC2} established Conjecture~\ref{DGconj2} for $n=3$, $s\in [1,2)$. {{In \cite{DFV}, Dipierro, Farina, and Valdinoci proved Conjecture 1.3 for $n = 3$ and $s_* = 0$.}}
Very recently, Figalli and Serra \cite{FS} have proved it for $n=4$, $s=1$. 
 For all $s\in (1,2)$, the existence of monotone solutions to \eqref{fract} which are not $1D$ in dimensions $n\ge9$ has been  announced  in \cite{CLW} ---to appear in a work by Chan, D\'avila, del Pino, Liu, and Wei \cite{CDPLW}.

In this paper we will prove Conjecture \ref{DGconj2b} in dimension $n=4$ for $s\in (s_*,1)$, for some $s_*\in (0,1)$. We will also give a sufficient condition (in terms of a rigidity statement on $s$-minimal cones) guaranteeing that Conjecture \ref{DGconj2b} holds for  some pair $(n, s)$, with $s\in (0,1)$ and $n\ge 3$.

Under the additional assumption \eqref{limits}, Conjecture~\ref{DGconj2} has been proved by Savin \cite{S-new,S-new2} for  $4\le n\le 8$, $s\in [1,2)$.
Finally, Conjectures \ref{DGconj2}  and \ref{DGconj2b} have been proven, also under the additional assumption \eqref{limits}, in  Dipierro, Serra, and Valdinoci \cite{dPSV} in two situations: for $n=3$ and $s_*=0$, as well as  for $4\le n\le 8$ and  $s\in (s_*,1)$ with $s_*<1$ is sufficiently close to $1$.

\subsection{Monotone vs.\ stable solutions} \label{sub-1.4}
The assumption  $\partial_{x_{n}} u >0$ easily yields that $u$ is a {\em stable solution}, i.e., a critical point of the localized version of $E_{s,\ep}$ (presented below in Subsection~\ref{subsectionstable}) with nonnegative second variation.
The additional assumption \eqref{limits} on limits at infinity is only used in \cite{Savin, S-new, S-new2, dPSV} to guarantee that $u$ is, in addition, a minimizer of $E_{s,\ep}$ in every bounded set 
of~$\R^n$.
To address the conjecture without the additional assumption \eqref{limits}, it is natural to introduce the two limit functions $u^{\pm} := \lim_{x_{n} \to \pm\infty} u$. These functions depend only on the first $n-1$ variables $x_1, \dots, x_{n-1}$ and are stable solutions of \eqref{fract} in $\R^{n-1}$.
It is not difficult to prove (see Proposition \ref{implication}), 
that if $u^\pm$ are 1D then $u$ is a minimizer. As a consequence, the following implication holds for all $s\in (0,2]$:
$$
\left.
\begin{array}{c}
\text{Stable sol'ns of \eqref{fract} in $\R^{n-1}$ are 1D}
\\
\text{and}\\
\text{Minimizers of \eqref{fract} in $\R^n$ are 1D}
\end{array}
\right\}
 \Rightarrow \text{Monotone sol'ns of \eqref{fract} in $\R^n$ are 1D.}
$$

The difficult problem of classifying stable solutions to \eqref{fract}  is connected to the following well-known conjecture for minimal surfaces:
\begin{conjecture}\label{conjstableminimal}{\it 
Stable embedded minimal hypersurfaces in  $\R^n$ are hyperplanes as long as $n\le 7$.}
\end{conjecture}
A positive answer to this conjecture is known for $n=3$,  by a result of Fischer-Colbrie and Schoen \cite{FishS} and of Do Carmo and Peng \cite{dCP} from the late seventies, and for $n=4$ by a very recent paper of Chodosh and Li~\cite{CL}. It remains open for $5\le n\le 7$.
Instead, the analogue of Conjecture \ref{conjstableminimal} for area-minimizing  (a stronger notion than stability)  hypersurfaces is completely understood in every dimension: it holds, indeed, if and only if $n\leq 7$. 
This shows that extending a result for minimizers to stable solutions may be a very difficult problem. 
Still, Conjecture~\ref{conjstableminimal} suggests the so-called ``stability conjecture'': 
\begin{conjecture}\label{DGconj3}{\it 
Let $u\in C^2(\R^{n})$, $|u|<1$, be a stable critical point of  $E_{s, \ep}$.
Then,  if $n\le 7$, $u$ must be $1D$ provided $s \in (s_* ,2]$, where $s_*\in [0,1)$ is a dimensional constant.}
\end{conjecture}

{
As explained before, once Conjecture \ref{DGconj3}  is known to hold for some $(n-1,s)$  then  Conjecture \ref{DGconj} / \ref{DGconj2} / \ref{DGconj2b}   for $(n,s)$ can be reduced to the question of proving 1D symmetry of  minimizers, which is well-understood thanks to the results in \cite{Savin, S-new, S-new2, dPSV}.
}

With the exception of the case $n=3$ and $s=1$ considered in \cite{FS}, before our work
Conjecture~\ref{DGconj3} was open in every dimension $n\ge 3$,  with the case of the Laplacian ($s=2$) being a long standing open problem. The case $n=2$ is simpler (thanks to a certain ``parabolic manifold'' type property, it can be proved through some Liouville theorems initiated in \cite{BCN, AmbrC}) and has been established for the whole range $s\in (0,2]$: for $s=2$ in \cite{GG, AmbrC,AAC},  for $s=1$ in \cite{C-SM}, and  for $0<s<2$ in \cite{CSi,Si-V}.

In this paper we establish Conjecture \ref{DGconj3} when $n=3$ and $s\in (s_*,1)$, for some $s_*<1$. From this and a result for minimizers from~\cite{dPSV}, we deduce Conjecture \ref{DGconj2b} when $n=4$ and  $s\in (s_*,1)$.

\subsection{Main result: a new classification theorem}\label{sub-1.5}
We study stable solutions of the fractional Allen-Cahn equation
\begin{equation}\label{theequation}
 (-\Delta)^{s/2}u + W'(u)=0,  \quad |u|<1\quad  \mbox{in }\R^n, 
\end{equation}
with $s\in(0,1)$, where 
\begin{equation} \label{quarticpot}
W(u):= \frac 1 4 (1-u^2)^2
\end{equation}
is the standard quartic double-well potential with wells at $\pm 1$. Note that this is equation \eqref{fract} with $\ep=1$ (we can always assume this value of the parameter after scaling).
For the precise definition of {\em stable solution} to \eqref{theequation}, see Subsection \ref{subsectionstable} below.

The main goal of the paper is to establish the following classification result.
For brevity, we use in its statement the following terminology. Recall that, as mentioned before, stable solutions to the fractional Allen-Cahn equation in all of $\R^2$ have been already classified.
\begin{defi}\label{defiwhatcones}
For $n\ge 3$, we say that {\em hyperplanes are the only stable $s$-minimal cones in~$\R^n\setminus \{0\}$} when the following holds:
if $ 3\le  m \le n$ and  $\Sigma\subset \R^m$ is a stable 
$s$-minimal  cone in $\R^m\setminus \{0\}$, whose boundary  $\partial \Sigma$ is nonempty and smooth away from~$0$,  then necessarily $\partial \Sigma$   is a hyperplane (up to sets of measure zero).

The notion of (smooth away  from $0$) stable $s$-minimal cone in $\R^m\setminus \{0\}$ is exactly that of \cite[Definition 1.1]{CCS}, which we recall in Definition \ref{defstablecone} below.\footnote{In this paper and in \cite{CCS} the perturbations to define stability do not modify the cone in a neighborhood of the origin (even if in \cite{CCS} we used the terminology ``stable $s$-minimal cones in~$\R^m$'').}
\end{defi}

{
\begin{rem}
Let us comment on Definition \ref{defiwhatcones}, which will be an assumption of our main result, Theorem \ref{thmclas} below. 
\begin{itemize}
\item First, we will not be assuming  that stable  $s$-minimal cones in $\R^2\setminus \{0\}$ must be flat: the integer $m$ in the definition is at least 3. 
This is important because (similarly as it happens in the classical case $s=1$)  certain $s$-minimal symmetric cones like the ``cross'' $\{x_1 x_2>0\}$ are stable in $\R^2 \setminus \{0\}$,  at least if $s$ is sufficiently close to $1$.
\item Second,  the reason why, given a dimension $n$ we need an hypothesis  in lower dimensions  $3\le m\le n$, is  the use of a  dimension reduction argument of Federer type in the proof of Theorem \ref{thmclas}.
This is also why Definition \ref{defiwhatcones} concerns only  cones with smooth trace on the sphere (the dimension reduction argument will always allow us to suppose this). 
Now, in this dimension reduction,  one needs to consider the case $m=2$ (and not only $m\ge3$). However, for the case $m=2$ we do not need to assume any flatness hypothesis  since we will actually prove in Lemma \ref{hwioheoithe} the following property, which holds true for all $s\in(0,1)$:  any cone in $\R^n$ of the form $\tilde\Sigma \times \R^{n-2}$ and such that  $\chi_\Sigma- \chi_{\Sigma^c}$ can  be obtained as limit of some sequence of stable critical points of $E_{s,\ep}$ (with vanishing parameter $\ep$) must be a half-space.
\end{itemize}
\end{rem}
}

The following is our main result.

\begin{thm}\label{thmclas}
Assume that, for some pair $(n,s)$ with $n\ge 3$ and $s\in (0,1)$, hyperplanes are the only stable $s$-minimal cones in $\R^n\setminus \{0\}$.

Then, every stable solution of \eqref{theequation}-\eqref{quarticpot} is a 1D layer solution, namely, $u(x)= \phi(e\cdot x)$ for some direction $e\in S^{n-1}$ and increasing function $\phi: \R \rightarrow (-1,1)$. 
\end{thm}

\begin{rem}\label{otherW}
As it will be clear from the proofs, throughtout the paper we do not need $W$ to be given exactly by  $W(u)= \frac 1 4  (1-u^2)^2$, nor to be even. It can be replaced everywhere by any function $W\in C^3([-1,1])$   satisfying $W(\pm1)=0$, $W>0$ in $(-1,1)$,  $W\le  W(t_0)$ {in $(-1,1)$} for some $t_0\in(-1,1)$,  $\{t\in [-1,1]\,:\,W'(t)=0\} =\{-1,t_0,1\}$,  $W''(\pm1) >0$, and $W''(t_0)<0$. For instance, $W$ can be replaced everywhere by the Peierls-Nabarro potential $1+\cos (\pi u)$. However, for simplicity, we will write the results for the potential \eqref{quarticpot}. 

Note also that the results of Subsection~\ref{sub-2.1} apply to any $C^3$ potential $W$.
\end{rem}

Thanks to the theorem,
to establish Conjecture \ref{DGconj3} in $\R^n$, $n\geq 3$, it suffices to prove that hyperplanes are the only stable $s$-minimal cones in $\R^n\setminus \{0\}$ for $s\in (s_*,1)$. Recently, in~\cite{CCS} we have
proved that  planes are the only stable $s$-minimal cones in $\R^3\setminus \{0\}$ for $s\in (s_*,1)$, for some $s_*<1$. As a consequence of this last result, we obtain the two following corollaries. 
First, we can prove the stability conjecture in $\R^3$ for $s\in (0,1)$ sufficiently close to $1$.

\begin{cor}\label{corclas2}
 Conjecture \ref{DGconj3} holds in $\R^3$ 
---that is, every stable critical point of  $E_{s, \ep}$ in $\R^3$ is $1D$ for $s\in (s_*,1)$ with $s_*<1$ sufficiently close to 1. 
\end{cor}

Using Corollary \ref{corclas2}, we establish the De Giorgi type conjecture in $\R^4$  for $s\in (0,1)$ sufficiently close to $1$.

\begin{cor}\label{corclas3}
Conjecture \ref{DGconj2b} holds in $\R^4$
---that is, every monotone solution of \eqref{theequation} in $\R^4$ is $1D$ for $s\in (s_*,1)$ with $s_*<1$ sufficiently close to 1.
\end{cor}

We recall that in \cite{FS} Figalli and Serra proved Conjecture \ref{DGconj2b} in $\R^4$ for $s=1$.
\subsection{Definitions: localized energy,  minimizers, and stable solutions}\label{subsectionstable}
Thro-\-ughout the paper we consider solutions $u: \R^n\to (-1,1)$ to integro-differential equations 
\begin{equation}\label{equK}
L_K u+ W'(u )=0 \quad \mbox{in } \Omega,
\end{equation}
where  $W\in C^3(\R)$,  $\Omega\subset \R^n$ is an open set, and
\begin{equation}\label{LK}
L_Ku(x):= \int_{\R^n} \big(u(x)-u(\bar x)\big)K(x-\bar x)\,d\bar x.
\end{equation}

Although in some parts of the paper we will take $L_K = (-\Delta)^{s/2}$, i.e., $K(z) = (2-s) |z|^{-n-s}$,
some other parts will hold and will be stated for more general kernels~$K$. Throughout the article,
$K$ belongs to the ellipticity class $\mathcal L_2$ of Caffarelli and Silvestre~\cite{CS3}. That is, we assume $s\in (0,2)$,
\begin{equation}\label{L0}
 K(z) =K(-z),
\quad
\frac{(2-s)\lambda}{|z|^{n+s}} \le  K(z) \le \frac{(2-s)\Lambda}{|z|^{n+s}},
\end{equation}
and that $K$ is of class $C^2$ away from the origin with derivatives satisfying
\begin{equation}\label{L2}
  \max\bigl\{|z|\, |\partial_{e}  K(z)| \,,\, |z|^2 |\partial_{ee} K(z) |\bigr\} \le \frac{(2-s)\Lambda}{|z|^{n+s}},
\end{equation}
for all $z\in\R^n\setminus \{0\}$ and $e\in S^{n-1}$. Here $\lambda$ and $\Lambda$ are given positive constants.

By a solution of \eqref{equK}, we mean a bounded function in all of $\R^n$ which is, in $\Omega$, a $C^2$ strong solution. 
Note  that, since $K$ is in the class $\mathcal L_2$, $W\in C^3(\R)$, and $L_K$ is translation invariant, any measurable function $u:\R^n \to (-1,1)$ which solves \eqref{equK} in the sense of distributions, i.e.,  $\int uL_K \xi+ W'(u )\xi=0$ for all $\xi \in C^\infty_c(\Omega)$, belongs to $C^2(\Omega)$ and hence it is a strong solution. Indeed, this follows from the existing regularity theory for nonlocal elliptic equations, as explained in Appendix~\ref{app-C}.

We now recall the standard definitions of localized energy, minimizers, and stable solution in this more general framework.

We first point out that \eqref{equK} is the Euler-Lagrange equation of the localized energy functional
\begin{equation}\label{energyW}
 \mathcal E_{\Omega} (v) :=  \mathcal E_{\Omega}^{\rm Sob} (v)   + \mathcal E_{\Omega}^{\rm Pot} (v) ,
 \end{equation}
 where 
 \[
 \mathcal E_{\Omega}^{\rm Sob} (v) := \frac{1}{4}\iint_{(\R^n\times\R^n)\setminus (\Omega^c \times \Omega^c)} 
  |v(x)-v(\bar x)|^2 K(x-\bar x) \,dx\,d\bar x,   \quad \mathcal E_{\Omega}^{\rm Pot} (v) :=\int_{\Omega} W(v)\,dx.
  \]
Here and throughout the paper, $\Omega^c$ denotes the complement of $\Omega$ in $\R^n$. 
Note that $\mathcal E_{\Omega}(v)$ is finite for every function $v$ which is bounded in $\R^n$ and Lipschitz in a neighborhood of $\overline\Omega$, since $s\in (0,2)$.

We say that $u$ is a {\em minimizer} of  $\mathcal E_\Omega$ (or a minimizer of \eqref{equK} in $\Omega$)  if $\mathcal E_\Omega(u)<\infty$ and  $\mathcal E_\Omega(u)\leq \mathcal E_\Omega(v)$ for all $v$ satisfying $v=u$ outside of $\Omega$.

If $u$ is a minimizer of $\mathcal E_\Omega$, then
\begin{equation}\label{secondvar}
\frac{d^2}{dt^2} \Big|_{t=0}  \mathcal E_\Omega(u + t \xi)  \ge 0 \quad \mbox{for all Lipschitz function } \xi \text{ with } \xi =0 \text{ in } \Omega^c.
\end{equation}
This motivates the definition of stable solution.
Within this framework, $u$ is a stable solution if and only if $u$ is a critical point of $\mathcal E_\Omega$ for which
the second variation of $\mathcal E_\Omega$ at $u$ ---with respect to compactly supported perturbations vanishing outside $\Omega$--- is nonnegative. In other words:

\begin{defi}[\textbf{Stability}]\label{stability}
Given an open set $\Omega\subset\R^n$, we say that a solution $u$ of \eqref{equK} is {\em stable} if \eqref{secondvar} holds, or equivalently,
\begin{equation}\label{stable}
\frac{1}{2}\iint_{\R^n\times\R^n} |\xi(x)-\xi(\bar x)|^2 K(x-\bar x)\,dx\,d\bar x + \int_{\R^n} W''(u)\xi^2\,dx\geq 0 
\end{equation}
for all Lipschitz functions $\xi$ in $\R^n$ with compact support in $\R^n$ and such that $\xi=0$ in $\Omega^c$.
As a particular case, we say that $u$ is a stable solution of  $\eqref{theequation}$ if \eqref{stable} holds   with $K(z)=(2-s) |z|^{-n-s}$ for all Lipschitz functions $\xi$  with compact support in~$\R^n$.
\end{defi}

We conclude this section giving the notion of fractional $s$-perimeter, as introduced by Caffarelli, Roquejoffre, and Savin in \cite{CRS}. Let $s\in (0,1)$. Given a set $E$ in $\R^n$ and a bounded open set $\Omega\subset \R^n$, we define the fractional $s$-perimeter of $E$ in $\Omega$ as
\begin{equation}\label{def-per}
\hspace{-.181cm} P_{s}(E,\Omega):=\iint_{(E\cap \Omega) \times (\R^n\setminus E)}\frac{1}{|x-\bar x|^{n+s}}\,dx \,d\bar x + \iint_{(\Omega\setminus E)\times(E \setminus \Omega)}\frac{1}{|x-\bar x|^{n+s}}\,dx \,d\bar x.
\end{equation}

Observe that, in our notation, $P_s(E,\Omega)=2\mathcal E^{\rm Sob}_\Omega(\chi_E)$,  where $\chi_E$ denotes the characteristic function of the set $E$ and we choose here $K(z)=|z|^{-n-s}$.

\section{Further new results and organization of the paper} \label{sec-2}
Theorem \ref{thmclas} will follow as a combination of several new results ---which in addition are of independent interest--- on stable solutions to integro-differential semilinear equations. They are valid in any dimension $n\ge 2$ and include: 
\begin{itemize}
\item A sharp $BV$ estimate;
\item A bound of $\mathcal E^{\rm Pot}$ by $\mathcal E^{\rm Sob}$,  a cube decomposition, and an energy estimate;
\item A density estimate;
\item The convergence of blow-downs to stable $s$-minimal cones.
 \end{itemize}

These results are presented in detail in the next subsections.
 
\subsection{$BV$ estimate} \label{sub-2.1}

For $s\in (0,1)$, in Section \ref{sec-3} we establish $BV$ and Sobolev estimates for stable solutions to the semilinear equation $L_K(u)+W'(u)=0$ in a ball~$B_{2R}$. Their proof follow the techniques introduced in \cite{CSV} within the context of stable nonlocal $s$-minimal surfaces.  Our estimates are optimal in their dependence on the radius $R$
and universal in the sense that they are independent of the potential~$W$. Note that a universal $BV$ estimate is somehow surprising since it corresponds to a quantity of higher order than the one of the equation, which is $s$. 

By the same method and independently of our work, the $BV$ and Sobolev estimates \eqref{BVest} and \eqref{energy-est-eq} have also been established by Gui and Li~\cite[Proposition~1.7]{GL}. In the current paper we additionally keep track on how the constants in the estimates blow-up as $s \uparrow1$. Such information could be useful in some applications.\footnote{This was the case in our previous paper \cite{CCS}, where we needed the dependence on $s$ (as $s\uparrow 1$) in the $BV$ estimate from \cite{CSV} for stable $s$-minimal sets.}

\begin{thm}\label{BV}
Given $s_0>0$, let $n\ge 2$, $s\in(s_0,1)$, $W$ be any $C^3(\R)$ function, and $K$ satisfy \eqref{L0} and \eqref{L2}. Let $R>0$ and $u:\R^n\to (-1,1)$ be a stable solution of $L_K u+W'(u)=0$ in $B_{2R}\subset\R^n$.

Then,
\begin{equation}\label{BVest}
\int_{B_R}|\nabla u|\,dx  \le \frac{C}{1-s} R^{n-1},\end{equation}
where  $C$ is a constant which depends only on $n$, $\lambda$, $\Lambda$, and $s_0$. 
\end{thm}

{
Two remarks on this result are in order.
\begin{rem}\label{remscaling}We emphasize that, in Theorem \ref{BV}, the constant $C$ is completely independent of the potential $W$, which is an arbitrary $C^3$ function ---in particular it is not necessarily a double-well potential.
As a consequence, the estimate  \eqref{BVest} is scale invariant: the estimate for any $R>0$ follows from the case $R=1$  applied to the rescaled function $\tilde u = u(R\,\cdot\,)$, which solves $L_K \tilde u+\widetilde W'(\tilde u)=0$ in $B_{2}$, for $\widetilde W := R^sW$.
\end{rem}

\begin{rem} In the case of  solutions $u_\ep :\R^n\to \R$ of the fractional  Allen-Cahn equation  \eqref{fract} (recall that $\ep>0$ is a small parameter), Theorem \ref{BV} reads:
\[
u_\ep\mbox{ stable in $B_2$} \quad \Rightarrow \quad \int_{B_1}|\nabla u_\ep|\,dx  \le \frac{C}{1-s}.
\]
We emphasize that the constant in the right-hand side of the previous estimate is independent of $\ep$. We do not expect the blow-up rate  $(1-s)^{-1}$ as $s\uparrow 1$ to be sharp, but for such uniform in $\ep$ estimate to hold, the right-hand side needs to blow up at least at the rate $(1-s)^{-1/2}$ ---we actually expect the power $-{1/2}$ to be the sharp rate. 
Indeed, by the argument sketched next, for all $n\ge 2$ and  $s\in (0,1)$ there exists $\ep = \ep(s)>0$ and a stable solution $u_\ep: \R^n\to \R$ of \eqref{fract}  in $B_2$ satisfying $\int_{B_1}|\nabla u_\ep|\,dx  \ge c(1-s)^{-1/2}$, where $c$ stays bounded away from zero as $s\uparrow 1$. 

To show this, consider the set  \[E_* : = \bigcup_{k\in \mathbb Z} \,\big\{ 2kC_*(1-s)^{1/2} \le x_n \le (2k+1)C_*(1-s)^{1/2}\big\},\] whose boundary consists of parallel hyperplanes arranged in a $C_*(1-s)^{1/2}$-periodic fashion.  Let us prove first that $E_*$ is stable $s$-minimal set in $(-1,1)^n\supset B_1$, provided $C_*$ is chosen large enough.
To verify this, we can use the stability criterium found in \cite{DDPW,FFFMM} involving the nonlocal analogue $c_s^2$ of the squared norm of the second fundamental form of $\partial E_*$. To this end, one first checks that $c_s^2$ behaves as the quantity $(C_*(1-s)^{1/2})^{-1-s}$.
Hence, the stability inequality in $(-1,1)^n$ reads 
\begin{equation}\label{wheiothewohw}
\sum_{j}  C_*^{-1-s} \int_{\Gamma_j}\eta^2 d\mathcal H^{n-1}_x \le  C(1-s)^{\frac{1+s}{2}}\sum_{i,j} \int_{\Gamma_i}  \int_{\Gamma_j} \frac{\big|\eta(x)-\eta(\bar{x})\big|^2}{|x-\bar{x}|^{n+s}}d\mathcal H^{n-1}_xd\mathcal H^{n-1}_{\bar{x}}, 
\end{equation}
for all $\eta\in C^{0,1}_c([-1,1]^n)$ where $\Gamma_j : = \{x_n = j C_*(1-s)^{1/2} \}$.

By the fractional Poincar\' e inequality for the $H^{(1+s)/2}$ seminorm in $[-1,1]^{n-1}$ (applied to the functions $\eta_{|_{\Gamma_j}} $, $\Gamma_j \cong  \R^{n-1}$) we have
\[
\int_{\Gamma_j}\eta^2 d\mathcal H^{n-1}_x \le  C (1-s) \int_{\Gamma_j}  \int_{\Gamma_j} \frac{\big|\eta(x)-\eta(\bar{x})\big|^2}{|x-\bar{x}|^{n+s}}d\mathcal H^{n-1}_xd\mathcal H^{n-1}_{\bar{x}},
\]
for all $\eta\in C^{0,1}_c([-1,1]^n)$, where $C$ is uniformly bounded as $s\uparrow 1$.  Therefore, since ${(1-s)}/{(1-s)^{\frac{1+s}{2}}} \to 1$ as $s \uparrow 1$, \eqref{wheiothewohw} will hold provided that we choose $C_*$ large enough. This makes $E_*$ stable. 

At the same time, it is known \cite{Cetal} that  the fractional Allen-Cahn equation   \eqref{fract}  admits periodic 1D solutions for all large enough periods (relatively to $\ep$). Hence, given $s\in (0,1)$, for every  $0<\ep \le c(1-s)^{1/2}$  there exists a periodic 1D solution $u_{*,\ep}(x_n)$ of \eqref{fract} such that $\{u_{*,\ep}=0\} = \partial E_*$. 
Now, as $\ep\downarrow 0$ we have, by construction,  $u_{*,\ep} \to  \chi_{E_*} - \chi_{E_*^c} $ and hence one can show that, for $\ep$ small enough depending on $s$,   
$u_{*,\ep}$ inherits from $E_*$ the stability in $(-1,1)^n$.  Finally, note that for such construction  $\int_{B_1}|\nabla u_\ep|\,dx$ blows up at the rate $(1-s)^{-1/2}$ (by the coarea formula), as claimed.
\end{rem}
}

As a simple consequence of Theorem \ref{BV} we deduce the following estimate for the Sobolev part of the energy  of stable solutions to \eqref{equK}.
\begin{cor}\label{energy-est}
Given $s_0>0$, let $n\ge 2$, $s\in(s_0,1)$, $W$ be any $C^3(\R)$ function, and $K$ satisfy \eqref{L0} and \eqref{L2}.  Let $R>0$ and $u:\R^n\to (-1,1)$  be a stable solution of $L_K u+W'(u)=0$  in $B_{2R}\subset \R^n$. 

Then,
\begin{equation}\label{energy-est-eq}
\mathcal E^{\rm Sob}_{B_R}(u)\le  \frac{C}{(1-s)^2} R^{n-s},
 \end{equation}
where  $C$ is a constant which depends only on $n$, $\lambda$, $ \Lambda$, and $s_0$. 
\end{cor}

\subsection{Control of $\mathcal E^\text{Pot }$ by $\mathcal E^\text{Sob}$}\label{pot0} 
 In Section  \ref{sec-4}  we need to assume $W$ to be a double-well potential. For simplicity we take $W$ to be the quartic potential \eqref{quarticpot}, although all proofs apply to the more general class of double-well potentials described in Remark~\ref{otherW}.

The following result states that $\mathcal E^{\rm Sob}$ in $B_R$ controls $\mathcal E^{\rm Pot}$ in a slightly smaller ball.
Here we can take $s\in (0,2]$ and thus we include the case of the Laplacian. To our knowledge, this result is new even for the Laplacian.{\footnote{For $s=2$ and in the particular case of stable solutions in all of $\R^n$, Villegas \cite[Proposition~1.5]{Vill} has recently proved that the quotient of Sobolev and potential energies in $B_R$ tends to $1$ as $R\uparrow \infty$.}}
To state our result for $s=2$,
we recall that if $K_s(z) \asymp (2-s)|z|^{-n-s}$  (where  $X\asymp Y$ here means $X\le CY$ and $Y \le CX$) we have $L_{K_s} v \to  a_{ij}\partial_{ij} v$ for all $v\in C^2_c(\R^n)$ as $s\uparrow 2$ (see the computations in \cite[Section 6] {CS-reg}), where $a_{ij}  \asymp {\rm Id}$  (that is $a_{ij}\partial_{ij}$ is some translation invariant second order elliptic operator). 
In the following proposition, $L_{K}$ is meant to be, when $s=2$, this limiting second order elliptic operator.

\begin{prop}\label{DircontrolsPot2}
Given $s_0>0$, let $n\ge 2$, $s\in(s_0,2]$,  $W(u)=\frac 1 4 (1-u^2)^2$, and $K$ satisfy \eqref{L0} and \eqref{L2}. Let $R>0$ and $u:\R^n\rightarrow (-1,1)$ be a stable solution of $L_{K}u + W'(u)=0$ in~$B_{2R}$.  

Then, 
\begin{equation}\label{Sobcontrols} \mathcal E^{\rm Pot}_{B_{R-R_0}}(u) \le C\mathcal E^{\rm Sob}_{B_R}(u)\end{equation}
whenever $R> R_0$, where $C$ and $R_0$ are positive constants which depend only on $n$, $\lambda$, $\Lambda$, and $s_0$.
\end{prop}

To prove Proposition \ref{DircontrolsPot2} we cover the ball $B_{R-R_0}$ by cubes of size comparable to $R_0$ and with controlled overlapping, and we show that the contribution of each of these cubes to the total Sobolev energy  controls  its contribution to the potential energy.  
We believe that this cube decomposition, which holds also for the classical Allen-Cahn equation, may be of independent interest.

In Appendix~\ref{app-B}, we present an alternative proof ---only in the case $s=2$, for brevity---  of a weaker version of the estimate of Proposition \ref{DircontrolsPot2} in which an error term appears on the  right-hand side of \eqref{Sobcontrols}. The error is of order $R^{n-s}$ when $s\in(0,1)$, $R^{n-1}$ for $s\in(1,2]$, and $R^{n-1}\log R$ for $s=1$. Even if the result is weaker, we give this alternative proof since  it is  new even for the Laplacian (to the best of our knowledge), very short, and interesting in itself. It relies on taking a suitable test function on the stability inequality and on elementary estimates.

Combining Proposition  \ref{DircontrolsPot2}  with Corollary \ref{energy-est} we obtain the following bound for the full energy of stable solutions in a ball.
\begin{thm}\label{thm1}
Given $s_0>0$, let  $n\ge 2$, $s\in(s_0,1)$, $W(u)=\frac 1 4 (1-u^2)^2$, and $K$ satisfy \eqref{L0} and \eqref{L2}. Let $R\geq 1$ and $u:\R^n\rightarrow (-1,1)$ be a stable solution of $L_K u+W'(u)=0$ in $B_{2R}\subset\R^n$. 

Then,
\begin{equation}\label{egest}
\mathcal E_{B_R}(u) \le \frac{C}{(1-s)^2}R^{n-s},
\end{equation}
where  $C$ is a constant which depends only on $n$, $\lambda$, $\Lambda$, and $s_0$.
\end{thm} 

For $s\in [1,2]$, \cite{CC1, CC2} established the related energy estimates $\mathcal E_{B_R}(u) \le C R^{n-1}$, for $s\in (1,2]$, and  $\mathcal E_{B_R}(u) \le C R^{n-1}\log R$, for $s=1$, in the case of minimizers in $B_{2R}\subset\R^n$. However, such strong estimates are not expected to hold if one considers stable solutions ---assuming stability only in the double ball $B_{2R}$ as in Theorem \ref{thm1}, not in the whole $\R^n$--- instead of minimizers.

\subsection{Density estimates}  \label{sub-2.3} 
Section  \ref{sec-5}  deals with density estimates for stable solutions of $(-\Delta)^{s/2}u + W'(u)=0$ in all of $\R^n$. 
We need to assume the operator $L_K$ to be the fractional Laplacian $(-\Delta)^{s/2}$ since an important ingredient in their proof is the  monotonicity formula from \cite{CC2}, which uses the extension property of $(-\Delta)^{s/2}$  and thus is not available for other kernels. In addition, our proof uses crucially \eqref{BVest}  and thus needs  $s<1$.

\begin{prop} \label{density}
Let  $n\ge 2$, $s\in (0,1)$, and $W(u)=  \tfrac 1 4 (1-u^2)^2$. Let $u:\R^n\rightarrow (-1,1)$  be a stable solution of $(-\Delta)^{s/2}u + W'(u)=0$ in $\R^n$. 

Then, for every $\bar c\in (0,1)$ there exist positive constants $\omega_0$ and $R_0$, which depend only on $\bar c$, $n$, and $s$, such that the following holds. For every $R\geq R_0$, if
\begin{equation}\label{hp-density}
R^{-n}\int_{B_R} |1+u| \, dx \leq \omega_0\qquad \left(\mbox{respectively,}\quad R^{-n}\int_{B_R} |1-u| \, dx \leq \omega_0\right),
\end{equation}
then
\begin{equation}\label{th-density}
\left\{ u\ge -\bar c\right\}\cap B_{R/2} = \varnothing\qquad  \left(\mbox{respectively,}\quad \left\{ u\le \bar c\right\}\cap B_{R/2} = \varnothing\right).
\end{equation}
\end{prop}

\subsection{Convergence of blow-downs}  \label{sub-2.4} 
The goal of Section  \ref{sec-6}  is to establish   that the blow-downs of entire stable solutions to $(-\Delta)^{s/2}u + W'(u)=0$ in $\R^n$ converge, up to a subsequence, to the characteristic function of a stable nonlocal $s$-minimal cone. This convergence is local, both in the $L^1$-sense and in the sense of the Hausdorff distance between the level sets and the (boundary of the) cone.  {To be precise, following \cite[Definitions  1.1 and 2.2]{CCS}, we recall now two notions of stability for the fractional perimeter $P_s$ defined in $\eqref{def-per}$. 

The first one is the notion of stability for a cone that is smooth away from the origin. It is the notion used in Definition \ref{defiwhatcones} above.
\begin{defi}\label{defstablecone}
Let $\Sigma\subset \R^{n}$ be a cone with nonempty boundary of class $C^2$ away from the origin.
We say that  $\Sigma$  is a stable cone for the $s$-perimeter in $\R^n\setminus \{0\}$, or stable $s$-minimal cone in $\R^n\setminus \{0\}$,  if
\begin{equation}\label{stabilityweak}
\liminf_{t\to 0} \frac{1}{t^2} \big(P_s(\phi^t_X(\Sigma), B_1)-P_s(\Sigma, B_1))\ge  0
\end{equation}
for all vector fields $X\in C^\infty_c(B_1\setminus\{0\}, \R^n)$. Here $\phi^t_X:\R^n\to \R^n$ denotes the integral flow of $X$ at time $t$ (which is a smooth diffeomorphism for $t$ small).
\end{defi}

The following is the notion of {\em weak stability} for a general set $E$ with finite $s$-perimeter, as given in \cite[Definition 2.2]{CCS}.
\begin{defi}\label{def-weakly}
A set $E\subset \R^n$ with $P_s(E,\Omega)<\infty$ is said to be {\em weakly stable}  in $\Omega$ for the $s$-perimeter if for every given vector field  $X=X(x,t)\in C^\infty_c(\Omega\times(-1,1); \R^n)$ we have
\[
\liminf_{t\to 0} \frac{1}{t^2} \big(P_s(\phi^t_X(E), \Omega)-P_s(E, \Omega)\big)\ge  0,
\]
where $\phi^t_X $ denotes the integral flow of $X$ at time $t$ (with $0$ as initial time).
\end{defi}

In contrast with this last notion, Definition \ref{defstablecone} assumes the vector field $X=X(x)$ to be autonomous. However, since a cone $\Sigma$ in Definition \ref{defiwhatcones} is smooth outside of the origin, it is simple to see that $\Sigma$ is a stable $s$-minimal cone in $\R^n\setminus\{0\}$ (in the sense of Definition \ref{defstablecone})  if and only if $\Sigma$ is a weakly stable set in $\R^n\setminus\{0\}$ for the fractional perimeter $P_s$ (in the sense of Definition \ref{def-weakly}) .}

In our following result, we will prove the limiting cone $\Sigma$ to be weakly stable not only in $\R^n\setminus\{0\}$, but in $\R^n$. In its statement we use the notation 
\[
\overline D(E;x) = \limsup_{r\to 0} \frac{|E\cap B_r{(x)}|}{|B_r|}  \qquad \mbox{and}\qquad \underline D(E;x) = \liminf_{r\to 0} \frac{|E\cap B_r{(x)}|}{|B_r|} 
\]
for the upper and lower densities of a set $E\subset\R^n$ at a point $x\in \R^n$.

\begin{thm}\label{thm2}
Let  $n\ge 2$, $s\in (0,1)$, and $W(u)=\frac 1 4 (1-u^2)^2$. Let $u:\R^n\rightarrow (-1,1)$  be a stable solution of $(-\Delta)^{s/2}u + W'(u)=0$ in $\R^n$. 

Then, for every given  blow-down sequence $u_{R_j}(x)= u(R_jx)$ with $R_j\uparrow \infty$, there is a subsequence $R_{j_k}$ such that
\[ u_{R_{j_k}} \rightarrow \chi_{\Sigma}-\chi_{\Sigma^c} \quad \mbox{ in }L^1(B_1)\]
for some cone $\Sigma$ which is a weakly stable set in $\R^n$ for the fractional perimeter~$P_s$ and which is nontrivial (not equal to $\R^n$ or $\varnothing$ up to sets of measure zero).

In addition,  up to changing $\Sigma$ in a set of measure zero,  we have
\begin{equation}\label{wngioewhtioeh1}
 x\in \partial \Sigma \qquad \Leftrightarrow\qquad  0<\underline D(x; \Sigma) \le \overline D(x; \Sigma)< 1
\end{equation}
and, for all given $c\in(-1,1)$ and $\rho\ge 1$, we have
\begin{equation}\label{wngioewhtioeh2}
  d_{\rm Hausdorff}\big(\, \{u_{R_{j_k}} \ge   c\}\cap B_\rho\,,\,  \Sigma \cap B_\rho\,\big) \rightarrow 0  \qquad \text{as } k\uparrow\infty,
\end{equation}
where $d_{\rm Hausdorff}(X,Y) = \inf\{ d>0 \,: \, X\subset Y+B_d \ \mbox{and} \ Y\subset X+B_d\}$  denotes the standard Hausdorff distance between subsets of $\R^n$.
\end{thm}

The proof of Theorem \ref{thm2} puts together all our results stated above in this section.
More precisely, the first part of Theorem \ref{thm2}  ($L^1$-convergence) is the content of Proposition~\ref{conv-L^1}, the proof of which uses the  $BV$ and energy estimates from Theorems \ref{BV} and \ref{thm1}, as well as the monotonicity formula from \cite{CC2}. 
The second part of the statement (uniform convergence) follows then from the density estimates of Proposition \ref{density}.

\subsection{Proofs of the main result and its corollaries}  \label{sub-2.5} 
In Section \ref{sec-7} we prove Theorem \ref{thmclas} by combining  Theorem \ref{thm2} and the ``improvement of flatness'' results established by Dipierro, Serra, and Valdinoci in \cite{dPSV}. 
There are some nontrivial technical details involved in our proof, like a dimension reduction argument that allows us to assume that the stable cones obtained after blow-down are smooth away from $0$ (as required in Definition \ref{defiwhatcones}). It is important to deal  with this  smoothness issue to guarantee the applicability of Theorem \ref{thmclas} in concrete cases. For instance, from \cite{CCS}, in $\R^3$ we only know how to classify stable cones that are smooth away from $0$.

Finally, also in Section \ref{sec-7}  we will give the straightforward proofs of Corollaries~\ref{corclas2} and \ref{corclas3}.

\subsection{Smooth stable $s$-minimal surfaces in $\R^3$ are planes when $s\sim 1$}\label{sub-2.6} 
In Appendix \ref{app-A}, the arguments of Section \ref{sec-7}  will be easily modified to establish the flatness of $C^2$ stable $s$-minimal surfaces in all of $\R^n$ whenever all stable $s$-minimal cones are known to be hyperplanes up to dimension $n$. As a consequence we deduce that every $C^2$ stable $s$-minimal surface in $\R^3$ must be a plane if $s\in (s_*,1)$ for some $s_* < 1$, since our main result  in \cite{CCS} states that planes are the only stable $s$-minimal cones in $\R^3\setminus\{0\}$ (smooth away from $0$) for all $s \in(s_*, 1)$. This was already announced without proof in \cite[Corollary 1.3]{CCS}.

More precisely, we have the following result:

\begin{thm}\label{thmRn}
Assume that, for some pair $(n,s)$ with $n\ge 3$ and $s\in (0,1)$, hyperplanes are the only stable $s$-minimal cones in $\R^n\setminus\{0\}$. 

Let $E$ be an open  subset of $\R^n$, with $\partial E$ nonempty and of class $C^2$, and such that $E$ is a weakly stable set for the $s$-perimeter in $\R^n$ (as defined in {Definition~\ref{def-weakly}}). 

Then, $E$ is a half-space.
\end{thm}

\begin{cor}\label{corR3}
Let $E\subset \R^3$ be a $C^2$ open  subset, with $\partial E\not = \varnothing$, which is weakly stable for the $s$-perimeter in $\R^3$  (as defined in {Definition~\ref{def-weakly}}). 

If $s$ is sufficiently close to 1, then $E$ is a half-space.
\end{cor}

Theorem \ref{thmRn} will follow from a blow-down procedure. Although the general strategy is similar to the classical one for minimizers of the perimeter, to deal with stable solutions one finds several technical issues that are completely analogous to those in Section \ref{sec-7}  for the semilinear equation. This close analogy was our reason to postpone the proof of Corollary 1.3 in \cite{CCS} (i.e., Corollary \ref{corR3}  here) to Appendix~\ref{app-A} of the current paper.

\section{$BV$ and Sobolev energy estimates}\label{sec-3} 

Throughout the paper, $L_K$ denotes the operator \eqref{LK}, where $K$ satisfies \eqref{L0} and \eqref{L2}.

This section extends some results and techniques introduced in \cite{CSV}, within the context of stable nonlocal minimal surfaces, to stable solutions of semilinear nonlocal problems.

Following \cite{SV-mon,CSV}, we consider translations of the solution $u$ in some direction $\boldsymbol {v}\in S^{n-1}$ and we compare the energies of $u$ and of $u(\cdot+t\boldsymbol v)$.
Since we need to consider perturbations vanishing outside the domain, following the notation of \cite{CSV} we introduce a Lipschitz radial cut-off function $\varphi_4$ such that $\varphi_4\equiv 1$ in $B_{2}$,  $\varphi_4\equiv 0$ outside~$B_4$, 
and $|\nabla \varphi_4| \leq 1/2$. For this we take $\varphi_4$ to be linear, as a function of the radius $|x|$, in $B_4\setminus B_2$.

For $\boldsymbol v\in S^{n-1}$ and $t\in(-1,1)$ we consider the map
\begin{equation}\label{perturb}
\Psi_{t}(y):= y + t\varphi_4 (y)\boldsymbol v \qquad\text{for $y\in\R^n$.}
\end{equation}
Clearly, $\Psi_t$ is a Lipschitz diffeomorphism from $\R^n$ onto itself. In particular, both $\Psi_t$ and $\Psi_t^{-1}$ are Lipschitz and coincide with the identity outside $B_4$. Given a function $u$ defined in all of~$\R^n$, we introduce
\[ u_{t}(x) := u\bigl(\Psi_{t}^{-1}(x)\bigr).\]
Even if we do not denote its dependence on $\boldsymbol v$, the function $u_t$ depends on both $\boldsymbol v$ and $t$. We finally set 
$$
M_t(x):=\max\{u(x),u_t(x)\} \quad \text{ and } \quad m_t(x):=\min\{u(x),u_t(x)\}. 
$$

Notice that $u_t$, $M_t$, and $m_t$ all coincide with $u$ outside $B_4$.

We now prove now the analogue of Lemma 2.1 in \cite{CSV}  for the energy \eqref{energyW} associated to \eqref{equK}.
\begin{lem}\label{lem2A}
Let $n\ge 2$, $s\in (0,1)$, let $W$ be any continuous  function in $\R$, and $K$ satisfy \eqref{L0} and \eqref{L2}. Let $u\in C^1(B_6) \cap L^\infty(\R^n)$ be a given function.

Then, for all $t\in(-1,1)$ we have that
\begin{equation}\label{eqlemAstables}
\mathcal E_{B_4}(u_{t}) + \mathcal E_{B_4}(u_{-t}) - 2\mathcal E_{B_4}(u) \le C  t^2 \mathcal E^{\rm Sob}_{B_4}(u),
\end{equation}
where $C$ is a constant which depends only on $n$, $\lambda$, and $\Lambda$.
\end{lem}

\begin{proof}
We set $A_4:=(\R^n\times \R^n) \setminus (B_4^c\times  B_4^c)$. We have
\[
\mathcal E_{B_4}(u_{\pm t}) = \frac{1}{4}\iint_{A_4} |u(\Psi_{\pm t}^{-1}(x))-u(\Psi_{\pm t}^{-1}(\bar x))|^2K(x-\bar x)\,dx\,d\bar x + \int_{B_4} W(u(\Psi_{\pm t}^{-1}(x)))\,dx.
\]

Changing variables $y= \Psi^{-1}_{\pm t}(x)$, $\bar y= \Psi^{-1}_{\pm t}(\bar x)$ in the integrals above (recall that $\Psi^{-1}_{\pm t}$ sends $B_4$ and $B_4^c$ onto themselves), we obtain the following expressions for the Sobolev and potential energies:
\[
\mathcal E^{\rm Sob}_{B_4}(u_{\pm t}) =  \frac{1}{4}\iint_{A_4} |u(y)-u(\bar y)|^2 K\bigl(\Psi_{\pm t}(y)-\Psi_{\pm t}(\bar y)) \,J_{\pm t}(y) \,J_{\pm t}(\bar y)\,dy \,d\bar y\,
\]
and
\[\mathcal E^{\rm Pot}_{B_4}(u_{\pm t})=\int_{B_4} W(u(y))\,J_{\pm t}(y)\,dy,\]
where $J_{\pm t}$  are the Jacobians. 
They are easily seen to be 
\[ J_{\pm t}(y)= 1 \pm t\partial_{\boldsymbol v}\varphi_4 (y) ,\]
which are positive quantities since $|t|<1$ and $|\partial_{\boldsymbol v}\varphi_4|\leq 1/2$.

Clearly,
\begin{equation}\label{pot-pm}\mathcal E^{\rm Pot}_{B_4}(u_{ t})+\mathcal E^{\rm Pot}_{B_4}(u_{- t})= \int_{B_4} W(u(y))(J_{t}(y)+J_{-t}(y))\,dy= 2\mathcal E^{\rm Pot}_{B_4}(u).
\end{equation}

To estimate the sum of the two Sobolev energies $\mathcal E^{\rm Sob}_{B_4}(u_{ t})+\mathcal E^{\rm Sob}_{B_4}(u_{- t})$, we first observe that, since $K\in \mathcal L_2$, then it satisfies the following estimates for its first and second derivatives:
$$\max\left\{|z||\partial_eK(z)|, |z|^2\sup _{|y-z|\leq|z|/2} |\partial_{ee}K(y)|\right\}\leq 2^{n+s+2}\frac{\Lambda}{\lambda} K(z)\leq 2^{n+3}\frac{\Lambda}{\lambda} K(z)$$
for every $z\in \R^n\setminus \{0\}$ and $e\in S^{n-1}$.
Using this fact and performing the same computations as in the proof of Lemma 2.1 in \cite{CSV} (we are taking there, with the notations of  \cite{CSV}, $R=4$ and $K^*=2^{n+3}(\Lambda/\lambda)K$), we deduce that
\begin{equation*}
\mathcal E^{\rm Sob}_{B_4}(u_t)+ \mathcal E^{\rm Sob}_{B_4}(u_{-t}) \leq 2\mathcal E^{\rm Sob}_{B_R}(u) + C t^2 \iint_{A_4} |u(y)-u(\bar y)|^2 K(y-\bar y)\,dy\,d\bar y,
\end{equation*}
where $C$ depends only on $n$, $\lambda$, and $\Lambda$.
This, combined with \eqref{pot-pm}, concludes the proof of the lemma.
\end{proof}

\begin{lem}\label{inner-stable}
Let $n\ge 2$, $s\in (0,1)$, let  $W$ be any $C^3([-1,1])$  function, and $K$ satisfy \eqref{L0} and \eqref{L2}. Let $u: \R^n\to (-1,1)$ be a stable solution of $L_Ku+W'(u)=0$ in $B_6\subset\R^n$.

Then, given
$\nu>0$, there exists $t_0\in (0,1)$  $($possibly depending on $\nu$, $W$, and $u$$)$,  such that
\begin{equation}\label{inner-stability}
\mathcal E_{B_4}(M_t) +  \mathcal E_{B_4}(m_t) -2\mathcal E_{B_4}(u)\geq -\nu t^2\qquad \mbox{for all\;\;} t\in(-t_0,t_0).
\end{equation}
\end{lem}

\begin{proof}
As shown in Appendix~\ref{app-C}, we know that $u\in C^{2}(B_6)$. Thus, the fractional Sobolev semi-norm $\mathcal E_{B_4}$ of $u$ is finite. 

We claim that, given any Lipschitz function $\xi$ vanishing outside $B_4$  and $t\in (-1,1)$ such that $|u+t\xi|\le 1$, we have
\begin{equation}\label{2variation}
\mathcal E_{B_4}(u+t\xi)-\mathcal E_{B_4}(u)\geq - C \|\xi\|_{L^\infty(B_4)}^3|t|^3
\end{equation}
for some constant $C$ depending only on $W$.
Indeed, we have \begin{equation*}
\begin{split}
&\mathcal E_{B_4}(u+t\xi)-\mathcal E_{B_4}(u)=\\
&\hspace{1em}= \frac{t}{2}\iint_{\R^n\times \R^n} (u(x)-u(\bar x))(\xi(x)-\xi(\bar x))K(x-\bar x)\,dx\,d\bar x \\
&\hspace{3em}+\frac{t^2}{4}\iint_{\R^n\times \R^n}|\xi(x)-\xi(\bar x)|^2K(x-\bar x)\,dx\,d\bar x +\int_{B_4} \big(W(u+t\xi)-W(u)\big)\,dx\\
&\hspace{1em}= \frac{t}{2}\iint_{\R^n\times \R^n} (u(x)-u(\bar x))(\xi(x)-\xi(\bar x))K(x-\bar x)\,dx\,d\bar x +\int_{B_4} W'(u)t\xi \, dx\\
&\hspace{3em}+\frac{ t^2}{4} \iint_{\R^n\times \R^n} |\xi(x)-\xi(\bar x)|^2K(x-\bar x)\,dx\,d\bar x+\int_{B_4} W''(u)\frac{(t\xi)^2}{2}\,dx  \\
&\hspace{3em} +\int_{B_4}  W'''(u^*(t,x)) \frac{(t\xi )^3}{6}\,dx\\
\end{split}
\end{equation*}
for some $u^*(t,x)$  in the interval with endpoints  $u(x)$ and $u(x)+ t\xi(x)$ (and thus satisfying $|u^*|\leq 1$). Hence,  using that $u$ is a stable critical point and $W\in C^3$,
we deduce \eqref{2variation}.

We now choose $\xi:=\frac{(u_t-u)_+}{t}$ and $\tilde \xi:=-\,\frac{(u_t-u)_-}{t}$ in the above computations, where $(\cdot)_+$ and $(\cdot)_-$ denote the positive and negative parts. Observe that $M_t=u+t\xi$ and $m_t=u+t\tilde\xi$. Also, since $u$ and $u_t$ agree outside $B_4$ and $u$ is $C^1$ in $\overline B_4$, the $L^\infty$-norms of $\xi$ and $\tilde\xi$ are bounded by a constant independent of $t$ (in fact, depending only on $u$, given the cut-off function $\varphi_4$ defining the diffeormorphism $\Psi_t$). Therefore, from \eqref{2variation} applied to these choices $\xi$ and $\tilde\xi$, adding both inequalities we conclude \eqref{inner-stability}.
\end{proof}

With this consequence of stability in hands, we can now state the following result, which is the analogue of Lemma 2.4 in \cite{CSV}.
We recall that we have set 
$$A_4=(\R^n\times \R^n) \setminus (B_4^c\times B_4^c).$$

\begin{lem}\label{lemEFtdelta}
Let $n\ge 2$, $s\in (0,1)$, let  $W$ be any $C^3([-1,1])$  function, and $K$ satisfy \eqref{L0} and \eqref{L2}.   Let $u: \R^n\to (-1,1)$ be a stable solution of $L_Ku+W'(u)=0$ in $B_6\subset\R^n$.

Then, for every $\nu>0$ there exists $t_0>0$  $($possibly depending on $\nu$, $W$, and $u$$)$ such that 
\begin{eqnarray*}
&&\min\left\{ \iint_{A_4}(u(x)-u_t(x))_+(u(\bar x)-u_t(\bar x))_-K(x-\bar x)\,dx \,d\bar x,\right.\\
&&\hspace{2.2em}\left.\iint_{A_4}(u(x)-u_{-t}(x))_+(u(\bar x)-u_{-t}(\bar x))_-K(x-\bar x)\,dx \,d\bar x \right\}\leq \left(\eta+2\nu\right)t^2
\end{eqnarray*}
holds for $t\in(-t_0,t_0)$, where $\eta= C\mathcal E^{\rm Sob}_{B_4}(u)$ and $C$ depends only on $n$, $\lambda$, and $\Lambda$. Here, $(\cdot)_+$ and $(\cdot)_-$ denote the positive and negative parts, respectively, of a function.

\end{lem}

\begin{proof}
We first observe that, since $u$ is stable and is $C^2(B_6)$ (by Appendix~\ref{app-C}), it satisfies both estimates \eqref{eqlemAstables}  and \eqref{inner-stability}.
Recall also that $M_t$, $m_t$ (as defined in the beginning of the section), $u_t$,  and $u$ all coincide outside $B_4$. Moreover, it holds
\begin{equation}\label{pot}
\mathcal E^{\rm Pot}_{B_4}(M_t)+\mathcal E^{\rm Pot}_{B_4}(m_t)=\mathcal E^{\rm Pot}_{B_4}(u)+\mathcal E^{\rm Pot}_{B_4}(u_t).
\end{equation}

We consider now the Sobolev energy. We claim that
\begin{equation}\label{claimA-ii}
\begin{split}
&|M_t(x)-M_t(\bar x)|^2+|m_t(x)-m_t(\bar x)|^2 - |u(x)-u(\bar x)|^2 - |u_t(x)-u_t(\bar x)|^2\\
&\hspace{2em}=-2(u(x)-u_t(x))_+(u(\bar x)-u_t(\bar x))_-.
\end{split}
\end{equation}
Indeed, we first observe that, if $(u(x)-u_t(x))
(u(\bar x)-u_t(\bar x))\geq 0$, then the right-hand side of \eqref{claimA-ii} vanishes and that the equality is clear.

Assume now that $(u(x)-u_{t}(x))(u(\bar x)-u_{t}(\bar x))< 0$. 
Then, by symmetry between $x$ and $\bar x$, we may assume $u(x)>u_{ t}(x)$, $u(\bar x)<u_{t}(\bar x)$. Now, a simple computation shows that
\begin{equation*}
\begin{split}
&|M_t(x)-M_t(\bar x)|^2+|m_t(x)-m_t(\bar x)|^2 - |u(x)-u(\bar x)|^2 - |u_t(x)-u_t(\bar x)|^2\\
&\hspace{1em} = -2 u(x)u_t(\bar x) - 2 u_t(x)u(\bar x) + 2u(x)u(\bar x) +2 u_t(x)u_t(\bar x)\\
&\hspace{1em} = -2(u(x)-u_t(x))(u_t(\bar x)-u(\bar x)).
\end{split}
\end{equation*}
This concludes the proof of \eqref{claimA-ii}.

Using \eqref{pot} and \eqref{claimA-ii}, we deduce that 
\begin{eqnarray*}
&&\mathcal E_{B_4}(M_t)+\mathcal E_{B_4}(m_t) +\frac{1}{2} \iint_{A_4} (u(x)-u_{t}(x))_+(u(\bar x)-u_{t}(\bar x))_- K(x-\bar x)\,dx\,d\bar x\nonumber\\
 &&\hspace{1em} = \mathcal E_{B_4}(u)+\mathcal E_{B_4}(u_{ t})\nonumber 
\end{eqnarray*}
for every $t\in (0,1)$. As noticed in \cite{CSV}, the third term on the left-hand side is the important novelty with respect to the analogue equality in the local case (in which it does not appear). It will be responsible of the $BV$ estimate.

Indeed, analogously we have
\begin{eqnarray*}
&&\mathcal E_{B_4}(M_{-t})+\mathcal E_{B_4}(m_{-t})+ \frac{1}{2}\iint_{A_4} (u(x)-u_{-t}(x))_+(u(\bar x)-u_{- t}(\bar x))_- K(x-\bar x)dxd\bar x\nonumber\\
 &&\hspace{1em}  =\,\,  \mathcal E_{B_4}(u)+\mathcal E_{B_4}(u_{ -t}).\nonumber 
\end{eqnarray*}
Adding the last two  equalities, and using the key bounds \eqref{eqlemAstables} and \eqref{inner-stability}, we deduce
\begin{eqnarray}\label{step2}
&&\mathcal E_{B_4}(M_t)+\mathcal E_{B_4}(m_t)+\mathcal E_{B_4}(M_{-t})+\mathcal E_{B_4}(m_{-t}) + \nonumber \\
&&\hspace{2em} + \frac{1}{2}\iint_{A_4} (u(x)-u_{t}(x))_+(u(\bar x)-u_{t}(\bar x))_- K(x-\bar x)\,dx\,d\bar x +\\
 &&\hspace{2em} + \frac{1}{2}\iint_{A_4} (u(x)-u_{-t}(x))_+(u(\bar x)-u_{- t}(\bar x))_- K(x-\bar x)\,dx\,d\bar x \nonumber\\
 &&\hspace{1em}  =   2\mathcal E_{B_4}(u)+\mathcal E_{B_4}(u_t)+\mathcal E_{B_4}(u_{-t})\nonumber\\
&&\hspace{1em} \leq 4\mathcal E_{B_4}(u)+\eta t^2\nonumber\\
&&\hspace{1em} \leq \mathcal E_{B_4}(M_t)+\mathcal E_{B_4}(m_t)+\mathcal E_{B_4}(M_{-t})+\mathcal E_{B_4}(m_{-t}) + \big(\eta +2\nu)t^2
\end{eqnarray}
for $t\in(-t_0,t_0)$, with $t_0>0$ small enough (possibly depending on $\nu$, $W$, and $u$).  From this, the lemma follows immediately.
\end{proof}

The following result is the analogue of Lemma 2.5 in \cite{CSV}.

\begin{lem}\label{key}
Let $n\geq 2$, $\eta>0$, and $u\in C^1(B_2)$ satisfy $\|u\|_{L^\infty(\R^n)}\le 1$. Assume that for every $\boldsymbol v\in S^{n-1}$, there exists a sequence $t_k\rightarrow 0$,  such that
\begin{equation}\label{hp-lemkey}
\limsup_{k\rightarrow \infty}\frac{\|\big(u(\cdot)-u(\cdot-t_k\boldsymbol v)\big)_+\|_{L^1(B_1)}\ \|\big(u(\cdot)-u(\cdot-t_k\boldsymbol v)\big)_-\|_{L^1(B_1)}}{t_k^2}\leq \eta.
\end{equation}

Then, the following estimates hold:
\begin{equation}\label{min-v}
\min\left\{\int_{B_1}(\partial_{\boldsymbol v} u)_+\,dx\,,\,\int_{B_1}(\partial_{\boldsymbol v} u)_-\,dx\right\}\leq \sqrt \eta,
\end{equation}
\begin{equation}\label{max-v}
\max\left\{\int_{B_1}(\partial_{\boldsymbol v} u)_+\,dx\,,\,\int_{B_1}(\partial_{\boldsymbol v} u)_-\,dx\right\}\leq 2 |B_1^{(n-1)}|+\sqrt\eta,
\end{equation}
and
\begin{equation}\label{grad}
\int_{B_1}|\nabla u|\,dx\leq 2n  \left(  |B_1^{(n-1)}|+\sqrt \eta\right),
\end{equation}
where  $B_1^{(n-1)}$ denotes the unit ball of $\R^{n-1}$. 
\end{lem}
\begin{proof}
The  proof is the same as the one of Lemma 2.5 in \cite{CSV}. Here the situation is even simpler since, being $u\in C^1$, the directional derivatives of $u$ exist in the classical sense.
We sketch the main steps of the proof, for the convenience of the reader.

By assumption \eqref{hp-lemkey}, we have that
$$
\limsup_{k\rightarrow \infty}\frac{\min\left\{\|\big(u(\cdot)-u(\cdot-t_k\boldsymbol v)\big)_+\|_{L^1(B_1)},\, \|\big(u(\cdot)-u(\cdot-t_k\boldsymbol v)\big)_-\|_{L^1(B_1)}\right\}}{|t_k|}\leq \sqrt\eta.
$$
Passing to the limit as $t_k\rightarrow 0$ we immediately get \eqref{min-v}.

To prove \eqref{max-v}, we simply observe that
\[ \int_{B_1}\partial_{\boldsymbol v}u \,dx = \int_{B_1}\lim_{k\rightarrow \infty}\frac{u(x+t_k{\boldsymbol v})-u(x)}{t_k} \,dx =\lim_{k\rightarrow \infty}\frac{\int_{B_1+t_k\boldsymbol v} u \,dx-\int_{B_1}u\,dx}{t_k} \]
and hence, since  $\|u\|_{L^\infty(\R^n)}\le 1$,
\[\left|\int_{B_1}\partial_{\boldsymbol v} u\,dx\right|\le  \limsup_{k\rightarrow \infty}\frac{\bigl|(B_1+t_k{\boldsymbol v})\setminus B_1\big|
+\bigl|B_1\setminus (B_1+t_k{\boldsymbol v})\big|
}{|t_k|} \le 2 |B_1^{(n-1)}|.\]

This, together with \eqref{min-v}, leads to \eqref{max-v}, since $\partial_{\boldsymbol v}u=(\partial_{\boldsymbol v}u)_+ - (\partial_{\boldsymbol v}u)_-$. Moreover, using that $|\partial_{\boldsymbol v}u|=(\partial_{\boldsymbol v}u)_+ + (\partial_{\boldsymbol v}u)_-$, we also deduce that
\[\int_{B_1}|\partial_{\boldsymbol v} u|\,dx\le  2\left( |B_1^{(n-1)}| + \sqrt \eta \right).\]

Finally, since $|\nabla u|\leq |\partial _{e_1} u|+\ldots + |\partial _{e_n} u|$, we conclude \eqref{grad}.
\end{proof}

Before giving the proof of Theorem \ref{BV}, we recall the following abstract result due to L. Simon \cite{Simon} (see also Lemma 3.1 in \cite{CSV}).
\begin{lem}[\cite{Simon}]\label{lem_abstract}
Let $\beta\in \R$ and $C_0>0$. Let $S: \mathcal B \rightarrow [0,+\infty)$ be a nonnegative function defined on the class $\mathcal B$  of open balls contained in the unit ball $B_1$ of $\R^n$ and satisfying the following subadditivity property:
\[ S(B)\le \sum_{j=1}^N S(B^j) \quad \mbox{ whenever }  N\in\Z^+, \{B^j\}_{j=1}^N \subset \mathcal B, \text{ and } B \subset \bigcup_{j=1}^N B^j. \]

It follows that there exists a constant $\delta>0$, depending only on $n$ and $\beta$,  such that if
\begin{equation}\label{hp-lem}
 \rho^\beta S\bigl(B_{\rho/4}(x_0)\bigr) \le \delta \rho^\beta S\bigl(B_\rho(x_0)\bigr)+ C_0\quad \mbox{whenever }B_\rho(x_0)\subset B_1,
 \end{equation}
then
\[ S(B_{1/2}) \le CC_0\]
for some constant $C$ which depends only on $n$ and $\beta$.
\end{lem}

We can now give the proof of Theorem \ref{BV}.

\begin{proof}[Proof of Theorem \ref{BV}]

We divide the proof into two steps.

\vspace{3pt}

{\em -Step 1.} We show that if $s\in (s_0, 1)$ and $u: B_6 \to (-1,1)$ is a stable solution of the semilinear equation $L_Ku + W'(u) =0$ in  $B_6$ then, for any given $\delta >0$, we have the estimate
\begin{equation}\label{wnhiotheoiht}
\int_{B_1}|\nabla u| \,dx \le  \frac{C_\delta}{1-s} +  \delta\, \int_{B_4}|\nabla u|  \,dx
\end{equation}
where $C_\delta$ depends only on $\delta$, $n$, $\lambda$, $\Lambda$, and $s_0$ (in particular, it does not depend on~$W$).

Indeed, note that  in this setting Lemmas \ref{lem2A},  \ref{lemEFtdelta}, and \ref{key} apply to $u$. 

Hence, by Lemma \ref{lemEFtdelta}, for every $\nu >0$ there exists $t_0>0$ such that 
\begin{eqnarray*}
&&\min\left\{ \iint_{A_4}(u(x)-u_t(x))_+(u(\bar x)-u_t(\bar x))_-K(x-\bar x)\,dx \,d\bar x,\right.\\
&&\hspace{4em}\left.\iint_{A_4}(u(x)-u_{-t}(x))_+(u(\bar x)-u_{-t}(\bar x))_-K(x-\bar x)\,dx \,d\bar x \right\}\leq \left(\eta+2\nu\right)t^2
\end{eqnarray*}
holds for every $t\in (0,t_0)$, where
\begin{equation}\label{eta}
\eta=C\mathcal E^{\rm Sob}_{B_4}(u)
\end{equation}
and $C$ depends only on $n$, $\lambda$, and $\Lambda$.
Now, by \eqref{L0} and since $s\in (0,1)$, we have $K\ge (2-s) 2^{-n-s}\lambda \ge 2^{-n-1}\lambda$  in~$B_2$. We deduce that there is some sequence $t_k\in(-1,1)$ with
$t_k \rightarrow 0$ such that
$$
\limsup_{k\rightarrow \infty}\frac{\|\big(u(\cdot)-u(\cdot-t_k\boldsymbol v)\big)_+\|_{L^1(B_1)}\ \|\big(u(\cdot)-u(\cdot-t_k\boldsymbol v)\big)_-\|_{L^1(B_1)}}{t_k^2}\leq \eta ,
$$
after changing the value of $C$ in \eqref{eta}.

We can now apply Lemma \ref{key} and, thanks to \eqref{grad}, we arrive at
\begin{equation}\label{key1}
 \int_{B_1}|\nabla u|\,dx\leq C \left( 1+ \sqrt{\mathcal E^{\rm Sob}_{B_4}(u)}\right)<\infty,
\end{equation}
where $C$ depends on $n$, $\lambda$, and $\Lambda$.

 In order to keep track of the precise dependence of the constants on $s$, as $s\uparrow 1$, in what follows $C$ will denote (possibly different) positive constants which depend only on $n$, $\lambda$, $\Lambda$, and $s_0$.

Defining $V(z):= |\nabla u(z)|$ for $z\in B_4$ and $V(z):= 0$ for $z\in \R^n\setminus B_4$, and given $x$ and $\bar x$ both in $B_4$, note that we have
$$
|u(x)-u(\bar x)| =\left| \int_0^1 (\bar x -x ) \cdot \nabla u(x+t(\bar x -x))\,dt\right|
\leq |x -\bar x| \int_0^1 V(x+t(\bar x -x))\,dt.
$$
Using this, that $K\in \mathcal L_2$, and $\|u\|_{L^\infty(\R^n)}\leq 1$,  we deduce (thanks to the presence of the factor $1-s$ on the next left-hand side) that
\begin{eqnarray}\label{key2}
(1-s)\mathcal E^{\rm Sob}_{B_4}(u)
&\le & \frac{1}{2} \Lambda(1-s) \bigg(\iint_{B_4\times B_4} \frac{|u(x)-u(\bar x)|^2}{|x-\bar x|^{n+s}}  \,dx\,d\bar x  \nonumber \\
& & \hspace{4cm}
+ 2\iint_{B_4\times  B_4^c} \frac{2^2}{|x-\bar x|^{n+s}}\,dx \,d\bar x\bigg)\nonumber 
\\
& \le&  \frac{1}{2} \Lambda (1-s)\iint_{B_4\times B_4} \frac{|u(x)-u(\bar x)|^2}{|x-\bar x|^{n+s}}  \,dx\,d\bar x + C\nonumber 
\\
&\le &  \Lambda (1-s)\iint_{B_4\times B_4} \frac{| u(x)-u(\bar x)|}{|x-\bar x|^{n+s}}   \,dx\,d\bar x + C\nonumber 
\\
&\leq & \Lambda(1-s)\int_{B_4} dx \int_ {B_8} dy \, |y|^{1-n-s} \int_0^{1} dt \,V(x + ty) +C\nonumber 
\\
&\leq& \Lambda(1-s)\int_ {B_8} dy \, |y|^{1-n-s} \int_0^{1} dt \int_{\R^n} dx  \,V(x + ty) +C\nonumber 
\\        
& = &  \Lambda(1-s)\int_ {B_8} dy \, |y|^{1-n-s} \int_{\R^n}dz \, V(z) +C\nonumber \\
&\le & C\int_{\R^n}V(z)\,dz +C \nonumber \\
&=& C\left( 1 + \int_{B_4}|\nabla u|\,dx \right),
\end{eqnarray}
where $C$ depends only on $n$, $\lambda$, $\Lambda$, and $s_0$; in particular, it stays bounded as $s\uparrow 1$.

Hence,  \eqref{key1}, \eqref{key2}, and the Cauchy-Schwarz inequality lead to
\begin{equation}\label{key3}
\begin{split}
\int_{B_1}|\nabla u| \,dx &\le  C\left(1+ \frac{1}{(1-s)^{1/2}}\left(1+\int_{B_4}|\nabla u| \,dx  \right)^{1/2}\right)
 \\
 & \le  C\left(1+\frac{1}{\delta(1-s)}+\delta\right)  + \delta\, \int_{B_4}|\nabla u|  \,dx
 \end{split}
\end{equation}
for all $\delta>0$.

\vspace{3pt}

{{\em -Step 2.} Now, in order to conclude the proof of the theorem we observe that, since its statement is scaling invariant (it is independent of $W$), we may assume without loss generality that $R=1$. So, let now $u: \R^n \to (-1,1)$ be a stable solution of $L_Ku+W'(u)=0$ in $B_{2R}=B_2$. Given $B_\rho(x_0)\subset B_1$ then the rescaled function $\tilde u(y) =u\left( x_0 +\rho y/4\right)$ is a stable solution in $B_6$ (since $x_0+\rho B_1 \subset B_1 \Rightarrow  x_0+\frac{\rho}{4} B_6 \subset B_2$)} of the equation $L_{\tilde K}\tilde u+(\rho/4)^s W'(\tilde u)=0$,
where
\[   \tilde K(z) : = {(\rho/4)}^{n+s} K(\rho z/4) \quad \mbox{belongs again to }\mathcal L_2(s,\lambda, \Lambda).\]
Thus, rescaling the estimate \eqref{wnhiotheoiht}, which applies to $\tilde u$, we obtain
\begin{equation}
 \rho^{1-n} \,\int_{B_{\rho/4}(x_0)}|\nabla u| \,dx \le \delta \,\rho^{1-n}\, \int_{B_{\rho}(x_0)}|\nabla u|\,dx + \frac{C_\delta}{1-s} 
\end{equation}
for all $\delta>0$, where $C_\delta$ depends only on $\delta$, $n$, $\lambda$, $\Lambda$, and $s_0$.

Therefore, considering the subadditive function 
\[S(B):=\int_{B}|\nabla u|\,dx \]
on the class of balls,
and taking $\beta:=1-n$ and  $\delta$ as given by Lemma \ref{lem_abstract} (and hence depending only on $n$), we find that
\[  \int_{B_{1/2}}|\nabla u| \,dx = S(B_{1/2}) \le \frac{C}{1-s},\]
where $C$ is a constant which depends only on $n$, $\lambda$, $\Lambda$, and $s_0$.

By scaling and using  a standard covering argument, we obtain the same estimate with $B_{1/2}$ replaced by $B_1=B_R$, concluding the proof.
\end{proof}

\begin{proof}[Proof of Corollary \ref{energy-est}]
Let $s\in (0,1)$.  Let $u_R: =  u(R\, \cdot\,)$. We combine the estimate $\int_{B_1}|\nabla u_R|\,dx \le C(1-s)^{-1}$ of Theorem \ref{BV} with the  interpolation type inequality 
\[(1-s)\mathcal E^{\rm Sob}_{B_1}(u_R)  \le C\bigl(1+\int_{B_1}|\nabla u_R|\,dx\bigr)\] ---which we proved, in a different ball,  in \eqref{key2}--- to obtain $\mathcal E^{\rm Sob}_{B_1}(u_R) \le C (1-s)^{-2}$. Now, the corollary follows after rescaling.
\end{proof}

\section{Sobolev energy controls potential energy}\label{sec-4} 

In this section we give the  proof of Proposition \ref{DircontrolsPot2}. In particular, we will have $s\in (0,2]$ and we include the case of the classical Allen-Cahn equation. All results in this section require the cubes to have large enough diameter and, in particular, the equation to be posed in a sufficiently large ball (as in the statement Proposition~\ref{DircontrolsPot2}).

The proof is based on  a suitable cube decomposition ---which may be of independent interest--- and covering arguments. 
We identify two types of cubes: of `type I'' and of ``type~II''.
Cubes of type I will be those containing at least one point where $u\sim 0$.
Cubes of type II are cubes in which either $u\sim1$ or $u\sim -1$ in the whole cube.
By cube we always mean a set of the form $Q=x_0+(-l,l)^n$ for some $x_0\in\R^n$ and $l>0$.

The following three lemmas are used in the proof of Proposition \ref{DircontrolsPot2}.
Essentially, the first one (Lemma \ref{tuttookintype1}) is used to show that the contribution of a cube of type~I to the total Sobolev energy controls its contribution to the potential energy. The third lemma (Lemma \ref{tuttookintype2}) will prove this same property for cubes of type II.

The second lemma states that if a cube is not of type I (i.e., it contains no points where $u\sim0$) then the corresponding half-cube is of type II (i.e., we have either $u\sim 1$ or $u\sim -1$ in all of it). Thus, through an appropriate covering argument, this second lemma establishes that there is essentially a dichotomy between cubes of type I and cubes of type II.

Throughout this section $W$ is the double-well potential $\frac{1}{4}(1-u^2)^2$. It satisfies 
\begin{equation}\label{WWW}
\begin{cases}
-W''(t)= -W''(-t) \ge \nu_0 \quad &  \mbox{for } 0\le t\le c_0
\\
W''(t)= W''(-t) \ge \nu_0 \quad  &\mbox{for } 1-c_0 \le t\le 1
\\
-W'(t) = W'(-t) \ge \nu_1 \quad &  \mbox{for } \frac{c_0}{2}\le t\le 1-c_0
\end{cases}
\end{equation}
for some constants $c_0\in (0,\frac{2}{3})$, $\nu_0>0$, and $\nu_1>0$ which are totally universal. To check quickly their existence, note that once $c_0$ and $\nu_0$ have been chosen small enough to guarantee the first two properties, the third condition will be satisfied for $\nu_1$ small enough.

\begin{lem}\label{tuttookintype1}
Given $s_0>0$, let $n\ge 2$, $s\in (s_0,2]$,  $W(u)=\frac 1 4 (1-u^2)^2$, and $K$ satisfy \eqref{L0} and \eqref{L2}. Let $R>0$ and $u:\R^n\rightarrow (-1,1)$ be a stable solution of $L_{K}u + W'(u)=0$ in~$B_{2R}$.  

Then, there exist constants $D_0\ge 1$ and $d_0>0$, which depend only on $n$, $\lambda$, $\Lambda$, and $s_0$, such that the following statement holds true.

For every cube $Q\subset B_R$ satisfying ${\rm diam}(Q)\ge D_0$ and
\[
 \{ |u|\le c_0/2\}\cap Q\neq\varnothing
\]
for the constant $c_0$ in \eqref{WWW}, there are two cubes $Q^{(1)}$, $Q^{(2)}$ contained in $Q$ and with diameters $d_0$ for which
\begin{equation}\label{goalQi}
\inf_{Q^{(1)}} u - \sup_{Q^{(2)}} u  \ge c_0/4.
\end{equation}
\end{lem}
\begin{proof}
Within the proof, $C$ will denote different positive constants which depend only $n$, $\lambda$, $\Lambda$, and $s_0$. The constants $D_0$ and $d_0$ will be chosen to have this same dependence.

Let $Q\subset B_R$ be some cube satisfying ${\rm diam}(Q)=D \ge D_0$.

\vspace{3pt}
-{\em Step 1.} Let us  prove that  ${\rm osc }_Q u \ge c_0/2$.
Indeed, arguing by contradiction assume that
\[
 {\rm osc }_Q u < c_0/2,  \quad \mbox{and recall that we have }\{ |u|\le c_0/2\}\cap Q\neq\varnothing.
\]
Then, $|u|\le c_0$ in $Q$ and using \eqref{WWW}  and the stability inequality \eqref{stable}, we obtain 
\begin{eqnarray}\label{ineq2-s}
\nu_0 \int_{\R^n} \xi^2\,dx &\le &-\int_{\R^n} W''(u) \xi^2\,dx \nonumber
\\
&\le & \frac{1}{2}\iint_{\R^n\times \R^n} \big|\xi(x)-\xi(\bar x)\big|^2 K(x-\bar x)\,dx \,d\bar x\nonumber
\\
&\le& C(2-s) \iint_{\R^n\times \R^n} \frac{\big|\xi(x)-\xi(\bar x)\big|^2}{|x-\bar x|^{n+s}} \,dx \,d\bar x
\end{eqnarray}
for every $\xi\in C^1_c(Q)$.

Now, if the diameter $D$ of $Q$ is large enough we may contradict this inequality just by scaling. Indeed, denote by $Q_d$ the cube centered at $0$ and with diameter $d$, and let $x_0$ be the center of $Q$ and $D$ its diameter.  Choose (universally) a function $\eta\not\equiv 0$ in $C^\infty_c(Q_1)$. Take now $\xi(x)= \eta((x-x_0)/D)$ and notice that
\begin{equation}\label{ineq2-sbis}
(2-s)\iint_{\R^n\times \R^n} \frac{\big|\xi(x)-\xi(\bar x)\big|^2}{|x-\bar x|^{n+s}} \,dx \,d\bar x \le C D^{n-s} \Vert\nabla\eta\Vert_{L^\infty(Q_1)}^2.
\end{equation}
while
\begin{equation}\label{ineq2-sbis2}
 \int_{\R^n} \xi^2\,dx = cD^n
\end{equation}
Choosing $D\geq D_0\geq 1$ large enough (note that $D_0^{n-s}\leq D_0^{n-s_0}$), we contradict \eqref{ineq2-s}. This proves Step 1.

Note that the same argument applies in the case $s=2$ when we replace $2-s$ times the double integral by the classical Dirichlet norm.

\vspace{3pt}
-{\em Step 2}. We now use Step 1, and regularity estimates for the equation,
to show that there are two cubes $Q^{(1)}$, $Q^{(2)}$ contained in $Q$ and with diameters $d_0$ such that $\inf_{Q^{(1)}} u - \sup_{Q^{(2)}} u  \ge c_0/4$.

Indeed, first note that, as shown in Appendix~\ref{app-C}, we have the estimate $|\nabla u|\le \bar C$ in $B_R$, for some constant $\bar C$ depending only on the quantities stated in the beginning of the proof; here we use that $Q\subset B_R$ and thus $ \textrm{diam}(B_R) \ge \textrm{diam}(Q)=D\ge D_0\ge 1$. Let
$$
d_0= \min \left\{ D_0, \frac{c_0}{8 \bar C}\right\}.
$$
Now, by Step 1 there are two points $x_1,x_2\in \overline{Q}$ such that $u(x_1)-u(x_2)\ge c_0/2$.
Let $Q^{(i)}$ be any two cubes with diameter $d_0$ such that $x_i\in \overline{Q^{(i)}}\subset \overline Q$ (recall that $d_0\le D_0\le \textrm{diam}(Q)$). Then we readily show that \eqref{goalQi} is satisfied by these cubes.
\end{proof}

We can now state the second lemma.

\begin{lem}\label{notype1impliestype2}
Given $s_0>0$, let $n\ge 2$, $s\in (s_0,2]$,  $W(u)=\frac 1 4 (1-u^2)^2$, and $K$ satisfy \eqref{L0} and \eqref{L2}. Let $R>0$ and $u:\R^n\rightarrow (-1,1)$ be a stable solution of $L_{K}u + W'(u)=0$ in~$B_{2R}$.  

Then, there exists a constant $D_0\ge 1$, which depends only on $n$, $\lambda$, $\Lambda$, and $s_0$, such that the following statement holds true.

For every cube $Q\subset B_R$ with ${\rm diam}(Q)\ge D_0$ we have
\begin{equation}\label{firstcase}
u \ge 1-c_0  \mbox{ in }Q' \quad\text{ if } \, u > c_0/2  \mbox{ in }Q \
\end{equation}
and
\begin{equation}\label{secondcase}
u \le -1+c_0  \mbox{ in }Q'  \quad\text{ if } \, u <-c_0/2  \mbox{ in }Q ,
\end{equation}
where $c_0$ is the constant in \eqref{WWW} and $Q'$ is the cube with the same center as $Q$ and half its diameter.

\end{lem}
\begin{proof}
Within the proof, the constant $C$ will depend only $n$, $\lambda$, $\Lambda$, and $s_0$. The constant $D_0$ will be chosen to have this same dependence.

Let $Q\subset B_R$ be some cube satisfying ${\rm diam}(Q)=D \ge D_0$.

We prove only \eqref{firstcase} since, by the even symmetry of $W$,  \eqref{secondcase} follows applying \eqref{firstcase}  to $-u$.

Assume that $u >  c_0/2$ in $Q$. Using \eqref{WWW} we have 
\[  L_Ku  =- W'(u)\ge \nu_1\quad\mbox{in }Q\setminus \{u\ge 1-c_0\}.\]
 Let $\eta\in C^\infty(\R^n)$ with $\eta \equiv 1-c_0$ in $Q_{1/2}$,  $-1\le \eta\le  1-c_0$  in $Q_1$,  and $\eta \equiv -1$ outside $Q_1$ (recall that $Q_r$ denotes the cube centered at $0$ and with diameter $r$). Consider
\[\tilde\eta(x)= \eta((x-x_0)/D)\]
 where $x_0$ is the center of $Q$ and $D$ its diameter. Note that $D\ge D_0$ and that, by taking $D_0$ large enough, we will have
\[ \big|  L_K \tilde\eta \big| \le C D^{-s}\le C D_0^{-s} \le \nu_1 \]
with $C$ and $D_0$ as stated in the beginning of the proof.\footnote{{Here we use the presence of the factor $2-s$ on the upper bound for the kernel $K$ as in \eqref{ineq2-sbis}. Instead, when $s=2$, the bound is obvious.}}
Let us then show that
\begin{equation}\label{goal}
 u\ge \tilde \eta \quad\mbox{in  }Q.
\end{equation}

Indeed, let $U= Q\setminus \{u\ge 1-c_0\}$. Since $ \tilde \eta\le  1-c_0$  in $Q$ and $\tilde \eta\equiv -1$ outside of $Q$ we have $\tilde\eta\le u$ outside of $U$. On the other hand
\[ L_K (u-\tilde \eta) \ge  L_K u - \big|  L_K \tilde\eta\big| \ge   0 \quad \mbox{in }U,\]
and thus the maximum principle leads to \eqref{goal}. Finally, since by construction $\tilde \eta \equiv 1-c_0$ in $Q'$ the lemma follows.
\end{proof}

{
Finally, we can state our last auxiliary lemma in this section.
\begin{lem}\label{tuttookintype2}
Given $s_0>0$, let $n\ge 2$, $s\in (s_0,2]$,  $W(u)=\frac 1 4 (1-u^2)^2$, and $K$ satisfy \eqref{L0} and \eqref{L2}. Let $R>0$ and $u:\R^n\rightarrow (-1,1)$ be a stable solution of $L_{K}u + W'(u)=0$ in~$B_{2R}$.  

Assume that for some cube  $Q\subset B_R$ we have that
\[
1-|u| \le c_0 \quad \mbox{in } Q,
\]
where $c_0$ is the constant in \eqref{WWW}. 

Then
\begin{equation}\label{notseq2}
(2-s)\int_Q \int_{\R^n} \frac{|u(x)-u(\bar x)|^2}{|x-\bar x|^{n+s}} \,dx\,d\bar x \ge \kappa_0\int_{Q} (1-|u|)^2\,dx,
\end{equation}
where  $\kappa_0>0$ is a constant depending only on $n$, $\lambda$, $\Lambda$, $s_0$, and ${\rm diam}(Q)$.

When $s=2$,  the left hand side of \eqref{notseq2}  must be replaced by $\int_Q |\nabla u|^2\,dx$.
\end{lem}

\begin{proof}
We will do the proof in the case  $1-c_0\le u\le 1$ in $Q$ (the case $-1\le u \le -1+c_0$ is similar).
Since for simplicity the lemma is stated for the explicit quartic potential $W(u) = \frac 1 4 (1-u^2)^2$ we may take  $c_0 := 1 -1/\sqrt 2$, and hence
\begin{equation} \label{21huwgiuw}
W'' = 3u^2-1 \in [1/2,2] \quad  \mbox{for } u\in  [1-c_0, 1].
\end{equation}

Let $v : = 1-u$. By assumption $v\ge 0$ in all of $\R^n$  and $v\le c_0$ in $Q$. Hence, using \eqref{21huwgiuw} and
\[
L_K v =  -L_K  u  = W'(u) = W'(u)-W'(1)  
\]
we obtain 
\begin{equation}
\label{wjhtiohweio}
-2 v \le L_K v \le -\frac{1}{2}v \quad \mbox{in }Q.
\end{equation}

Notice that 
\[
(2-s)\int_{Q} \int_{\R^n} \frac{|u(x)-u(\bar x)|^2}{|x-\bar x|^{n+s}} \,dx\,d\bar x  =  (2-s)\int_{Q} \int_{\R^n} \frac{|v(x)-v(\bar x)|^2}{|x-\bar x|^{n+s}} \,dx\,d\bar x
\]
and
\[
\int_{Q} W(u) dx \le  \int_{Q} \frac { 2^2(1-u)^2}{4} dx =  \int_{Q} v^2dx.
\]

Hence, to prove the lemma we need to show that there exists  a constant $C$ depending only on $n$, $\lambda$, $\Lambda$, $s_0$, and ${\rm diam}(Q)$, such that 
\[
C (2-s)\int_{Q} \int_{\R^n} \frac{|v(x)-v(\bar x)|^2}{|x-\bar x|^{n+s}} \,dx\,d\bar x  \ge \int_{Q} v^2dx.
\]

To prove this, let $\tilde v : = v/ \|v\|_{L^2(Q)}$ and let 
\[
 \kappa: = (2-s)\int_{Q} \int_{\R^n} \frac{|\tilde v(x)-\tilde v(\bar x)|^2}{|x-\bar x|^{n+s}} \,dx\,d\bar x.
\]
In the remaining of the proof we bound $\kappa$ by below.

First, by the fractional Poincar\'e inequality we  have
\[
 \kappa \ge (2-s) \int_{Q} \int_{Q} \frac{|\tilde v(x)-\tilde v(\bar x)|^2}{|x-\bar x|^{n+s}} \,dx\,d\bar x 
 \ge c_0   \int_{Q}  (\tilde v-t)^2,
\]
where $t = \frac 1{|Q|}\int_{Q} \tilde v\, dx$ and $c_0=c_0(n,s_0, {\rm diam}(Q))>0$.

Note that, by the triangle inequality, we have
\begin{equation}
\begin{split}
 1- (\kappa/c_0)^{1/2} &\le  \|\tilde v\|_{L^2(Q)}-    \|\tilde v-t\|_{L^2(Q)}   \\&
 \le t|Q|^{1/2} \le   \|\tilde v\|_{L^2(Q)}  + \|\tilde v-t\|_{L^2(Q)} \le  1+ (\kappa/c_0)^{1/2}.
 \end{split}
\end{equation}
Hence, if $\kappa$ is sufficiently small (which we may of course assume) we have
\begin{equation}\label{wtiowhtoiwh}
\frac 1 2| Q|^{-1/2}\le t \le 2|Q|^{-1/2}.
\end{equation}

On the other hand, $\tilde v$ satisfies $L_K \tilde v \le -\frac 1 2\tilde  v$. Thus, if we  fix  $\xi \in C^\infty_c (Q')$ such that $\int_{\R^n} \xi \,dx =1$ and $\xi\ge0$, where $Q'$ is the cube with the same center as $Q$ and half its diameter, we have 
\begin{equation}\label{wtnwowneoktgbeowb}
\begin{split}
 \int_{\R^n}  \tilde v L_K\xi \,dx  &\le -\frac{1}{2} \int_{\R^n}  \tilde v \xi \,dx \le -\frac 1 2 \int_{\R^n}  t \xi \,dx +  \int_{\R^n}  |\tilde v -t| \xi \,dx \\ &\le - \frac t 2  + \|\xi\|_{L^2(Q)} \|\tilde v -t\|_{L^2(Q)} 
\le - \frac t 2  + C (\kappa/c_0)^{1/2}.
 \end{split}
\end{equation}
At the same time, since  $\int_{\R^n}   L_K\xi \,dx =0$ we have
\[
\int_{\R^n}  \tilde v L_K\xi \,dx=   \int_{Q}  (\tilde v-t) L_K\xi \,dx  +  \int_{\R^n\setminus Q}  (\tilde v-t) L_K\xi \,dx.
\]

Now similarly as before
\[
\bigg|\int_{Q}  (\tilde v-t) L_K\xi \,dx\bigg| \le \|L_K\xi\|_{L^2(Q)} \|\tilde v -t\|_{L^2(Q)} \le C (\kappa/c_0)^{1/2}.
\]

Also, since  $|L_K\xi (x)|  \le C (2-s) {\rm dist }(x,Q')^{-n-s}$ for $x\in \R^n\setminus Q$, we have 
\[
\bigg|\int_{\R^n\setminus Q}  (\tilde v-t) L_K\xi \,dx \bigg| \le  C (2-s) \int_{\R^n\setminus Q} \frac{|\tilde v-t|}{  {\rm dist }(x,Q')^{n+s}} \,dx.
\]
Now, since  
\[
\int_{\R^n\setminus Q} \frac{2-s}{ {\rm dist }(x,Q')^{n+s}}\,dx \le C
\]
and 
\[
 {\rm dist }(x,Q')^{-n-s} \le C |x-\bar x|^{-n-s} \quad \mbox{for all } \bar x \in Q
 \]
we obtain, using also the definition of $t$,
\[
\begin{split}
(2-s) \int_{\R^n\setminus Q} \frac{|\tilde v(x) -t|}  { {\rm dist }(x,Q')^{n+s}} \,dx &\le  (2-s) \frac 1 {|Q|} \int_Q d\bar x \int_{\R^n\setminus Q}\,dx \, (2-s)\frac{|\tilde v(x) -\tilde v(\bar x)|}  { {\rm dist }(x,Q')^{n+s}} 
 \\
  &\le   C \bigg( \int_Q d\bar x \int_{\R^n\setminus Q}\,dx  \,  (2-s)\frac{|\tilde v(x) -\tilde v(\bar x)|^2}  { {\rm dist }(x,Q')^{n+s}}  \bigg)^{1/2}
 \\
 &\le C \bigg( \int_Q \int_{\R^n\setminus Q} (2-s) \frac{|\tilde v(x)-\tilde v(\bar x)|^2}{|x-\bar x|^{n+s}} \,dx\,d\bar x  \bigg)^{1/2}
 \\
 & \le C \kappa^{1/2}.
\end{split}
\]

Hence, we have shown
\[
\bigg| \int_{\R^n}  \tilde v L_K\xi \, dx\bigg| \le C \kappa^{1/2}.	
\]
Using also \eqref{wtnwowneoktgbeowb}, we deduce
\[
\frac t 2  \le -\int_{\R^n}  \tilde v L_K\xi    + C\kappa^{1/2} \le 2C\kappa^{1/2}. 
\]
Hence, recalling \eqref{wtiowhtoiwh}, this shows the desired  lower bound on $\kappa$.
\end{proof}
}

We can now give the
\begin{proof}[Proof of Proposition \ref{DircontrolsPot2}]
Throughout the proof, all the constants will depend only $n$, $s$, $\lambda$, $\Lambda$, and $s_0$. Let $D_0$ and $d_0$ be constants for which Lemmas \ref{tuttookintype1} and \ref{notype1impliestype2} hold  ---for this, take them to be, respectively, the largest of the constants $D_0$ in the lemmas and the smallest of the constants $d_0$.

Let $\mathcal F$ denote the family of cubes $Q$ with center in the lattice $(D_0/\sqrt{n})\, \mathbb Z^n$ and side-length $2D_0/\sqrt{n}$, that is, of the form
$$
Q=\frac{D_0}{\sqrt{n}}\,z_0+\left( -\frac{D_0}{\sqrt{n}},\frac{D_0}{\sqrt{n}} \right)^n \quad \text{for some } z_0\in\Z^n.
$$
Let $\mathcal F_R$ be the family of those cubes $Q\in \mathcal F$ such that $Q\subset B_R$. 
Given a cube $Q$, denote by $Q'$ the cube with the same center as $Q$ and half its diameter. Finally, define
$$
R_0:=\frac{3}{2}D_0,
$$
and recall that we assume $R>R_0$. 

It is then easy to check that:

\begin{itemize}
\item[(a)]
For all $Q\in \mathcal F_R$, we have $Q\subset B_R$ and $\textrm{diam}(Q)=2D_0$.
\vspace{1mm}

\item[(b)] Each point of $\R^n$ belongs to at most  $2^n$ cubes of the family $\mathcal F_R$.
\vspace{1mm}

\item[(c)] The union $\bigcup_{Q\in {\mathcal F_R}} \, Q'$ is disjoint and covers $B_{R-R_0}$ except for a set of measure zero.
\end{itemize}

Throughout the proof, it is easy to check that all arguments hold true when $s=2$ by replacing, within the double integrals, $2-s$ times one of the integrals by the pointwise gradient squared.

We first notice that, by properties (a) and (b),
\begin{equation}\label{chain}
\begin{split}
\mathcal E^{\rm Sob}_{B_R}(u) &\ge \frac{1}{4} \int_{B_R}\int_{\R^n} |u(x)-u(\bar x)|^2 K(x-\bar x)\,dx\,d\bar x
\\
&\ge
\frac{\lambda(2-s)}{4} \int_{B_R}\int_{\R^n} \frac{ |u(x)-u(\bar x)|^2 }{ |x-\bar x|^{n+s}} \,dx\,d\bar x
\\
&\ge \frac{\lambda(2-s)}{4\cdot2^n} \sum_{Q\in \mathcal F_R} \int_{Q}\int_{\R^n} \frac{ |u(x)-u(\bar x)|^2 }{ |x-\bar x|^{n+s}} \,dx\,d\bar x.
\end{split}
\end{equation}

Next, for each  $Q\in \mathcal F_R$ we have the following dicothomy. Either
\begin{equation}\label{type1}
\{ |u|\le c_0/2\}\cap Q\neq\varnothing
\end{equation}
or
\begin{equation}\label{type2}
\{ |u|\le c_0/2\}\cap Q =\varnothing.
\end{equation}

Now, if  $Q\in \mathcal F_R$ satisfies \eqref{type1} then by Lemma \ref{tuttookintype1}  (since $Q$ has diameter $2D_0$) there are two cubes $Q^{(1)}$ and $Q^{(2)}$ contained in $Q$, and with diameters $d_0$, such that
\[
\inf_{Q^{(1)}}u -\sup_{Q^{(2)}} u \ge c_0/4.
\]
It thus follows, using  the fractional Poincar\'e inequality and  \[\int_{Q'} W(u)\,dx\le |Q'| \max_{[-1,1]} W \le C,\] that
\begin{equation}\label{caseI}
\begin{split}
C (2-s)\int_{Q}\int_{Q} \frac{ |u(x)-u(\bar x)|^2 }{ |x-\bar x|^{n+s}}\,dx\,d\bar x  &\ge \inf_{t\in \R} \int_{Q}  |u(x)-t|^2 \,dx 
\\ &
\ge \left|\frac{(\inf_{Q^{(1)}}u -\sup_{Q^{(2)}} u)^2 }{2}\right|^2 \big| Q^{(i)}\big| 
\\
&\ge \left(\frac{c_0}{8}\right)^2 \bigg( \frac{d_0}{\sqrt n}\bigg)^n  \ge  c\int_{Q'} W(u) \, dx
\end{split}
\end{equation}
for some constant $c>0$.

On the other hand, if $Q$ satisfies \eqref{type2} then by Lemma \ref{notype1impliestype2} we have that
\[ |u|\ge 1-c_0 \quad \mbox{in }Q'.\]
Hence, using  Lemma \ref{tuttookintype2}  (with $Q$ replaced by $Q'$) we have
\begin{equation}\label{caseII}
(2-s)\int_{Q}\int_{\R^n} \frac{ |u(x)-u(\bar x)|^2 }{ |x-\bar x|^{n+s}}\,dx\,d\bar x \ge  c\int_{Q'} (1-|u|)^2\,dx \ge  c\int_{Q'} W(u)\,dx.
\end{equation}

Finally, recalling property (c) above, we combine \eqref{caseI} and \eqref{caseII} with \eqref{chain} to obtain
\[
\begin{split}
\mathcal E^{\rm Sob}_{Q_R}(u) &\ge \frac{\lambda(2-s)}{2^n} \sum_{Q\in \mathcal F_R} \int_{Q}\int_{\R^n} \frac{ |u(x)-u(\bar x)|^2 }{ |x-\bar x|^{n+s}} \,dx\,d\bar x
\\
&\ge c \sum_{Q\in \mathcal F_R} \int_{Q''} W(u) \, dx
\\
&\ge c \int_{B_{R-R_0}} W(u) \, dx
\\
&= c\ \mathcal E^{\rm Pot}_{B_{R-R_0}}(u),
\end{split}
\]
that finishes the proof.
\end{proof}

Let us give now the proof of  the bound for the full energy when $s\in (0,1)$.

\begin{proof} [Proof of Theorem \ref{thm1}] 
Let $s\in (s_0,1)$. By Corollary \ref{energy-est} we have that
\begin{equation}\label{est-Dir}\mathcal E_{B_{R}}^{\rm Sob}(u)\leq \frac{C}{(1-s)^2}R^{n-s}.
\end{equation}
Now, using Proposition \ref{DircontrolsPot2} we deduce from \eqref{est-Dir} the same estimate for ${\mathcal E}_{B_{R/2}}^{\rm Pot}(u)$, provided $R\ge 2 R_0$ (since $R-R_0\geq R/2$). The same bound for ${\mathcal E}_{B_{R/2}}^{\rm Pot}(u)$ is obvious when $1\le R\le 2R_0$.

Finally, a standard covering and scaling argument allows to control the full energy in $B_R$ instead of $B_{R/2}$.
\end{proof}

We close this section with a useful lemma concerning the decay towards $\pm 1$ of solutions $u_\ep$ of $(-\Delta)^s u_\ep  +\ep^{-s}W'(u_\ep)=0$ for points $x$ away from  $\{|u_\ep| \le \frac 9{10}\}$.
It will be used in the proof of Proposition \ref{propnew}.
\begin{lem}\label{lemdecay}
Given $s_0>0$, let $n\ge 2$, $s\in (s_0,2]$,  $W(u)=\frac 1 4 (1-u^2)^2$, and $K$ satisfy \eqref{L0} and \eqref{L2}. Let $u_\ep :\R^n\rightarrow (-1,1)$ be a solution of $L_{K}u + \ep^{-s}W'(u)=0$ in~$B_{2r}$, with $r\ge \ep>0$, satisfying $1-|u_\ep| \le c_0$ in~$B_{2r}$, where $c_0$ is the constant in \eqref{WWW}.  

Then, 
\[
0\le 1-|u_\ep| \le C \Big(\frac{\ep}{r}\Big)^s \quad \mbox{in } B_r,
\] 
where $C$ depends only on $n$, $\lambda$, $\Lambda$, and $s_0$.
\end{lem}

\begin{proof}
By scaling ---i.e., replacing $u$ by $u(\ep\,\cdot\,)$--- it is enough to consider the case $\ep=1$.
As in the proof of Lemma \ref{tuttookintype2}, for the explicit quartic potential $W$ we may take $c_0 := 1 -1/\sqrt 2$ and suppose  $1-c_0\le u\le 1$ in $B_{2r}$. Hence,  \eqref{21huwgiuw}  holds and,
similarly as in  \eqref{wjhtiohweio},  the function $v: = 1-u$ satisfies
\begin{equation}
\label{wjthoihwhe}
L_K v \le -\tfrac{1}{2}v \quad \mbox{in }B_{2r}.
\end{equation}

Fix now some smooth radial ``bowl-type'' function $\xi$ satisfying $\chi_{\R^n \setminus B_1} \ge \xi \ge \chi_{\R^n \setminus B_2}$, and let  $\xi_r(x) := \xi (x/r)$.
By scaling we have  $|L_K \xi_r |  \le C_0 r^{-s}$ in $\R^n$.

Let now $w :=2 \xi_r + b$, with $b := 4C_0 r^{-s}$. We have $L_K  w \ge -2C_0 r^{-s}  =   - \frac 1 2  b  \ge - \frac 1 2  w$ in $B_{2r}$. Hence,  recalling \eqref{wjthoihwhe} and using  that $w\ge 2+b \ge 2\ge v$ in 
$\R^n\setminus B_{2r}$ we obtain ---by the maximum principle--- $w\ge v$ in $\R^n$. Hence using that $\xi_r\equiv 0$ in $B_r$ we have shown $b\ge v$ in $B_r$, and the lemma follows.
\end{proof}

\section{Density estimates}\label{sec-5} 

In this section we establish density estimates when $s\in(0,1)$. We need to restrict our attention to the case of the fractional Laplacian since a crucial ingredient in the proof is the monotonicity formula from \cite{CC2}. This monotonicity formula is known to hold only for equations involving  the fractional Laplacian, since its proof relies on the so-called Caffarelli-Silvestre extension~\cite{CafSi}, that we recall here below.

By the result in \cite{CafSi}, $u: \R^n \to (-1,1)$ is a solution of $(-\Delta)^{s/2}u+W'(u)=0$ in~$\R^n$, for $s\in (0,2)$,  if and only if $u(x)=U(x,0)$ where $U: \R^{n+1}_+\to (-1,1)$, defined in the half-space $\R^{n+1}_+:= \{(x,y) \, : \, x\in \R^n, y>0\}$, is a  solution of 
\begin{equation}\label{ext}
\begin{cases}
\mbox{div}(y^{1-s}\nabla U(x,y))=0 & \mbox{in}\;\;\R^{n+1}_+\\
-d_s \lim_{y\rightarrow 0}y^{1-s}\partial_y U(x,y)= -W'(u(x))& \mbox{on}\;\;\R^{n}
\end{cases}
\end{equation}
and $d_s=2^{s-1}\Gamma(\frac{s}{2})/\Gamma(1-\frac{s}{2})>0$. 
Here $W$ is any $C^3$ potential and, for simplicity of notation, $\nabla$ denotes the full gradient ---i.e., with respect to the $(x,y)$ variables--- in~$\R^{n+1}_+$.

\begin{defi}\label{defextension}
Given a bounded function $u$ defined on $\R^n$, we will call the unique bounded extension $U$ of $u$ in $\R^{n+1}_+$ satisfying $\mbox{div}(y^{1-s}\nabla U(x,y))=0$, the \textit{$s$-extension} of $u$ (the existence and properties of this extension were studied in \cite{CafSi}; see also~\cite{C-Si1}).
\end{defi}

Let $\widetilde B^+_r((x_0,0))=\{(x,y)\in \R^{n+1}_+\, :\, |(x,y)-(x_0,0)|<r\}$ and $\widetilde B^+_r=\widetilde B^+_r((0,0))$ ---the tilde is used to distinguish balls in $\R^{n+1}$ from balls in $\R^n$. The energy associated to problem \eqref{ext} in a half-ball $\widetilde B_R^+$ is given by
\begin{equation}\label{energy-ext}
\widetilde{\mathcal E}_{R}(U)=\frac{d_s}{2}\int_{\widetilde B_R^+}y^{1-s} |\nabla U(x,y)|^2 \,dx\,dy + \int_{B_R}W(U(x,0))\,dx.
\end{equation}
As before, we distinguish between the Sobolev and the potential energies using the following notations:
\begin{equation*}
\widetilde{\mathcal E}_{R}^{\rm Sob}(U)=\frac{d_s}{2}\int_{\widetilde B_R^+}y^{1-s} |\nabla U(x,y)|^2 \,dx\,dy \quad \mbox{and} \quad \widetilde{\mathcal E}_{R}^{\rm Pot}(U)=\int_{B_R}W(U(x,0))\,dx.
\end{equation*}

We recall now a result from \cite{CRS} which allows to control the Sobolev energy of the $s$-extension $U$ of $u$ in $\R^{n+1}_+$ by the local contribution of the $H^{s/2}$-seminorm of~$u$ itself. We write the result for the specific case of half-balls in $\R^{n+1}$, since this is what we will need
later on. The result is true for a general Lipschitz domain $\Omega$, as one can see in Proposition 7.1 of \cite{CRS}.

\begin{lem}[Proposition 7.1 in \cite{CRS}]\label{up-down}
Let $s\in (0,2)$, $\widetilde B^+_{\rho}((x_0,0))$ be, as before, a half-ball in $\R^{n+1}_+$ centered at $(x_0,0)$ and with radius $\rho$, and let $U$ be the $s$-extension of $u$.

Then,
$$\int_{\widetilde B^+_{\rho}((x_0,0))} y^{1-s} |\nabla U(x,y)|^2\,dx\,dy \leq C \iint_{(\R^n\times \R^n )\setminus (B_{2\rho}^c(x_0)\times B_{2\rho}^c(x_0))} \frac{|u(x)-u(\bar x)|^2}{|x-\bar x|^{n+s}}\,dx\,d\bar x
$$
for some constant $C$ which depends only on $n$ and $s$.
\end{lem}

We also recall the monotonicity formula established in \cite{CC2}.

\begin{prop}[Proposition 3.2 in \cite{CC2}]\label{mon}
Let $s\in (0,2)$, $W\in C^3([-1,1])$ be nonnegative, and $U:\R^{n+1}_+\to (-1,1)$ be a solution of \eqref{ext}.

Then, 
$$\Phi(R):=\frac{1}{R^{n-s}} \widetilde{\mathcal E}_{R}(U)$$
is a nondecreasing function of $R>0$.
\end{prop}

The main ingredients in the proof of our density estimate are the previous monotonicity formula and the $BV$ estimate for stable solutions established in Theorem~\ref{BV}.

Before giving the proof of Proposition \ref{density}, we
state the following easy lemma that allows to interpolate between $L^1$ and $W^{1,1}$.

\begin{lem}[Theorem 1 in \cite{Brezis1}]\label{interpo} Let $s\in (0,1)$ and $R>0$.
Let $u$ be a $C^1$ function in $B_R\subset \R^n$ satisfying
$$
R^{-n}\int_{B_R} |u + k|\,dx\leq V \quad \mbox{and}\quad R^{1-n}\int_{B_R} |\nabla u|\,dx\leq P
$$
for some $k\in \R$ and constants $V$ and $P$.

Then,
$$
R^{s-n}\int_{B_R}\int_{B_R}\frac{|u(x)-u(\bar x)|}{|x-\bar x|^{n+s}} \,dx\,d\bar x \leq C V^{1-s}P^{s}
$$
for some constant $C$ which depends only on $n$ and $s$.
\end{lem}

\begin{proof}
It follows, after scaling, from  \cite[Theorem 1]{Brezis1} used with $p_1=p_2=p=1$, $s_1=0<1=s_2$, $\theta=1-s$, and $\Omega = B_1$.
\end{proof}

In the proof of Proposition \ref{density}, we will apply Lemma \ref{interpo} above with $k=\pm 1$, $V=\omega_0$, where $\omega_0$ is given in \eqref{hp-density}, and $P=C(1-s)^{-1}$, where $C(1-s)^{-1}$ comes from the $BV$ bound \eqref{BVest}.

We can now give the proof of our density estimate.

\begin{proof}[Proof of Proposition \ref{density}]

We argue by contradiction. Assume that there exists $\bar c\in (0,1)$ for which $R^{-n}\int_{B_R} |1+u| \,dx \le \omega_0$  and $\{u\ge -\bar c\}\cap B_{R/2}\neq \varnothing$.

Throughout the proof, all constants will depend only on $\bar c$, $n$, and $s$.

First, by continuity of $u$ and by taking $\omega_0<R^{-n} |B_{R/2}|$, there will be a point
$x_0\in B_{R/2}$  for which $|u(x_0)|\le \bar c$.
Moreover, by the uniform  continuity of $u$ (recall that $|\nabla u| \le C$ in $\R^n$; see Appendix~\ref{app-C}), we will have that $|u|\leq \frac{1+\bar c}{2}$ in  some ball of radius $r>0$ centered at $x_0$ (we emphasize that we can take $r$, as well as all other constants in the rest of the proof, to depend only on $\bar c$, $n$, and $s$). Using that $W(u)=\frac{1}{4}(1-u^2)^2$, we deduce
$$r^{s-n}\int_{B_r(x_0)} W(u) \,dx\geq \theta>0$$
for some positive constant $\theta$.

Let now $U$ be the $s$-extension of $u$ in $\R^{n+1}_+$. The previous lower bound on the potential energy in $B_r(x_0)$ leads to
$$r^{s-n}\widetilde{\mathcal E}_{r}(U)=r^{s-n}\left(\frac{d_s}{2}\int_{\widetilde B^+_r(x_0)}y^{1-s}|\nabla U|^2\,dx\,dy + \int_{B_r(x_0)} W(u)\,dx\right) \geq \theta,$$
where, for simplicity of notation, we keep denoting by $\widetilde{\mathcal E_r}$ the energy in the half-ball $\widetilde B^+_r((x_0,0))$ centered at $(x_0,0)$ (instead of at the origin).

Applying the monotonicity formula  of Proposition \ref{mon}, we deduce that
\begin{equation}\label{upper-bound}
\rho^{s-n} \widetilde{\mathcal E}_{\rho}(U)\geq \theta\quad \mbox{for every } \rho\ge r.
\end{equation}

We use now Lemma \ref{up-down} to translate the bound in \eqref{upper-bound} for the energy of the extension $U$ into a lower bound for the energy of $u$. We get

\begin{eqnarray}\label{bound2}
\theta&\leq &\rho^{s-n} \widetilde{\mathcal E}_{\rho}(U) \nonumber \\
& \leq &C\rho^{s-n} \left(\iint_{(\R^n\times \R^n )\setminus ( B^c_{2\rho}(x_0)\times B^c_{2\rho}(x_0))} \frac{|u(x)-u(\bar x)|^2}{|x-\bar x|^{n+s}}\,dx\,d\bar x +\int_{B_{\rho}(x_0)} W(u)\,dx  \right) \nonumber\\
&\leq &C \rho^{s-n} \iint_{B_{2\rho}(x_0)\times \R^n} \frac{|u(x)-u(\bar x)|^2}{|x-\bar x|^{n+s}}\,dx\,d\bar x ,
\end{eqnarray}
where in the last inequality we have used Proposition \ref{DircontrolsPot2} to bound the potential energy by the Sobolev energy in the larger ball, which requires to take $\rho\ge R_0$ (with $R_0$ being the constant in that proposition) since, then, $\rho+R_0\leq 2\rho$.

We aim now to use assumption \eqref{hp-density} and Lemma \ref{interpo} in order to find a contradiction with \eqref{bound2}.  To this end, we need to introduce a larger radius $R$, with $R\ge 4\rho$, and observe that the set of integration in \eqref{bound2} satisfies
\[
B_{2\rho}(x_0)\times \R^n
\subset \big(B_R(x_0)\times B_R(x_0)\big)\cup \big(B_{2\rho}(x_0)\times B_R^c(x_0)\big).
\]
Now, applying Lemma \ref{interpo}, we are able to control the $W^{1,s}$-seminorm of $u$ in $B_R$ by a positive power of $\omega_0$. More precisely, by \eqref{hp-density}, the quantity $R^{-n}\|u+1\|_{L^1(B_R)}$ is controlled by $\omega_0$. On the other hand, our $BV$ estimate \eqref{BVest} gives a bound for $R^{1-n}\|\nabla u\|_{L^1(B_R)}$ by a constant  $C(1-s)^{-1}$. Thus, using Lemma \ref{interpo}, $|u|\le 1$, and $R-2\rho\ge R/2$ (since we take $R\ge 4\rho$), we deduce
\begin{eqnarray}\label{bound3}
\theta&\leq &\rho^{s-n} \widetilde{\mathcal E}_{\rho}(U) \leq C\left(\frac{\rho}{R}\right)^{s-n}\,R^{s-n}\iint_{B_{R}(x_0)\times B_R(x_0)}\frac{|u(x)-u(\bar x)|}{|x-\bar x|^{n+s}}\,dx\,d\bar x   \nonumber \\
& & \hspace{3cm}+ C\rho^{s-n}\iint_{B_{2\rho}(x_0)\times B_R^c(x_0)} \frac{|u(x)-u(\bar x)|}{|x-\bar x|^{n+s}}\,dx\,d\bar x  \nonumber \\
&\leq &C \left(\frac{\rho}{R}\right)^{s-n} \omega_0^{1-s} + C\rho^{s-n}\iint_{B_{2\rho}(x_0)\times B_R^c(x_0)} \frac{dx\,d\bar x}{|x-\bar x|^{n+s}} \nonumber \\
&\leq &C \left(\frac{\rho}{R}\right)^{s-n} \omega_0^{1-s} + C\rho^{s}\int_{\frac{R}{2}}^{+\infty}\frac{r^{n-1}}{r^{n+s}}\,dr \nonumber \\
&\leq & C_1 \left[\left(\frac{\rho}{R}\right)^{s-n}\omega_0^{1-s} +\left(\frac{\rho}{R}\right)^s\right].
\end{eqnarray}

We now take $R$ and $\rho$ such that $\frac{\rho}{R}=(\frac{\theta}{4C_1})^{1/s}$. We may ensure $R\ge 4\rho$ (as required before) by making larger the constant $C_1$ in \eqref{bound3}, if necessary. Since we needed $\rho\ge \max \{r,R_0\}$  within the proof, with $R_0$ being the radius from Proposition~\ref{DircontrolsPot2}, this gives a lower bound for $R$ which becomes our final choice of radius $R_0$ in the statement of Proposition \ref{density}. 

Finally, with this choice of $\rho/R$, \eqref{bound3} becomes 
\begin{equation}\label{final2}\theta \leq C_1 \left(\frac{\theta}{4C_1}\right)^{\frac{s-n}{s}} \omega_0^{1-s} +  \frac{\theta}{4}.\end{equation}
Therefore, choosing $\omega_0$ small enough, we obtain a contradiction and conclude the proof.
\end{proof}

\section{Convergence results}\label{sec-6} 
The goal of this section is to prove the following convergence result, which will allow us to give the Proof of Theorem \ref{thm2}.
\begin{prop}\label{conv-L^1}
Let $n\ge 2$, $s\in(0,1)$, and $W(u)=\frac 1 4 (1-u^2)^2$. 
Let $u:\R^n\rightarrow (-1,1)$ be a stable solution of $(-\Delta)^{s/2} u +W'(u)=0$ in $\R^n$.

Then, for every sequence $R_j\uparrow \infty$ there exists a subsequence $R_{j_k}$ such that, defining  $u_{R} (x):= u(R x)$, we have 
\[
u_{R_{j_k}} \rightarrow  u_\infty: = \chi_\Sigma-\chi_{ \Sigma^c} \quad \mbox{ in } L^1_{\rm loc}(\R^n)\]
for some cone  $\Sigma\subset \R^n$ which is nontrivial $($i.e., it is not equivalent to $\R^n$ nor to $\varnothing$ up to sets of zero measure$)$ and which is a weakly stable set  in $\R^n$   for the fractional perimeter~$P_s$.

Moreover, we have the following convergence of the localized energies:
\begin{equation}\label{whoiwhoiehr}
\mathcal E^{\rm Sob}_{B_{R'}} (u_{R_{j_k}}) \to \mathcal E^{\rm Sob}_{B_{R'}}(u_\infty)  \quad \mbox{and} \quad  R_{j_k}^s \int_{B_{R'}} W(u_{R_{j_k}})\,dx \to 0= \int_{B_{R'}} W(u_\infty)\,dx,
 \end{equation}
as $k\uparrow\infty$, in every ball $B_{R'}\subset \R^n$.
\end{prop}

\begin{proof}We divide the proof into two steps.

\vspace{3pt}
-{\em Step 1.}  We start by proving that some subsequence of $u_{R_{j}}$ converges  in $L^1_{\rm loc}(\R^n)$ to $u_\infty : = \chi_\Sigma-\chi_{ \Sigma^c}$, for some nontrivial cone $\Sigma$, and that the convergence of energies \eqref{whoiwhoiehr} holds.

Throughout the proof, $C$ will denote (possibly different) positive constants which depend only on $n$ and $s$. 

Let us take a radius $R' \geq 1$. Theorem \ref{BV} and Corollary \ref{energy-est} yield that
\[
\int_{B_R} |\nabla u|\,dx \le C R^{n-1}\quad \mbox{and}\quad  \mathcal E_{B_R} (u) \le C R^{n-s} \quad \mbox{for all }R\ge1.
\]
Thus, for $R\ge 1$, by rescaling we deduce
\begin{equation}\label{W11uR}
\int_{B_{R'}} |\nabla u_R|\,dx \le C_{R'}  \quad \mbox{and} \quad  \mathcal E^{\rm Sob}_{B_{R'}} (u_R)  \le C_{R'},
\end{equation}
where $C_{R'}$ denote  constants which depend only on $R'$, $n$, and $s$.

Let $U$ be, as in the previous section, the $s$-extension of $u$ in $\R^{n+1}_+$.
By Lemma \ref{up-down}, and since the potential energy is nonnegative, we have 
\begin{equation*}
\widetilde {\mathcal  E}^{\rm {Sob}}_{R} (U)
\leq C\iint_{(\R^n\times \R^n)\setminus( B^c_{2R}\times B^c_{2R})}  \frac{|u(x)-u(\bar x)|^2}{|x-\bar x|^{n+s}}\,dx\,d\bar x= 4C\, \mathcal E^{\rm {Sob}}_{B_{2R}} (u) \le C R^{n-s}
\end{equation*}
for $R\ge 1$. In addition, by the monotonicity formula of Proposition \ref{mon}, the quantity
\[\Phi(R) = R^{s-n} \widetilde{\mathcal E}_R(U)\]
is monotone nondecreasing. At the same time, by the previous bound on the Sobolev energy and by Proposition \ref{DircontrolsPot2} (which gives control on the potential energy in $B_R$ by the Sobolev energy in $B_{2R}$ if we take $R+R_0\leq 2R$), we deduce that $\Phi$ is bounded above by a finite constant.
We deduce that 
\begin{equation}\label{mono-lim}
\Phi(R'R)-\Phi(\bar R) \rightarrow 0 \quad \text{ as } R'R\ge \bar R \uparrow \infty.
\end{equation}

Now, from the proof of the monotonicity formula, which is based on a Pohozaev identity ---see the proof of Proposition 3.2 in \cite{CC2}--- we have that
\begin{equation*}
\Phi'(\rho) = \frac{d_s}{\rho^{n-s}} \int_{\partial ^+ \widetilde B_\rho^+} y^{1-s} (\partial_{r} U)^2 \,d{\mathcal H}^n(x,y) + \frac{s}{\rho^{n-s+1}}\int_{B_\rho} W(u) \,dx,
\end{equation*}
where $\partial^+$ denotes the part of the boundary contained in the open half-space $\{y>0\}$ and $\partial_r$ denotes the radial derivative in $\R^{n+1}_+$.
After rescaling, this becomes
\begin{equation}\label{wnoiwoieht}
\Phi'(R\tilde\rho) = \frac{d_s}{R\tilde\rho^{n-s}} \int_{\partial^+ \widetilde B_{\tilde\rho}^+} \tilde y^{1-s} (\partial_{r} U_R)^2 \,d{\mathcal H}^n(\tilde x,\tilde y) + \frac{s}{R^{1-s}\,{\tilde\rho}^{n-s+1}}\int_{B_{\tilde\rho}} W(u_R) \,d\tilde x,
\end{equation}
where  $U_R$ is the $s$-extension of $u_R$.
We integrate now with respect to $ \rho$ to obtain, for $R'/2\geq \bar R/R$,  
\begin{equation}\label{wnoiwoieht2}
\begin{split}
\Phi(R' R)-\Phi(\bar R) &=\int_{\bar R}^{R' R} \Phi'(\rho)\,d\rho=R\int_{\bar R/R}^{R'}\Phi'(R{\tilde\rho})\,d{\tilde\rho}\\
&=
d_s \int_{\bar R/R}^{R'} d{\tilde\rho} \, {\tilde\rho}^{s-n} \int_{\partial ^+ \widetilde B_{\tilde\rho}^+} \tilde y^{1-s} (\partial_{r} U_R)^2 \,d{\mathcal H}^n(\tilde x,\tilde y)\\
&\hspace{3em} + sR^{s}\int_{\bar R/R}^{R'} d{\tilde\rho}\,{\tilde\rho}^{s-n-1}\int_{B_{\tilde\rho}}W(u_R)\,d\tilde x\\
&\ge d_s (R')^{s-n} \int_{\widetilde B^+_{R'}\setminus \widetilde B^+_{\bar R/R}} \tilde y^{1-s} (\partial_{r} U_R)^2 \,d  \tilde x\,d\tilde y,\\
&\hspace{3em} { + sR^{s}(R')^{s-n-1} \frac{R'}{2} \int_{B_{R'/2}} W(u_R)\,d\tilde x  }\quad\text{ if } R\geq 2\bar R/R'.
\end{split}
\end{equation}

Note that, given $R'$ and $\bar R$, we have  
$$
 \int_{\widetilde B^+_{\bar R/R}} \tilde y^{1-s} (\partial_{r} U_R)^2 \,d  \tilde x\,d\tilde y \leq C R^2 \int_{\widetilde B^+_{\bar R/R}} \tilde y^{1-s} \,d  \tilde x\,d\tilde y\leq C \bar R^{2-s+n}R^{s-n} \to 0 \quad\text{ as }R\uparrow\infty.
$$ 
This together with \eqref{mono-lim} and \eqref{wnoiwoieht2}, leads to
\begin{equation}\label{limithomog}
\int_{\widetilde B^+_{R'}} y^{1-s} (\partial_{r} U_R)^2 \,d\tilde x\,d\tilde y \rightarrow 0  \quad \mbox{as }R\uparrow \infty
\end{equation}
and
\begin{equation}\label{limitpm1}
R^s \int_{B_{R'/2}} W(u_{R}) \,d\tilde x\rightarrow 0  \quad \mbox{as }R\uparrow \infty
\end{equation}
for every $R'\ge 1$.

Next, choose any $\sigma \in(s,1)$ and let $N\ge 1$ be an integer.
By the $W^{1,1}$ estimate \eqref{W11uR} applied with $R'=N$, and since $|u|\le 1$, Lemma~\ref{interpo} (applied with $s$ replaced by $\sigma$) leads to
\[   \int_{B_{N}\times B_{N}} \frac{\big|u_R(x)-u_R(\bar x)\big|^2}{|x-\bar x|^{n+\sigma}} \,dx\,d\bar x\leq C_{N,\sigma}\]
for all $R\ge 1$,
where $C_{N,\sigma}$ is a constant depending only on $N$, $\sigma$, $n$, and $s$.
Hence, using that $|u_R|\le 1$ and the compactness of $W^{\sigma/2,2}$ inside $W^{s/2,2}$,  there exists a subsequence $u_{R_{j_k}}$, that we still denote by $u_{R_k}$, and a function $u_\infty$ such that
\begin{equation}\label{strong}
\big\| u_{R_{k}} -u_\infty\big\|_{L^2(B_{N})}  + \int_{B_{N}\times B_{N}} \frac{\big| (u_{R_{k}} -u_\infty)(x)- (u_{R_{k}}-u_\infty)(\bar x)\big|^2}{|x-\bar x|^{n+s}}   \,dx\,d\bar x\rightarrow  0
\end{equation}
as $k\uparrow\infty$.
In addition, letting $N\uparrow\infty$, taking further subsequences, and using a Cantor diagonal argument, we obtain a new subsequence converging in $L^2$ in every ball of $\R^n$.
 
Now, given $R'\ge 1$, we use once more Lemma \ref{up-down} (now applied to $U_{R_k}-U_\infty$ in $\widetilde B_{R'}^+$,  where $U_\infty$ denotes the $s$-extension of $u_\infty$) to control 
$$ 
\int_{\widetilde B_{R'}^+}  y^{1-s}\big| \nabla U_{R_k}- \nabla U_\infty\big|^2\,dx\,dy
$$
by the double integral in \eqref{strong} computed now in $B_{2R'}\times \R^n$. Taking $N>2R'$, since \eqref{strong} gives control on the integrals computed over $B_{2R'}\times B_N$, it only remains to make the double integral on $B_{2R'}\times (\R^n\setminus B_N)$ arbitrary small. Such bound is obvious, since $|u_{R_{k}} -u_\infty| \le2$, by taking $N$ large enough. Therefore, we conclude
\begin{equation}\label{hwiohrhee1}
 \int_{\widetilde B_{R'}^+}  y^{1-s}\big| \nabla U_{R_k}- \nabla U_\infty\big|^2\,dx\,dy  \rightarrow  0
\end{equation}
as $k\uparrow\infty$.

From this strong convergence and the local uniform convergence of $u_{R_k}$, passing to the limit in \eqref{limithomog} and \eqref{limitpm1} we obtain
\[  \int_{\widetilde B_{R'/2}^+}y^{1-s} (\partial_r U_\infty)^2 \,dx\,dy  = 0 \quad \mbox{and}\quad \int_{B_{R'/2}}W(u_\infty)\,dx =0\]
for every $R'\ge 1$. Therefore, $\partial_r U_\infty = 0$ in $\R^{n+1}_+$ and $W(u_\infty)= 0$ in $\R^n$.
In other words, $U_\infty$ and  its trace $u_\infty$ are homogeneous of degree $0$, and $u_\infty$ takes
values $\pm 1$ ---the two wells of $W$. Equivalently, we have that
\[u_\infty=\chi_\Sigma-\chi_{ \Sigma^c}\quad \mbox{in }\R^n\]
 for some cone $\Sigma$. In addition, by the same convergences for $u_{R_k}$ that we have just used, we see that \eqref{whoiwhoiehr} holds. 
 
Finally, using  the monotonicity of $\Phi$, \eqref{hwiohrhee1}, and \eqref{limitpm1}, we obtain 
 \begin{equation*}
\begin{split}
0&<\Phi(1)\le \Phi(R)=R^{s-n}\widetilde{\mathcal E}_R(U)=\widetilde{\mathcal E}^{\rm {Sob}}_1(U_R)+R^s\int_{B_1} W(u_R) \,dx)\\
&\le \lim_{R\uparrow \infty}\Big(\widetilde{\mathcal E}^{\rm {Sob}}_1(U_R)+R^s\int_{B_1} W(u_R) \,dx \Big)\\
&=\widetilde{\mathcal E}^{\rm {Sob}}_1(U_\infty).
\end{split}
 \end{equation*}
Thus ${U_\infty}$ and ${u_\infty}$ have positive energy,  and hence $\Sigma$ is nontrivial (i.e. it is not equal to $\R^n$ nor $\varnothing$ up to sets of measure zero).

\vspace{3pt}

-{\em Step 2}. It remains to prove that $\Sigma$ is a weakly stable set  in $\R^n$ for the fractional perimeter~$P_s$.

To this end, let us take an arbitrary  smooth vector field $X=X(x,t)$ which is compactly supported in $B_1\times(-1,1)$ (by scaling, since $\Sigma$ is a cone, we may take the support to be the unit ball) and let  $\Psi = \Psi_t(x)$ denote the map $(x,t) \mapsto \phi_X^t(x)$, where $\phi_X^t$ is the integral flow defined by the ODE
\[
\frac{d}{dt} \phi_X^t(x) = X(\phi_X^t(x),t) \quad \mbox{with initial condition }\quad   \phi_X^0(x) = x.
\]
Note that $\Psi_0= {\rm Id}$ in $\R^n$, $\Psi_t = {\rm Id}$ outside of $B_1$, and that $\Psi_t:\R^n\to\R^n$ is a diffeomorphism for $|t|$ small.

Let us introduce the rescaled energy functional (of which $u_R$ is a stable critical point)
\[
\mathcal E^R_{B_1} (v) := \frac{1}{4} \iint_{(\R^n\times \R^n) \setminus (B_1^c\times B_1^c)}\, \frac{| v(y)-v(\bar y) |^2}{
 | y-\bar y|^{n+s}} \,dy\, d\bar y + R^s\int_{B_1} W(v(y))\,dy \, ,\]

We first show that the function $u_{R,t} := u_R\circ \Psi_t^{-1}$ satisfies
\begin{equation}\label{theexpansion}
   \mathcal E^R_{B_1} (u_{R,t}) \ge  \mathcal E^R_{B_1} (u_R) - C_{\Psi} t^3, \quad \mbox{for }t\in(-T_{\Psi}, T_{\Psi}),
\end{equation}
where $T_{\Psi}>0$ and $C_{\Psi}$ will be, from now on, different positive constants which depend only on $X$, $n$, and $s$ ---in particular, they are independent of~$R$.

To prove this, and as in the proof of Lemma \ref{lem2A}, we make the change of variables $y= \Psi^{-1}_{ t}(x)$, $\bar y= \Psi^{-1}_{t}(\bar x)$ for $|t|$ small. Since $\Psi_{t}^{-1}$ sends $B_1$ and $B_1^c$ onto themselves, setting $A_1:=(\R^n\times \R^n)\setminus (B_1^c\times  B_1^c)$ we have
\begin{equation}\label{Eg-urt}
\mathcal E^R_{B_1}(u_{R,t}) =  \frac{1}{4}\iint_{A_1}\, \frac{| u_R(y)-u_R(\bar y) |^2}{
 \big| \Psi_{t}(y)-\Psi_{t}(\bar y)\big|^{n+s}} \,J_t(y) \,J_t(\bar y)\,dy \,d\bar y + R^{-s} \negthinspace\int_{B_1} W(u_R(y))\,J_t(y)\,dy,
\end{equation}
where $J_{t}$  is the Jacobian  ${\rm det} D\Psi_t$.

A Taylor expansion for $J_t$ yields
\begin{equation} \label{expansionJac}
\big| J_t(y) - 1  - h_1(y) t  - h_2(y) t^2 \big| \le C_{\Psi}t^3  \quad \mbox{for }t\in(-T_{\Psi}, T_{\Psi}),
\end{equation}
with $ \| h_1\|_{L^\infty(\R^n)} +\| h_2\|_{L^\infty(\R^n)}  \le C_{\Psi}$.
Similarly, we have
\[
\left|\, \Psi_{t}(y) - \Psi_{t}(\bar y) - ( y-\bar y)  -  |y-\bar y|\big( { g} _1(y,\bar y) t +  { g}_2(y,\bar y) t^2 \big)\,\right| \le C_{\Psi}  |y-\bar y| t^3
\]
for $t\in(-T_{\Psi}, T_{\Psi}) $, where
$\| { g}_1\|_{L^\infty (\R^n\times\R^n )} +\|  g_2\|_{L^\infty (\R^n\times\R^n )}  \le C_{\Psi}$.

Therefore
\begin{equation} \label{expansionK}
\begin{split}
\big| \Psi_{t}(y)-\Psi_{t}(\bar y)\big|^{-n-s}
&= \big| y-\bar y  + |y-\bar y|\,\big(  { g}_1(y,\bar y) t + { g}_2(y,\bar y) t^2 +O(t^3)\big)\big|^{-n-s}
\\
&\hspace{-2cm}= |y-\bar y|^{-n-s} \big( 1+ tk_1(y,\bar y) + t^2 k_2(y,\bar y) + O(t^3) \big),
\end{split}
\end{equation}
where $\| k_1\|_{L^\infty (\R^n\times\R^n )} +\| k_2\|_{L^\infty (\R^n\times\R^n )}  \le C_{\Psi}$.

Using  \eqref{expansionJac} and \eqref{expansionK} in \eqref{Eg-urt} we obtain 
\[
\big| \mathcal E^R_{B_1}(u_{R,t})  - \big(  \mathcal E^R_{B_1}(u_{R})  + a_1t + a_2 t^2 \big) \big|  \le C_{\Psi} \,t^3 \mathcal E^R_{B_1}(u_{R})
 \]
for some constants $a_1$ and $a_2$, since the quantity
\[t^3\mathcal E^R_{B_1}(u_{R}) =t^3\left( \frac{1}{4} \iint_{A_1}\, \frac{| u_R(y)-u_R(\bar y) |^2}{
 | y-\bar y|^{n+s}} \,dy\, d\bar y + R^s\int_{B_1} W(u_R(y))\,dy \right)\]
controls the error terms which are cubic in the variable $t$.
Now, since by assumption $u_R$ is a stable solution, it must be $a_1 =0$ and $a_2\ge 0$, and thus
\[
   \mathcal E^R_{B_1} (u_{R,t}) \ge  \mathcal E^R_{B_1} (u_R) - C_{\Psi}  \mathcal E^R_{B_1}(u_{R}) t^3, \quad \mbox{for }t\in(-T_{\Psi}, T_{\Psi}).
\]
Finally, thanks to \eqref{W11uR} and \eqref{limitpm1} we have $\mathcal E^R_{B_1} (u_R) \le C$, with $C$ depending only on $n$ and $s$. This concludes the proof of \eqref{theexpansion}.

Now, recalling the convergence of the energies, \eqref{whoiwhoiehr}, we have that \eqref{theexpansion} passes to the limit and we deduce
\[
   \mathcal E^{\rm Sob}_{B_1} (u_{\infty,t}) \ge  \mathcal E^{\rm Sob}_{B_1} (u_{\infty}) - C_{\Psi}  t^3 \quad \mbox{ for }t\in(-T_{\Psi}, T_{\Psi}),
\]
where $u_{\infty,t} = u_{\infty}\circ \Psi^{-1}_t$.

Since $u_\infty = \chi_\Sigma-\chi_{ \Sigma^c}$,  we have
\[
\mathcal E^{\rm Sob}_{B_1} (u_\infty) = 2 P_s(\Sigma,B_1)  \quad \mbox{and}\quad  \mathcal E^{\rm Sob}_{B_1} (u_{\infty,t}) =2P_s(\Psi_t(\Sigma), B_1).
\]
Thus,
\[
   P_s(\Psi_t(\Sigma),B_1) \ge  P_s(\Sigma, B_1)  - C_{\Psi}  t^3, \quad \mbox{for }t\in(-T_{\Psi}, T_{\Psi}).
\]

Recalling Definition \ref{def-weakly} and since the smooth compactly supported vector field~$X$ defining $\Psi$ was arbitrarily, we have shown that  $\Sigma$ is a weakly stable set in $B_1$  for the fractional perimeter $P_s$. Finally, using that $\Sigma$ is a cone, we easily deduce, by scaling, that it is in fact weakly stable in all of $\R^n$.
\end{proof}

We can now give the
\begin{proof} [Proof of Theorem \ref{thm2}] 
The first part of the statement on the $L^1$-convergence has just been proven in Proposition \ref{conv-L^1} above. We have also proved that $\Sigma$ is nontrivial. 

The last part of the statement, i.e., that  \eqref{wngioewhtioeh1} and \eqref{wngioewhtioeh2}  hold after choosing the representative of  $\Sigma$ (for which every point of $\Sigma$ with density $1$ belongs to its interior and every point of density 0 belongs to its complement) follows, as usual, from the local  $L^1$-convergence and the density estimate of Proposition \ref{density}.
\end{proof}

In  Proposition \ref{conv-L^1} we showed that the potential energies of sequences of blow-downs converge to zero. 
This was a consequence of the monotonicity formula.
To end this section we now give a stronger property (which will be useful in Section~\ref{sec-7}): a quantitative convergence of the potential energy to zero, as $\ep\downarrow 0$, for stable solutions of   $(-\Delta)^{s/2}u_\ep + \ep^{-s} W'(u_\ep)=0$ with $s\in (0,1)$.
\begin{prop}  \label{propnew}
Let  $n\ge 2$, $s\in (0,1)$, and $W(u)=  \frac 1 4 (1-u^2)^2$. For $\ep>0$, let $u_\ep:\R^n\rightarrow (-1,1)$  be a stable solution of $(-\Delta)^{s/2}u_\ep + \ep^{-s} W'(u_\ep)=0$ in $\R^n$.

If $R\ge \ep$, then 
\[
\int_{B_{R}} (\ep/R)^{-s} W(u_\ep)\,dx  \le C R^n (\ep/R)^{\beta},
\]
where  $\beta := \min\big(\frac{1-s}{2}, s \big) >0$ and $C$ depends only on $n$ and $s$.
\end{prop}

\begin{proof}
By scaling we may  (and do) assume without loss of generality that $R=1$.  

Given $x\in B_1$, let 
\[ 
r_x : = \max ( \min\big( \tfrac 1 8 , \tfrac 1 2 {\rm dist} (x, \{|u|\le \tfrac {9}{10}\}),  C_0\ep),
\]
where $C_0>0$ is a large enough constant, depending only on $n$ and $s$, to be chosen later. 
Note that if $\ep$ satisfies $1 \le 8C_0\ep$, then
\[
\int_{B_{1}} \ep^{-s} W(u_\ep)\,dx  \le (8C_0)^s  ({\textstyle \max_{[-1,1]} }W )   |B_{1}|   \le  C   \le C  C_0^{\beta}   \ep^{\beta}.
\]
Thus, we may (and do) assume that $\frac{1}{8}> C_0\ep$. In particular, $r_x\in [C_0\ep,\frac{1}{8}]$ for all $x\in B_1$.

We will take $C_0$ satisfying $C_0 \ge \max (R_0,1)$, where $R_0$ is the constant of Proposition \ref{density} for $\bar c = \tfrac{9}{10}$. Then, for the constant $\omega_0$ of Proposition \ref{density}, we claim that 
\begin{equation}\label{whtiohwioh}
\min\bigg( \int_{B_{4r_x} (x)}  |u_\ep -1|\,dx,  \int_{B_{4r_x} (x)}  |u_\ep +1|\,dx\bigg) \ge \omega_0 (2r_x)^n
\end{equation}
for all  $x\in B_1$ such that $r_x< \frac 1 8$.
Indeed,  since $r_x<\frac 1 8$, there exists $z\in \{|u|\le \tfrac {9}{10}\} \cap\overline{B}_{1+ \frac 1 4}$ such that $|x-z| \leq 2r_x$. Hence, by Proposition \ref{density}  ---applied to  $u = u_\ep (\ep \,\cdot\,)$---,
\[
\min\bigg( \int_{B_{2r_x} (z)}  |u_\ep -1|\,dx,  \int_{B_{2r_x} (z)}  |u_\ep +1|\,dx\bigg) \ge \omega_0 (2r_x)^n, 
\]
and \eqref{whtiohwioh} follows since  $B_{2r_x} (z)\subset  B_{4r_x} (x)$.

Note also that  $B_{4r_x} (x) \subset B_{3/2}$ for all $x\in B_1$.

On the other hand, thanks to the potential energy estimate in Theorem \ref{thm1} (rescaled) we have  (using $4r_x/\ep\ge C_0\ge 1$)
\begin{equation}\label{wngthwiohw}
C_0^s \ave_{B_{4r_x}}  \frac{1}{4} (1-|u_\ep|)^2 \,dx \le   (\ep/r_x)^{-s}  \ave_{B_{4r_x}}  W(u_\ep)\,dx \le C, 
\end{equation}
where $\ave$ denotes the average.

Let us next show that Poincar\'e's inequality, \eqref{whtiohwioh}, and \eqref{wngthwiohw} for $C_0$ sufficiently large, yield
\begin{equation} \label{whtoiwhoewih}
\int_{B_{4r_x} (x)} |\nabla u_\ep| \ge c  r_x^{n-1} \quad \mbox{whenever } r_x<\tfrac 1 8,
\end{equation}
for some constant $c>0$ depending only on $n$ and $s$.
Indeed, if $\int_{B_{4r_x} (x)} |\nabla u_\ep| =: \kappa \,  r_x^{n-1}$, then by Poincar\'e's inequality  we have 
\[
\ave_{B_{4r_x} (x)} |u_\ep-t| \le C\kappa \quad \mbox{for some }t\in [-1,1].
\]
But then  using  \eqref{wngthwiohw},  we have
\[
\begin{split}
|1-|t||^2 &= \ave_{B_{4r_x} (x)} |1-|t||^2 \le 2 \ave_{B_{4r_x}}  (1-|u_\ep|)^2  + 2\ave_{B_{4r_x}}  |u_\ep -t|^2
\\
& \le  CC_0^{-s}  + 4\ave_{B_{4r_x}}   |u_\ep -t| \le  C(C_0^{-s} + \kappa).
\end{split}
\]
Recalling now \eqref{whtiohwioh} we obtain
\[
\begin{split}
\frac{2^n \omega_0}{|B_4|} &\le \min\bigg( \ave_{B_{4r_x} (z)}  |u_\ep -1|\,dx,  \ave_{B_{4r_x} (z)}  |u_\ep +1|\,dx\bigg)
\\& \le  \ave_{B_{4r_x} (z)}  |u_\ep -t|\,dx + |1-|t|| 
\le C(\kappa + (C_0^{-s} + \kappa)^{1/2}),
\end{split}
\]
and this gives a lower bound for $\kappa$ provided the $C_0$ is chosen sufficiently large. Thus, \eqref{whtoiwhoewih} is now proved.

We now produce a covering of $B_1$, by some of the balls  $B_{r_x}(x)$,  as follows. 
Given $k \le -4$, let $X_k := \{x\in B_1 \,: \,r_x \in (2^k, 2^{k+1}]\}$ and let $\{ x_j^k\}_{j\in \mathcal J_k}$ be a  maximal subset of $X_k$  with the property that the balls $B_{\frac 1 4 r_{x_j^k}}(x_j^k)$ are disjoint. 
It then follows (using that all radii $r_x$ belong to $(2^k, 2^{k+1}]$ for $x\in X_k$) that 
$$
X_k \subset \bigcup_{j\in \mathcal J_k} B_{r_{x_j^k}}(x_j^k)
$$
and that the family of quadruple balls 
$$
\{ B_{4r_{x_j^k}}(x_j^k) \}_{j\in \mathcal J_k}
$$ 
has  (dimensional) finite overlapping.\footnote{That is, every point $x\in \R^n$ belongs to at most $N$ of these balls, with $N$ depending only on~$n$. This is easy to check: if $x\in \R^n$ belonged to $N$ of such balls, we would have the existence of $N$ points $x^k_j$ in $B_{4\cdot 2^{k+1}}(x)$ such that the balls $B_{\frac{1}{4} 2^k}(x^k_j)$ are disjoint and contained in $B_{9\cdot 2^{k}}(x)$.} 
Note also that, by construction,  the union of the sets $X_k$ when $k$ runs on  $\{\lfloor\log_2 (R_0\ep)\rfloor \le k \le -4\}$ covers all of $B_1$.

Now, on the one hand,  the BV estimate $\int_{B_{3/2}} |\nabla u_\ep| \,dx\le C$ (which follows from Theorem \ref{BV}) yields, for all $k \le -4$,
\begin{equation} \label{whtoiwhoewih2}
\#\mathcal J_k \le C(2^k)^{1-n}.
\end{equation}
Indeed, this follows using that the balls  $B_{ 4 r_{x_j^k}}(x_j^k)$  have finite overlapping and are contained in $B_{3/2}$: when $k<-4$  then $r_{x_j^k} <\frac 1 8$ and hence all  the balls satisfy \eqref{whtoiwhoewih} and are contained in $B_{3/2}$ by construction; while for $k=-4$ the radius of the balls is at least $\frac 1 {16}$ so their number must be bounded.

On the other hand, we claim that Lemma \ref{lemdecay} yields
\[
\int_{B_{r_x}(x)} \ep^{-s} W(u_\ep) \,dx\le \int_{B_{r_x}(x)} \ep^{-s} (1-|u_\ep|)^2 \,dx\le C \ep^{-s} \Big(\frac{\ep}{r_x}\Big)^\alpha r_x^n
\]
for any given $\alpha \in [0,2s]$. Indeed, note that if $r_x= C_0 \ep$ the previous estimate is trivial, while if $r_x> C_0\ep$ then $r_x\le \tfrac 1 2 {\rm dist} (x, \{|u|\le \tfrac {9}{10}\})$ and hence we may apply Lemma \ref{lemdecay}
(recall that $r_x\ge C_0\ep\ge \ep$).

Therefore, choosing $\alpha := \min\big( \frac{1+s}{2}, 2s\big)\in (0,1)$ we obtain ---using \eqref{whtoiwhoewih2}---
\[
\begin{split}
\int_{B_1} \ep^{-s} W(u_\ep)\,dx &\le C \sum_{k= \lfloor\log_2 (R_0\ep)\rfloor}^{-4}  \sum_{j\in\mathcal J_k} \int_{B_{r_{x^k_j}}(x^k_j)} \ep^{-s} W(u_\ep) \,dx
\\
& \le C  \sum_{k= \lfloor\log_2 (R_0\ep)\rfloor}^{-4}  \sum_{j\in\mathcal J_k} \ep^{-s} \Big(\frac{\ep}{r_{x_j}}\Big)^\alpha r_{x_j}^n
\\
& \le C \sum_{k= \lfloor\log_2 (R_0\ep)\rfloor}^{-4} \ep^{-s} \Big(\frac{\ep}{2^k}\Big)^\alpha (2^{k+1})^n\, \#\mathcal J_k
\\
& \le C \sum_{k= \lfloor \log_2 (R_0\ep)\rfloor}^{-4} \ep^{\alpha-s}  (2^k)^{n-\alpha} (2^k)^{1-n} \le C \ep^{\alpha-s}\sum_{k= -\infty}^{-4}  (2^k)^{1-\alpha} 
\\
& \le C\ep^{\beta}, 
\end{split}
\]
as we wanted to show.
\end{proof}

\section{Proofs of the classification results}\label{sec-7} 

In this section we give the proof of our classification results. In order to prove Theorem \ref{thmclas}, we will need some preliminary ingredients.

We start by recalling the main results in \cite{dPSV}, which are a consequence of an improvement of flatness theory for phase transitions in the ``genuinely nonlocal'' regime (meaning that the order~$s$ of the operator is less than 1). The first one will be used to conclude one-dimensionality of solutions.

\begin{thm}[Theorem 1.2 in \cite{dPSV}]\label{123}
Let $n\ge 2$, $s\in(0,1)$, and $W(u)=\frac 1 4 (1-u^2)^2$.  Let $u:\R^n\rightarrow (-1,1)$ be a  solution of $(-\Delta)^{s/2}u + W'(u)=0$ in $\R^n$.

Assume that there exists a function $a:(1,\infty) \rightarrow (0,1]$ such that  $a(R)\downarrow 0$ as $R\uparrow\infty$ and
such that, for all $R>0$, we have
\begin{equation}\label{ASS-R}
\{ e_R\cdot x\le -a(R)R\} \subset \big\{u\le -{\textstyle \frac 4 5} \big\}\subset  \big\{u\le {\textstyle \frac 4 5}\big \} \subset \{e_R\cdot x\le a(R)R\} \quad \mbox{in }B_{R}
\end{equation}
for some $e_R\in S^{n-1}$
which may depend on $R$.

Then, $u(x)=\phi(e\cdot x)$ for some direction $e\in S^{n-1}$ and an increasing function $\phi:\R\to (-1,1)$. \end{thm}

We next prove a corollary of Theorem 1.1 in \cite{dPSV} which will be useful in the sequel. It is an``iterated version'' of  Theorem 1.1 in \cite{dPSV} in the particular case $L= (-\Delta)^s$ and $f(u)= - W'(u) = u-u^3$. 
 
\begin{prop}\label{cor:improvement}
Let $s\in(0,1)$ and $n\ge2$. There exist constants~$\alpha_0\in (0,s/2)$, $
p_0\in(2,\infty)$, and~$a_0\in(0,1/4)$, depending only on $n$ and $s$, such that the following statement holds.
\smallskip

Let $a\in(0,a_0]$ and let 
\begin{equation}\label{ja}
j_a := \left\lfloor  \frac{\log a}{\log(2^{-\alpha_0})} \right\rfloor \in \mathbb N.
\end{equation}
Let $\eps>0$ and $k\in \mathbb N$  satisfy $ 2^{k-1}\ep \le \big(2^{-\alpha_0(k-1)}a \big)^{p_0}$  and let  $u_\ep: \R^n \rightarrow (-1,1)$ be a solution 
of 
\[
(-\Delta)^su_\eps + \varepsilon^{-s}W'(u_\eps) =0 \quad  \mbox{in  }B_{2^{j_a}} \subset \R^n
\]
satisfying $0\in \big\{-\frac 3 4 \le u \le  \frac 3 4 \big\}$ and
\begin{equation} \label{wethiowehtowh}
 \big\{\omega_j \cdot x\le -a 2^{j(1+\alpha_0)} \big\}\, \subset\, \big\{u\le  -{\textstyle \frac 3 4}\big\} \,\subset \,\big\{u\le  {\textstyle \frac 3 4}\big\}  \,\subset \, \big\{\omega_j\cdot x\le a 2^{j(1+\alpha_0)} \big\} \quad \mbox{in }B_{2^j}, 
\end{equation}
for $0\le j\le j_a$,  for some  $\omega_j\in S^{n-1}$ .

Then, for all $i= 1,2, \dots, k$ we have
\begin{equation} \label{hwoeithwoih}
\left\{\omega_{-i}\cdot x \le - 
\frac{a}{2^{(1+\alpha_0)i}} \right\} \,\subset \, \big\{u\le  -{\textstyle \frac 3 4} \big\} \,\subset\,  \big\{u\le  {\textstyle \frac 3 4} \big\}  \,\subset\, \left\{\omega_{-i}\cdot x\le  \frac{a}{2^{(1+\alpha_0)i}} \right\} \;\; \mbox{in }B_{2^{-i}},
\end{equation}
for certain $\omega_{-i}\in S^{n-1}$.
\end{prop}
\begin{proof}
The proof will apply inductively Theorem 1.1 of \cite{dPSV} to $u^{(i)}(x) : = u_\ep(2^{i-1}x)$. 
Indeed, recall first that ---see \eqref{21huwgiuw}--- we have $-W''(t) \in [-2,-1/2]$ for $|t|\ge \frac{1}{\sqrt 2}$. Since $\frac 3 4 \ge \frac 1 {\sqrt 2}$ we may take the constant $\kappa$ from  \cite{dPSV} equal to $1/4$.

Notice that the case $i=1$ of \eqref{hwoeithwoih} follows directly from Theorem 1.1 in \cite{dPSV} since $ 2^{k-1}\ep \le \big(2^{-\alpha_0(k-1)}a \big)^{p_0}$  and $k\ge 1$ guarantee $\ep\le a^{p_0}$. 

Assume now that \eqref{hwoeithwoih} holds for $ i = 1, 2, \dots, i_\circ-1$ and that $i_\circ\le k$.  Then $u^{(i_\circ)}(x)$ satisfies  the assumptions of Theorem 1.1 in \cite{dPSV} with 
 and $a$ replaced by $2^{-\alpha_0(i-1)}a$  and $\ep$ replaced by $2^{i-1}\ep$  (too see this it may be useful to notice that, by the definition of $j_a$ in \eqref{ja}, we have $ j_{2^{-\alpha_0(i-1)}a} = j_a+ (i-1)$), since we have
\[ 
2^{(i_\circ-1)}\ep \le 2^{k-1}\ep \le \big(2^{-\alpha_0(k-1)}a \big)^{p_0}\le  \big(2^{-\alpha_0(i_\circ-1)}a \big)^{p_0}.
\]

Hence, $u^{(i_\circ)}$ satisfies the conclusion  Theorem 1.1 in \cite{dPSV} so, after rescaling, we  obtain that \eqref{hwoeithwoih}  also holds also for $i = i_\circ$.
\end{proof}

The second result is an easy consequence of Proposition \ref{cor:improvement}: flatness implies a $C^{1,\alpha}$ type result.

\begin{thm}\label{thm63utbg}
Let $n\ge 2$, $s\in(0,1)$, and $W(u)=\frac 1 4 (1-u^2)^2$. Given $\tilde a>0$ there exist positive constants $\sigma_0$, $\delta_0$,  $\alpha_0$, and $\varrho_0$,  depending only on $\tilde a$, $n$, and $s$, such that the following holds. 
Assume that  $u_{\tilde \ep}$ is a solution of $(-\Delta)^{s/2} u_{\tilde \ep} + {\tilde \ep}^{-s} W'(u_{\tilde \ep})=0$ in~$B_1$ satisfying 
\begin{equation}\label{xxxeee11}
\{x_n< -\sigma_0\} \subset \{u_{\tilde \ep}<-{\textstyle \frac 3 4}  \} \subset \{u_{\tilde \ep}<{\textstyle \frac 3 4} \} \subset \{x_n< \sigma_0\} \quad \mbox{in } B_1.
\end{equation}

Then,  for all $z\in \{u_{\tilde \ep}=0\} \cap B_{3/4}$ and $k\ge 2$ satisfying $2^{-k}  \ge \tilde \ep^{\,\delta_0}$ we have 
\begin{equation}\label{xxxeee22}
\begin{split}
\{\omega_{z,k}\cdot(x-z)<  -\tilde a 2^{-(1+\alpha_0)k}\varrho_0 \}  \subset \{u_{\tilde \ep} <-{\textstyle \frac 3 4}  \} \subset \\
 \subset \{u_{\tilde \ep}<\textstyle \frac 3 4\} \subset \{\omega_{z,k}\cdot(x-z)< \tilde a 2^{-(1+\alpha_0)k}\varrho_0\}
 \end{split}
\end{equation}
in $B_{2^{-k} \varrho_0}(z)$, for some $\omega_{z,k}\in S^{n-1}$.
\end{thm}
\begin{proof}
 Let $a_0$, $\alpha_0$, $p_0$ be the constants from Proposition \ref{cor:improvement}  (which depend only on $n$ and $s$) and define $\delta_0 := \frac{1}{2+\alpha_0p_0}$.
It remains to choose $\sigma_0>0$ depending only on $\tilde a$, $n$, and $s$.
Note that we may assume 
\begin{equation}\label{hweiothiowehoiw}
\tilde a \tilde \ep^{(1+\alpha_0)\delta_0} \le \sigma_0
\end{equation}
since otherwise \eqref{xxxeee22}  follows immediately from \eqref{xxxeee11}.

Choose
\begin{equation}\label{ewhtiowhiowh}
a : = \min\{\tilde a, a_0\} \qquad \mbox{and} \qquad \varrho_0:= 2^{-j_a-2},
\end{equation}
where $j_a$ was defined as in \eqref{ja}.

Now, for any $z\in \{u_{\tilde \ep} =0\} \cap B_{3/4}$, let
\[ v^z : = u_{\tilde \ep}( z + \varrho_0\,\cdot\,).\]
Note that $v^{z}$ satisfies $(-\Delta)^{s/2} v^{z} + (\tilde\ep/\varrho_0)^{-s} W'(v^{z})=0$ in $B_{\frac{1}{4\varrho_0}} = B_{2^{j_a}}$.

Choose now $\sigma_0>0$ small so that $\sigma_0/\varrho_0 \le a$. Then it is immediate to verify that, thanks to \eqref{xxxeee11},  $v^{z}$ satisfies the assumption  \eqref{wethiowehtowh} of Proposition \ref{cor:improvement} with $\omega_j = e_n$ for all $j=0,\dots, j_a$.

Hence, defining $\ep := \tilde \ep/\varrho_0$, provided 
\begin{equation} \label{whoithowi}
2^{k-1}\ep  \le (2^{-\alpha_0(k-1)} a)^{p_0},
\end{equation}
Proposition \ref{cor:improvement} yields  
\begin{equation}\label{xxxeee3}
\{\omega_{z,k}\cdot x  < -\tilde a 2^{-(1+\alpha_0)k} \}  \subset \{u_{\tilde \ep} <-{\textstyle \frac 3 4}  \} \subset 
\{u_{\tilde \ep}<\textstyle \frac 3 4\} \subset \{\omega_{z,k}\cdot x< \tilde a 2^{-(1+\alpha_0)k}\}
\end{equation}
in $B_{2^{-k}}$, for some $\omega_{z,k}\in S^{n-1}$.  Note that \eqref{xxxeee3} immediately yields \eqref{xxxeee22} after scaling. 
Thus, it only remains to show that \eqref{whoithowi} is satisfied thanks to our assumption $2^{-k}  \ge \tilde \ep^{\,\delta_0}$ and our choice of $\delta_0$. 

Indeed, 
\[
\eqref{whoithowi} \quad  \Leftrightarrow \quad \frac{\tilde \ep}{\varrho_0} \le ( 2^{1-k})^{\alpha_0p_0+1}  a^{p_0}\quad  \Leftrightarrow \quad  2^{-k} \ge  \frac{(\varrho_0 a^{p_0})^{-\frac{1}{\alpha_0p_0 +1}}}{2} \tilde \ep^{\frac{1}{1+\alpha_0p_0}} = c_a\tilde \ep^{\frac{1}{1+\alpha_0p_0}}.
\]
But  since we choose $\delta_0< \frac{1}{1+\alpha_0p_0}$,  recalling \eqref{hweiothiowehoiw} we can absorb the multiplicative constant after possibly decreasing $\sigma_0$ in order to ensure that $\tilde \ep$ is sufficiently small ---recall that $a$ was fixed in \eqref{ewhtiowhiowh}.
\end{proof}

Now, we introduce a  class of ``good'' sets. By definition, the characteristic function of a ``good'' set must be the limit of stable solutions to the fractional Allen-Cahn equation with parameter $\ep$, as a sequence of $\ep$ tends to $0$. As we will show in Proposition \ref{good} below, these sets  are ``good'' in the sense that they inherit several good properties from the approximating sequence, such as $BV$  and energy estimates, a monotonicity formula, density estimates, and the improvement of flatness. As approximating sequence, we will take later the blow-downs of an entire stable solution to the fractional Allen-Cahn equation.

\begin{defi}\label{defA}
We say that a set $E\subset \R^n$ belongs to the class $\mathcal A$ when
there exists a sequence of functions  $u_j$, with $|u_j|< 1$,  which are stable solutions of
\begin{equation}\label{approx-AC}(-\Delta)^{s/2}u_j+\varepsilon_j^{-s} W'(u_j)=0 \quad \mbox{in } \R^n, \quad \mbox{with}\quad \varepsilon_j\downarrow 0 \;\;\mbox{as}\;\;j\uparrow \infty,\end{equation}
where $W(u) = \frac 1 4 (1-u^2)^2$, and such that
$$u_j \stackrel{L^1_{\rm loc}}{\longrightarrow} \chi_E-\chi_{E^c}\;\;\mbox{as}\;\;j\uparrow \infty.$$
\end{defi}

\begin{prop}\label{good}
Any set $E\in\mathcal A$ satisfies the following properties.
\begin{enumerate}
\item \textbf{BV and energy estimates}. 
\begin{equation}\label{estimates-sets}{\rm Per}(E,B_R)\leq CR^{n-1}\quad \mbox{and}\quad P_s(E,B_R)\leq CR^{n-s},\end{equation}
where $C$ is a constant which depends only on $n$ and $s$.
\vspace{0.5em}
\item \textbf{Monotonicity formula}. Let $\bar u :=\chi_E-\chi_{E^c}$ and let $\bar U$ be its $s$-extension in $\R^{n+1}_+$ (see Definition \ref{defextension}). We set
$$\Phi_E(R)=\frac{1}{ R^{n-s}}\int_{\widetilde B_{R}^+} y^{1-s}|\nabla \bar U(x,y)|^2\,dx\,dy.$$
Then, $\Phi_E$ is a nondecreasing function of $R$ and $\Phi_E(R)$ is constant if and only if $E$ is a cone (i.e., $\bar U$ is homogeneous of degree $0$).
\vspace{0.5em}
\item \textbf{Density estimate}. For some positive constant $\omega_0$, which depends only on $n$ and $s$, we have that
if
\begin{equation*}\label{hp-densitysets}
R^{-n}|E\cap B_R|  \leq \omega_0 \quad (\mbox{respectively,  }  R^{-n}| E^c\cap B_R|  \leq \omega_0) 
\end{equation*}
for some $R>0$, then
\begin{equation*}\label{th-densitysets}
|E \cap B_{R/2}| = 0 \quad (\mbox{respectively,  }  | E^c \cap B_{R/2}| = 0). 
\end{equation*}
\item \textbf{Improvement of flatness}. There exists $\sigma_0>0$, which depends only on $n$ and $s$, such that
$$\mbox{if}\quad \partial E\cap B_2 \subset \{x\in \R^n\,:\, |x\cdot e_n|\leq  \tilde \sigma_0\},$$
then $\partial E\cap B_{1/2}$ is a $C^{1,\alpha}$ graph in the $e_n$ direction.
\vspace{0.5em}
\item \textbf{Blow-up}. Let $E_{r_i,x_0}:=\frac{E-x_0}{r_i}$ with $r_i\downarrow 0$ as $i\uparrow \infty$. If $E_{r_i,x_0}\stackrel{L^1_{\rm loc}}{\longrightarrow}  E_*$, then $E_* \in \mathcal A$.
\end{enumerate}
\end{prop}

\begin{rem}
None of the properties (1)-(5) in Proposition \ref{good} are known to hold within the class of all weakly stable sets for the fractional perimeter $P_s$.
This is the reason that brings us to introduce the class $\mathcal A$.
\end{rem}

The results established in our paper allow us to give the

\begin{proof}[Proof of Proposition \ref{good}]
Recall that (see Definition \ref{defA}) $E\in \mathcal A$ if there exists a sequence $u_j$ of stable solutions to the fractional Allen-Cahn equation with parameter $\ep_j\downarrow 0$ such that $u_j \to \chi_E- \chi_{E^c}$ as $j\uparrow \infty$.

(1) The first estimate in \eqref{estimates-sets} follows easily passing to the limit the $BV$ estimate $u_j$ established in \eqref{BVest}  (the total variation is lower-semicontinuous). 
We emphasize again ---see Remark \ref{remscaling}--- that the $BV$ estimate \eqref{BVest}  is independent of the potential $W$ in the statement of Theorem \ref{BV}. We are strongly  using this here since $u_j$ satisfies $(-\Delta)^{s/2}u_j+\varepsilon_j^{-s} W'(u_j)=0$ and hence the associated potential $\varepsilon_j^{-s} W(u)$ converges to infinity  (since $\ep_j \downarrow 0$) except at $u=\pm1$. Similarly, the estimate on the fractional perimeter follows passing to the limit the estimate of Corollary \ref{energy-est}. 
\smallskip

(2) Denote $\bar u : = \chi_{E}-\chi_{E^c}$.  Let $U_j$ be the $s$-extension of $u_j$ (as defined in Section~\ref{sec-5}). By Proposition \ref{mon} (rescalled), we know that 
$$\Phi_j(R):= \frac{1}{R^{n-s}} \Big\{  \frac{d_s}{2} \int_{\tilde B_R^+} y^{1-s} |\nabla U_j (x,y)|^2 \,dx\,dy + \int_{B_R} \ep_j^{-s} W(u_j(x))\,dx \Big\}
$$
is a nondecreasing function of $R$.

Hence, we have a sequence $\Phi_j(R)$ of nondecreasing functions which is uniformly bounded (in $R$ and $j$) by the energy estimate \eqref{egest} of Theorem \ref{thm1}. 
Moreover, arguing as in the proof of Proposition \ref{conv-L^1} ---see \eqref{strong}---,   the convergence $u_j\to \bar u$ in $L^1_{\rm loc}(\R^n)$ can be upgraded to a strong convergence  in $W^{s/2,2}_{\rm loc}(\R^n)$ and also  $U_j\to \bar U$ in $W^{1,2}_{\rm loc}(\R^{n+1}_+, y^{1-s})$.
Also, thanks to Proposition \ref{propnew}, the potential term in $\Phi_j(R)$ converges to zero (as $\ep_j \downarrow 0$) for any fixed $R>0$.

Hence, passing to the limit as $j\to \infty$, we deduce that this monotonicity property is satisfied by the limiting function $\Phi_E(R)$. Moreover similarly as in \eqref{wnoiwoieht}-\eqref{wnoiwoieht2} we obtain
\begin{equation}
\Phi_j(R_2)-\Phi_j(R_1)  \ge d_s \int_{B_{R_2}\setminus B_{R_1}} \frac{y^{1-s}}{( |x|^2 +y^2)^{\frac{n-s}{2}}}(\partial_{r} U_j)^2 \,dx\,dy.
\end{equation}
Since $U_j\to \bar U$ strongly $W^{1,2}_{\rm loc}(\R^{n+1}_+, y^{1-s})$ we  also obtain that  if $\Phi_E$ is constant for $\bar U$ then $E$ must be a cone.

\smallskip

(3) The density estimate follows easily by passing to the limit the corresponding density estimate (established in Proposition \ref{density}) for the approximating sequence $u_j$.

\smallskip

(4) It follows from Theorem \ref{thm63utbg} and the density estimates in Proposition \ref{density}. Indeed,  fix $\tilde a>0$ and assume that  $\partial E\cap B_2 \subset \{x\in \R^n\,:\, |x\cdot e_n|\leq \tilde \sigma_0\}$. Then the density estimates imply the convergence of $\{u_j>t\}$ in Hausdorff distance towards $E$ for all $t\in (-1,1)$ {(see the last part of the statement of Theorem \ref{thm2})} we have, for $\sigma_0 = 8\tilde \sigma$,
\[
\big\{x_n< - \sigma_0\big\} \subset \{u_j<-{\textstyle \frac 3 4 } \} \subset \{u_j<{\textstyle \frac 3 4 }\} \subset \big\{x_n<  \sigma_0\big\} \quad \mbox{in } B_{1}.
\]  

Then, Theorem \ref{thm63utbg} (rescaled) yields that for all $z\in \{u_{j}=0\} \cap B_{\frac{3}{4}}$ and $k\ge 1$ satisfying $2^{-k}  \ge \ep_j^{\,\delta_0}$ we have 
\[
\begin{split}
\{\omega_{z,k}\cdot(x-z)<  -\tilde a 2^{-(1+\alpha_0)k} \varrho_0\}  \subset \{u_{\tilde \ep_j} <-{\textstyle \frac 3 4 } \} \subset \\
 \subset \{u_{\tilde \ep}<{\textstyle \frac 3 4 }\} \subset \{\omega_{z,k}\cdot(x-z)< \tilde a 2^{-(1+\alpha_0)k}\varrho_0\}
 \end{split}
 \]
in $B_{2^{-k}\varrho_0}(z)$, for some $\omega_{z,k}\in S^{n-1}$. After passing this information to the limit (using again Hausdorff convergence), we deduce that for all $z\in \partial E \cap B_{3/4}$ we have
\begin{equation}\label{wtnoiwhtiowhwo}
\{\omega_{z,k}\cdot(x-z) \le -\tilde a 2^{-(1+\alpha_0)k}\varrho_0\} \subset E  \subset \{\omega_{z,k}\cdot(x-z)< \tilde a 2^{-(1+\alpha_0)k}\varrho_0\}
\end{equation}
in $B_{2^{-k}\varrho_0}(z)$ for all $k \ge 0$.
Similarly as for classical minimal surfaces, this implies that $\partial E$ is a $C^{1,\alpha}$ graph inside $B_{1/2}$ provided $\tilde a$ is chosen small enough depending on $\alpha_0$ (note that \eqref{wtnoiwhtiowhwo} yields $|\omega_{z,k} -\omega_{z,k+1}| \le C_0 \tilde a 2^{-\alpha_0 k}$ and hence, by triangle inequality and summing a geometric series $|e_n - \omega_{z,k}| \le  C_0 \tilde a \frac{1}{1-2^{-\alpha_0}}<\frac {1}{10}$ for all $k\ge 0$, provided $\tilde a$ is chosen small).

\smallskip

(5) Since $E^i:=E_{r_i,x_0}$ belongs to $\mathcal A$ for every $i\in \mathbb N$, then it can be approximated by a sequence $u^i_j$ as in \eqref{approx-AC}. By assumption for each $m\in \mathbb N$ there exists $i_m$ such that $| (E^{i_m}\setminus E_*) \cup (E_*\setminus E^{i_m} ))\cap B_m| \le \frac 1 m$. Also given this $i_m$ there exists $j_m$  such that 
$$
 \int_{B_m} \big|u^{i_m}_{j_m} - (\chi_{E^{i_m}}-\chi_{(E^{i_m})^c}) \big| \le \frac 1 m.
 $$ 
 Hence, $u^{i_m}_{j_m} \stackrel{L^1_{\rm loc}}{\longrightarrow} \chi_{E_*}-\chi_{E^c_*}$ and thus $E_*\in \mathcal A$.
\end{proof}

We also need the following classification theorem for cones in $\mathcal A$ which are translation invariant in all directions but two of them.
\begin{lem}\label{hwioheoithe}
Assume that some nontrivial cone $\Sigma\subset \R^n$ belongs to $\mathcal A$ and is of the form 
\[\widetilde\Sigma \times \R^{n-2}\]
for some cone $\widetilde\Sigma\subset \R^2$. 
Then, $\Sigma$ is a half-space.
\end{lem} 
\begin{proof}
If $\Sigma$  (which is nontrivial) is not a half-space, then $\tilde \Sigma$ is not a half-plane (and is also nontrivial). Hence, $\partial\widetilde\Sigma$ must contain at least two non-aligned rays. 
Recall that, thanks to the density estimates for the class $\mathcal A$, we can chose a representative among sets that differ from $\Sigma$ for a set of measure zero such that every point of the topological boundary of $\Sigma$ has positive density for both $\Sigma$ and $\Sigma^c$. 
{Also, thanks to the improvement of flatness property for the class $\mathcal A$, the outwards normal vectors (in~$\R^2$) $\nu_1$ and $\nu_2$ to these two non-aligned rays need to form some positive angle, that is, $|\nu_1-\nu_2|^2\ge c>0$, for some positive constant $c$ depending only on $n$ and $s$.
Let us denote by $\widetilde H_1$ and $\widetilde H_2$ these two non aligned rays, that is for $i=1,2$, we set
$$\widetilde H_i:=\{\widetilde x \in \R^2\,|\,\widetilde x=t\,\omega_i,\,t>0\},$$
for some vectors $\omega_1,\,\omega_2\in \R^2$ such that $\omega_i \cdot \nu_i=0$ for $i=1,2$.
Moreover, we set
$$H_i =\widetilde H_i\times \R^{n-2},\quad i=1,\,2.$$

With these notations, we have that
$$H_1\times H_2 \subset \partial \Sigma \times \partial \Sigma.$$}
We fix a non-increasing cutoff function $\xi \in C^\infty_c([0,3))$ such that $\xi \equiv 1$ in $[0,2]$, and $\xi \le 1$, and we define 
\begin{equation}\label{psi}
\psi(x) = \xi\bigg(\sqrt{x_1^2+x_2^2}\bigg) \xi(|x_3|)\cdots \xi(|x_n|).
\end{equation}

In what follows, we change notation and we denote points in $\R^n \times \R^n$ by $(x,y)$  ---instead of the usual $(x, \bar x)$.

We claim that, for every $0<r<1$, the following estimate holds:
\begin{equation}\label{I1}
I_1 : = \iint_{\partial\Sigma \times \partial\Sigma}  \frac{\big|\nu_\Sigma (x) -\nu_\Sigma(y)\big|^2}{(r^2 +|x-y|^2)^{\frac{n+s}{2}}}\,  \psi^2(x)  \psi^2(y)d\mathcal H^{n-1}(x) d\mathcal H^{n-1}(y) \ge C(n,s) r^{-s},
\end{equation}
for some positive constant $C(n,s)$ depending on $n$ and $s$.

To prove the claim, we use the notation $x=(\widetilde x,x')\in \R^2 \times \R^{n-2}$. We have that
\begin{equation*}
\begin{split}
I_1 &\ge c\int_{H_1}\int_{H_2}\frac{\psi^2(x)  \psi^2(y)}{(r^2+|x-y|^2)^{\frac{n+s}{2}}}d\mathcal H^{n-1}(x) d\mathcal H^{n-1}(y)\\
&\ge c\iint_{\R^{n-2}\times \R^{n-2}}dx'\,dy'\int_{\widetilde H_1}d\mathcal H^1({\widetilde x})\int_{\widetilde H_2}d\mathcal H^1({\widetilde y})\frac{\psi^2(x)  \psi^2(y)}{(r^2+|x-y|^2)^{\frac{n+s}{2}}}.
\end{split}
\end{equation*}
Using the change of variables $X=(0,|\widetilde x|)$, $Y=(0,-|\widetilde y|)$ and the triangle inequality $|\widetilde x -\widetilde y| \le |\widetilde x|+|\widetilde y|=|X-Y|$, we get 
\begin{equation*}
\begin{split}
I_1 &\ge c \iint_{\R^{n-2}\times \R^{n-2}}dx'\,dy'\int_{\{0\}\times\R^{+}}dX\int_{\{0\}\times\R^{-}}dY\frac{\psi^2(X,x')\psi^2(Y, y')}{(r^2+|X-Y|^2+|x'-y'|^2)^{\frac{n+s}{s}}}\\
&\ge c \int_{B_1^{n-2}}dx'\int_{\{y'\in \R^{n-2}\,|\,|x'-y'|<1\}}dy'\int_0^1 dX\int_{-1}^0dY\frac{1}{(r^2+|X-Y|^2+|x'-y'|^2)^{\frac{n+s}{s}}}\\
&\ge C c|B_1^{n-2}| \int_0^1d\rho\int_0^1 dX\int_{-1}^0dY\frac{\rho^{n-3}}{(r^2+|X-Y|^2+\rho^2)^{\frac{n+s}{s}}},
\end{split}
\end{equation*}
where we are identifying a point $(0,X)\in \{0\}\times \R^+$ with the real number $X\in \R^+$ and the integration over $\{0\}\times \R^+$ with integration over $\R^+$ (analogously for $Y\in \R^-$), and in the last inequality we have used polar coordinates in $\R^{n-2}$. 

Finally, using the change of variables $\bar X=X/r$, $\bar Y=Y/r$, $\bar \rho=\rho/r$, and recalling that $0<r<1$, we deduce
\begin{equation*}
\begin{split}
I_1 &\ge C(n,s)\frac{1}{r^{n+s}}\int_0^1d\rho\int_0^1 dX\int_{-1}^0dY\frac{\rho^{n-3}}{\left(1+\frac{|X-Y|^2}{r^2}+\frac{\rho^2}{r^2}\right)^{\frac{n+s}{2}}}\\
&=C(n,s)\frac{1}{r^{n+s}}\int_0^{1/r}d\bar\rho\int_0^{1/r} d\bar X\int_{-1/r}^0 d\bar Y \frac{r^n\cdot \bar \rho^{n-3}}{(1+|\bar X-\bar Y|^2+\bar \rho^2)^{\frac{n+s}{2}}}\\
&\ge C(n,s)r^{-s}\int_0^{1}d\bar\rho\int_0^{1} d\bar X\int_{-1}^0d\bar Y \frac{ \bar \rho^{n-3}}{(1+|\bar X-\bar Y|^2+\bar \rho^2)^{\frac{n+s}{2}}}=C(n,s) r^{-s},
\end{split} 
\end{equation*}
which concludes the proof of \eqref{I1}.

Recall now that since $\Sigma$ belongs to the class $\mathcal A$, there exists, by definition, a sequence of functions  $u_j: \R^n \to (-1,1)$  which are stable solutions of $(-\Delta)^{s/2}u_j+\varepsilon_j^{-s} W'(u_j)=0$
and such that
$$u_j \stackrel{L^1_{\rm loc}}{\longrightarrow} \chi_\Sigma-\chi_{\Sigma^c}\;\quad \mbox{as}\;\;j\uparrow \infty.$$

We claim now that the following inequality holds:

\begin{equation}\label{I2-I3}
\begin{split}
I_2: = \iint_{\R^n \times \R^n} &\frac{\big|n_j (x) -n_j(y)\big|^2}{|x-y|^{n+s}} \psi(x)^2 |\nabla u_j|(x) dx\, |\nabla u_j|(y) dy
\\
&\qquad \le  \iint_{\R^n \times \R^n} \frac{|\psi(x) -\psi(y)|^2}{|x-y|^{n+s}}  |\nabla u_j|(x) dx\, |\nabla u_j|(y) dy =:I_3, 
\end{split}
\end{equation}
where 
\begin{equation}\label{nu}
n_j(x)  : =
\begin{cases} 
\frac{\nabla u_j}{|\nabla u_j|} &\quad \mbox{where } \nabla u \neq 0 \\
0  & \quad\mbox{where } \nabla u = 0.
\end{cases}
\end{equation}

Let us prove \eqref{I2-I3}. We start by observing that the stability condition \eqref{stable} (with $\Omega=\R^n$), written for the functions $u_j$, is equivalent to requiring that
\begin{equation}\label{stable2}
\int_{\R^n} \xi(x)(-\Delta)^{s/2}\xi(x)\,dx + \int_{\R^n}\varepsilon^{-s}_j W''(u_j)\xi^2(x)\,dx\ge 0
\end{equation}
for any Lipschitz function $\xi$ which is compactly supported in $\R^n$.

Let us now choose $\xi$ of the form $\xi=\eta \cdot \psi$, where $\eta$ is a Lipschitz function and $\psi \in C^\infty_0(\R^n)$.
A simple computation gives that
$$(-\Delta)^{s/2}\xi(x)=\psi(x)(-\Delta)^{s/2}\eta(x) + \int_{\R^n}\eta(y) \frac{\psi(x)-\psi(y)}{|x-y|^{n+s}}\,dy,$$
which implies
\begin{equation*}
\begin{split}
\int_{\R^n}\xi(x)(-\Delta)^{s/2}\xi(x)\,dx &= \int_{\R^n} \psi^2(x)\eta(x) (-\Delta)^{s/2}\eta(x)\,dx\\
&\hspace{1em}+\iint_{\R^n \times \R^n}\eta(x)\eta(y)\psi(x) \frac{\psi(x)-\psi(y)}{|x-y|^{n+s}}\,dx\,dy\\
&=\int_{\R^n} \psi^2(x)\eta(x) (-\Delta)^{s/2}\eta(x)\,dx\\
&\hspace{1em}+\frac{1}{2}\iint_{\R^n \times \R^n}\eta(x)\eta(y)\frac{|\psi(x)-\psi(y)|^2}{|x-y|^{n+s}}\,dx\,dy.
\end{split}
\end{equation*}
Hence, the stability condition \eqref{stable2} becomes
\begin{equation}\label{stable3}
\begin{split}
&\int_{\R^n} \psi^2(x)\eta(x) (-\Delta)^{s/2}\eta(x)\,dx
+\frac{1}{2}\iint_{\R^n \times \R^n}\eta(x)\eta(y)\frac{|\psi(x)-\psi(y)|^2}{|x-y|^{n+s}}\,dx\,dy\\
&\hspace{2em} + \int_{\R^n}\varepsilon^{-s}_j W''(u_j)\eta^2(x)\psi^2(x)\,dx \ge 0.
\end{split}
\end{equation}

We use now the fact that, for any $i=1,\dots,n$, $\partial_{x_i}u_j$ satisfies the linearized equation $(-\Delta)^{s/2}v+\varepsilon_j^{-s}W''(u)v=0$. By multiplying the (vectorial) equation satisfied by $\nabla u_j$ by $\nabla u_j$ itself, we deduce that
\begin{equation}\label{eq-grad}
\nabla u_j \cdot(-\Delta)^{s/2}\nabla u_j + \varepsilon_j^{-s}W''(u)|\nabla u_j|^2=0.
\end{equation}

Let us now choose $\eta=|\nabla u_j|$ in the stability inequality \eqref{stable3} (this is an admissible choice by the regularity results of Appendix \ref{app-C}) and use \eqref{eq-grad}, to obtain
\[
\begin{split}
&\int_{\R^n}\psi^2(x)|\nabla u_j|(x)(-\Delta)^{s/2}|\nabla u_j|(x)\,dx \\
&\hspace{2em}+\frac{1}{2}\iint_{\R^n \times \R^n}|\nabla u_j(x)||\nabla u_j(y)|\frac{|\psi(x)-\psi(y)|^2}{|x-y|^{n+s}}\,dx\,dy\\
&\hspace{2em}-\int_{\R^n}\psi^2(x)\nabla u_ j(x) \cdot (-\Delta)^{s/2} \nabla u_j(x)\,dx \ge 0.
\end{split}
\]
To conclude the proof of \eqref{I2-I3}, it is enough to observe that 
\[
\begin{split}
&\int_{\R^n}\psi^2(x)|\nabla u_j|(x)(-\Delta)^{s/2}|\nabla u_j|(x)\,dx -\int_{\R^n}\psi^2(x)\nabla u_ j(x) \cdot (-\Delta)^{s/2} \nabla u_j(x)dx\\
&=\negmedspace \iint_{\R^n \times \R^n}\negmedspace\psi^2(x)\negmedspace\left(|\nabla u_j|(x) \frac{|\nabla u_j|(x)-|\nabla u_j|(y)}{|x-y|^{n+s}}-\nabla u_j(x)\cdot \frac{\nabla u_j(x)-\nabla u_j(y)}{|x-y|^{n+s}}\right)\negmedspace dx dy\\
&=\negmedspace \iint_{\R^n \times \R^n}\negmedspace\psi^2(x)\frac{\nabla u_j(x)\cdot \nabla u_j(y) - |\nabla u_j|(x)|\nabla u_j|(y)}{|x-y|^{n+s}}dxdy.
\end{split}
\]
This, together with the definition \eqref{nu} of $n_j$, proves \eqref{I2-I3}.

By our uniform $BV$ estimates, we have that  $\nabla u_j \to -2D\chi_{\Sigma}$ weakly$^*$ as Radon measures (see, e.g., Proposition 3.13 and formula (3.11) in \cite{AFP}). Here and in the following, we denote by $D\chi_{\Sigma}$ the perimeter measure and by $|D\chi_{\Sigma}|$ its total variation.

Let us show that
\[ I_1 \le I_2\]
for $j$ large enough, where these quantities were defined in \eqref{I1} and \eqref{I2-I3}.

Indeed, we start by observing that
\[
\begin{split}
I_2 &= 2 \iint_{ \R^n\times \R^n}    \frac{ \big( |\nabla u_j|(x) |\nabla u_j|(y) - \nabla u_j(x)\cdot \nabla u_j(y) \big)}{|x-y|^{n+s}} \psi^2(x)\, dx\, dy \\
&\ge  2 \iint_{ \R^n\times \R^n}    \frac{ \big( |\nabla u_j|(x) |\nabla u_j|(y) - \nabla u_j(x)\cdot \nabla u_j(y) \big)}{(r^2+ |x-y|^2)^{\frac{n+s}{2}}} \psi^2(x)\psi^2(y) \, dx\, dy. 
\end{split}
\]

Now, we claim that:
\begin{equation}\label{conv1}
\begin{split}
&\lim_{j\rightarrow \infty} \iint_{ \R^n\times \R^n} \frac { \nabla u_j(x) \cdot\nabla u_j(y)} {(r^2+ |x-y|^2)^{\frac{n+s}{2}}} \psi^2(x)\psi^2(y) \, dx\, dy  \\
&\hspace{2em}= \iint_{ \R^n\times \R^n}  \frac{ 2 D \chi_{\Sigma}(dx) \cdot 2 D \chi_\Sigma (dy)} {(r^2+ |x-y|^2)^{\frac{n+s}{2}}} \psi^2(x)\psi^2(y) 
\end{split}
\end{equation}
and
\begin{equation}\label{conv2}
\begin{split}
&\liminf_{j\rightarrow \infty} \iint_{ \R^n\times \R^n} \frac { |\nabla u_j|(x) |\nabla u_j|(y)} {(r^2+ |x-y|^2)^{\frac{n+s}{2}}} \psi^2(x)\psi^2(y) \, dx\, dy  \\
&\hspace{2em} \ge\iint_{ \R^n\times \R^n}  \frac{ 2|D \chi_{\Sigma}|(dx) 2|D \chi_\Sigma |(dy)} {(r^2+ |x-y|^2)^{\frac{n+s}{2}}} \psi^2(x)\psi^2(y).
\end{split}
\end{equation}

We first prove \eqref{conv1}.
For every $y\in \R^n$, we define 
$$G_j(y):=\int_{\R^n}\nabla u_j(x)\frac{\psi^2(x)\psi^2(y)}{(r^2+|x-y|^{2})^{\frac{n+s}{2}}}\,dx=\int_{\R^n}\nabla u_j(x)\Psi(x,y)\,dx.$$
Using the BV estimate for $u_j$ (which is uniform in $j$) and that, for $y$ fixed, the function $\Psi$ is smooth and compactly supported (as function of $x$), we deduce that the family of functions $\{G_j\}_{j\in \mathbb N}$ is equibounded and equicontinuous. This, combined with the weak$^*$-convergence of $\nabla u_j$ to $-2D\chi_{\Sigma}$, gives that, as $j\rightarrow \infty$, $G_j$ converges uniformly to $G$, where
$$G(y):= -2\int_{\R^n}D\chi_{\Sigma}(dx)\Psi(x,y).$$
Hence, we have that
\begin{equation*}
\begin{split}
&\iint_{ \R^n\times \R^n} \nabla u_j(x) \cdot\nabla u_j(y)\Psi(x,y)\, dx\, dy\\
&\hspace{2em}-\iint_{ \R^n\times \R^n}  2 D \chi_{\Sigma}(dx) \cdot 2 D \chi_\Sigma (y)\Psi(x,y) dy  \\
&\hspace{1em}=\int_{\R^n}\nabla u_j(y)\cdot G_j(y)\, dy +2\int_{R^{n}}D\chi_{\Sigma}(dy)\cdot G(y)\\
&\hspace{1em}=\int_{\R^n}\nabla u_j(y)\left(G_j(y)-G(y)\right)\,dy+\int_{\R^n}G(y)\left(\nabla u_j(y)+2 D\chi_{\Sigma}(dy)\right)\rightarrow 0,
\end{split}
\end{equation*}
where the first term tends to zero thanks to the uniform convergence of $G_j$ to $G$ and the uniform BV estimate (on compact sets) for $u_j$, and the second term also vanishes in the limit since $\nabla u_j$ weak$^*$-converge to $-2 D\chi_{\Sigma}$ and $G$ is smooth and compactly supported.

To get \eqref{conv2}, we reason similarly as before and use the lower-semicontinuity of the total variation.

Hence, from \eqref{conv1} and \eqref{conv2}, we get, as $j\to \infty$,
\[
I_2\ge 4I_1 -o(1) \ge I_1.
\]

Finally, let us show that 
\[
I_3 \le C_1(n,s),
\]
where $I_3$ was defined in \eqref{I2-I3}.

Indeed,  by Theorem \ref{BV} we have  $\int_{B_\varrho (x)} |\nabla u_j(y)|\,dy \le C(n,s) \varrho^{n-1}$ for all $x\in \tilde B_3\times[-3,3]^{n-2}$ and $\rho >0$ (here $\tilde B_3$ denote the ball in $\R^2$ centered at $0$ and with radius $3$). Hence
for all $x\in \tilde B_3\times[-3,3]^{n-2}$, defining $A_j(x)= B_{2^{j+1}}(x) \setminus B_{2^j}(x)$ and using that $|\psi(x) -\psi(y)|^2 \le C(n)(|x-y|^2\wedge 1)$, we have

\[
\begin{split}
 \int_{\R^n}  \frac{(|x-y|^2\wedge 1)}{|x-y|^{n+s}} |\nabla u_j (y)|  dy &\le C\sum_{j\in \Z}  \frac{2^{2j}\wedge 1}{2^{j(n+s)}} 2^{j (n-1)} 
 \\
 &=  C \bigg(\sum_{j<0} 2^{(1-s)j}  +  \sum_{j\ge 0} 2^{-(1+s)j}\bigg) \le C_1(n,s).
\end{split}
\]

Hence we have shown that,  for $j$ large enough,  
\[C(n,s) r^{-s}\le  I_1 \le I_2 \le I_3\le  C_1(n,s). \]
Choosing $r>0$ small we obtain a contradiction.
\end{proof}

We have now all the ingredients to prove our main theorem.

\begin{proof} [Proof of Theorem \ref{thmclas}]
By our convergence result Theorem \ref{thm2} for any given  blow-down sequence $u_{R_j}(x)= u(R_jx)$ with $R_j\uparrow \infty$, there is a subsequence $R_{j_\ell}$ such that
\[ u_{R_{j_\ell}} \rightarrow \chi_{\Sigma}-\chi_{\Sigma^c} \quad \mbox{ in }L^1(B_1),\]
where $\Sigma$ is a cone which is a weakly stable set in $\R^n$ for the $s$-perimeter, which belongs to the class $\mathcal A$ (by definition of $\mathcal A$), and which is nontrivial. 
We next prove that under our assumption, i.e., that the half-spaces are the only smooth (away from $0$) stable nonlocal $s$-minimal cones in $\R^m\setminus\{0\}$ for any $3\leq m\leq n$, $\Sigma$ must be a half-space. The proof follows Federer's dimension reduction argument and is done by contradiction. Indeed, assume that $\Sigma$ is not a half-space, then, by the just mentioned assumption, there exists at least one point $p_1\in \partial \Sigma\cap S^{n-1}$ at which $\partial \Sigma$ is not smooth in any neighbourhood of $p_1$.

Let us consider the blow-up of $\Sigma$ at $p_1$,
$$\Sigma_{p_1,r}:=\frac{\Sigma-p_1}{r}.$$
Using the energy estimate and the monotonicity formula (points (1) and (2) in Proposition \ref{good}) we have that, up to a subsequence,
$$\Sigma_{p_1,r}\stackrel{L^1_{\rm loc}}{\longrightarrow} \Sigma_1,$$
where $\Sigma_1$ belongs to $\mathcal A$ by point (5) in Proposition \ref{good}. Moreover, by the density estimate (point (3) in Proposition \ref{good}), the convergence of blow-ups in the  $L^1_{\rm loc}$-sense can be upgraded to a local uniform convergence (i.e., locally  in Hausdorff distance). 
Now, if $\Sigma_1$ were a half-space, then by the improvement of flatness property (point (4) of Proposition \ref{good}) we would deduce that $\partial\Sigma$ is smooth in some neighbourhood of $p_1$, reaching a contradiction. Hence,  $\Sigma_1$ is not  a half-space. Now, $\Sigma_1$ being the blow-up of a cone at a point $p_1\neq 0$, we find that it must be translation invariant in the direction $p_1$.\footnote{This is done exactly as in the case of minimal surfaces: 
\[x\in \Sigma_{p_1,r}  \  \Leftrightarrow \  (p_1 + rx) \in \Sigma   \  \Leftrightarrow \  \frac{\lambda}{r}(p_1 + rx) \in \Sigma \  \Leftrightarrow \ \big( {\textstyle \frac{\lambda-1}{r}}p_1 + \lambda x\big)\in \Sigma_{p_1, r} \]
and hence, taking $\lambda = 1+tr$ and sending $r\to 0$, we obtain $x\in \Sigma_{1} \ \Leftrightarrow\  x+ tp_1\in \Sigma_1$.
} Hence,  up to a rotation, $\Sigma_1$  it must be of the form
$$\Sigma_1=\widetilde \Sigma_1\times \R,$$
where $\widetilde \Sigma_1 \subset \R^{n-1}$ is a nontrivial cone different from a half-space (since we proved that $\Sigma_1$ cannot be a half-space).

We can now iterate the same argument: if  $\tilde \Sigma_1$ were  smooth then it should  be flat by our assumption. Hence  $\widetilde \Sigma_1$ is  not smooth and thus there exists $p_2 = (\tilde p_2, 0)$, with $\tilde p_2\in \widetilde \Sigma_1 \cap S^{n-2}$, such that the blow-up $\Sigma_2$ of $\Sigma_1$ at the point $p_2$ belongs to $\mathcal A$ is nontrivial, is not a half-space, and is translation invariant in the direction $x_1$ and $x_2$. In other words, up to a rotation this second blow-up must be  some set $\Sigma_2$  in $\mathcal A$ of the form 
$$\Sigma_2=\widetilde \Sigma_2\times \R^2,$$
where $\widetilde \Sigma_2 \subset \R^{n-2}$ is a nontrivial cone different from a half-space.

 After $n-2$ iterations we arrive at a blow-up $\Sigma_{n-2}$ which belongs to $\mathcal A$ and must be of the form
$$\Sigma_{n-2}=\widetilde \Sigma_{n-2}\times \R^{n-2},$$
where $\widetilde \Sigma_{n-2} \subset \R^{2}$ is a nontrivial cone different from a half-space. Hence, using  Lemma \ref{hwioheoithe}, we reach a contradiction, proving that the initial cone $\Sigma$ must be a half-space.

Having proved that $\Sigma$ must be a half-space, we now recall that (by  Theorem \ref{thm2}) the convergence of sub-levelsets of $u_{R_{j_\ell}}$ to the half-space $\Sigma$ (in $B_1$) holds also in the sense of the Hausdorff distance. As a consequence $u$ satisfies the asymptotic  flatness assumption of Theorem \ref{123} and hence it follows that   $u(x)=\phi(e\cdot x)$ for some direction $e\in S^{n-1}$ and some increasing function $\phi:\R\to (-1,1)$. 
\end{proof}

\begin{rem}\label{thesame}
The following will be used in Appendix \ref{app-A} to deal with global stable $s$-minimal sets.
Notice that, reasoning exactly as in the proof of Theorem \ref{thmclas}, one can prove that under the same assumption, i.e. that, for some pair $(n,s)$ with $n\ge 3$ and $s\in(0,1)$, hyperplanes are the only stable $s$-minimal cones in $\R^n\setminus\{0\}$, then any set $E\subset \R^n$ belonging to the class $\mathcal A$ is necessarily a half-space. Indeed, after doing a blow-down of $E$, we reduce to a cone $\Sigma$ which is stable and belongs to $\mathcal A$. Hence, by the exact same argument as before, one get that $\Sigma$ is an half-space and finally, by the improvement of flatness property of $\mathcal A$, that so is $E$. More generally, an analogue abstract result which reduces the classification of a stable set $E$ to the classification of stable cones (smooth away from the origin), holds whenever $E$ belongs to a class for which the properties listed in Proposition \ref{good} are satisfied (BV and energy estimates, monotonicity formula, density estimates, improvement of flatness, and blow-up).
\end{rem}

We can now easily deduce our rigidity result in $\R^3$.

\begin{proof} [Proof of Corollary \ref{corclas2}]
Thanks to Theorem \ref{thmclas} we just need to show that for $s$ sufficiently close to 1,  any weakly  stable $s$-minimal cone in $\R^{3}\setminus\{0\}$ whose boundary is smooth away from $0$ must be either trivial or a half-space. This result was obtained in our previous paper \cite{CCS}. 
\end{proof}

Instead, to prove Corollary \ref{corclas3} on monotone solutions in $\R^4$ we first need to establish the following result, which holds in any dimension.

\begin{prop}\label{implication}
Let $n\ge 2$, $s\in(0,1)$, and $W(u)=\frac 1 4 (1-u^2)^2$.
Assume that $u:\R^n \to (-1,1)$ is a solution of \eqref{theequation} satisfying $\partial_{x_n} u>0$. Define $u^{\pm} := \lim_{x_{n} \to \pm\infty} u$. 

If each $u^+$ and $u^-$ is either a increasing 1D solution or is identically $\pm 1$, then $u$ is a minimizer in $\R^n$. 
\end{prop} 

The result for $s=2$ was proven in \cite[Theorem 1.1]{FariVal}. Here we show that the same type of argument works also for $s\in (0,2)$.

\begin{proof}[Proof of Proposition \ref{implication}]
Let $\Omega \subset \R^n$ be a bounded domain and let $v$ be a minimizer of $\mathcal E_\Omega$ with exterior datum $u$ outside of $\Omega$. 
Let us first show that $u^-\le v\le u^+$ in $\R^n$. Indeed, in the complement of $\Omega$ the inequality holds since $v=u$ there and  $u^-\le u\le u^+$ in $\R^n$.
To show that $u^-\le v\le u^+$ in $\Omega$, let us consider $\overline w^\pm = \max(v, u^\pm)$ and  $\underline w^\pm = \min(v, u^\pm)$. 

Assume by contradiction that $v>u^+$ at some point in $\Omega$. Then, arguing similarly as in \eqref{claimA-ii}, we would have 
\begin{equation}\label{wniogoihb}
\mathcal E_\Omega(\overline w^+)+\mathcal E_\Omega(\underline w^+)<\mathcal E_\Omega(v)+\mathcal E_\Omega(u^+).
\end{equation}
But now since on the one hand $v$ is a minimizer with exterior datum equal to $u$ in the complement of $\Omega$ we have $\mathcal E_\Omega(v) \le \mathcal E_\Omega(\underline w^+)$. On the other hand, since $u^+$ is either $+1$ or $1D$ and increasing,  it follows\footnote{A simple way of proving this consists of using the standard argument which involves the foliation of $\R^{n}\times(-1,1)$ given by the horizontal translations of the graph of $u^+$; see \cite[Proof of Proposition~6.2]{CC2}.} that  $u^+$ is also a minimizer.   Hence, noticing that $\overline w^+$ and $u^+$ coincide outside of $\Omega$, we obtain $\mathcal E_\Omega( u^+)\le \mathcal E_\Omega(\overline w^+)$. We therefore reach a contradiction with \eqref{wniogoihb}.

A similar argument using  $\overline w^-$ and $\underline w^-$ shows $u^-\le v$. 

Finally, using the standard ``foliation'' $\big\{u(x', x_n +t), \ t\in \R\big\}$, unless $v\equiv u$ we may  find a translation of  the graph of $u$ touching by above (or by below) the graph of $v$ as some point interior point in  $\Omega$, and this contradicts the strong maximum principle. Hence the only possibility is that $v\equiv u$ and thus $u$ is a minimizer in $\Omega$. Since $\Omega$ is an arbitrary bounded domain, we conclude that $u$ is a minimizer in~$\R^n$.
\end{proof}

We can finally give the

\begin{proof} [Proof of Corollary \ref{corclas3}]
By Corollary \ref{corclas2} the limits  $u^{\pm} := \lim_{x_{4} \to \pm\infty} u$ (which are stable solutions in $\R^3$)  must be either $\pm 1$ or increasing $1D$ solutions.
Thus, by Proposition \ref{implication}, $u$ is a minimizer in $\R^4$. Now, since $s$ is sufficiently close to 1, the corollary follows from Theorem 1.5 in  \cite{dPSV}.
\end{proof}

\appendix

\section{Smooth stable $s$-minimal surfaces in $\R^3$ are flat when $s\sim 1$}\label{app-A}

We give next the details of the proof of Theorem \ref{thmRn} and, as a consequence, of Corollary~\ref{corR3}.  

To prove Theorem \ref{thmRn}, we need to introduce the following definition, which is analogous to the one of the class $\mathcal A$ introduced in Section~\ref{sec-7}. 

\begin{defi}\label{defAprime}
We say that a set $E\subset \R^n$ belongs to the class $\mathcal A'$ when 
there exist a sequence of sets $E_j\subset \R^n$ with $E_j \to E$ in $L^1_{\rm loc}(\R^n)$ such that: 
\begin{itemize}
 \item the boundaries $\partial E_j$ are $(n-1)$-dimensional manifolds of class $C^2$;
 \item $E_j$ are weakly stable sets for the $s$-perimeter in $\R^n$. 
\end{itemize}
\end{defi}

\begin{prop}\label{good2}
Any set of the class $\mathcal A'$ satisfies the five properties  $(1)$-$(5)$ in Proposition $\ref{good}$.
\end{prop}
\begin{proof}

(1) \textbf{$\mathbf{BV}$ and energy estimates}. Since $E_j$ have smooth boundaries, the fact that these sets are weakly stable  easily gives that they are also stable in the sense of  Definition~1.6 in \cite{CSV}. Hence, by Corollary 1.8 in \cite{CSV}, we obtain that ${\rm Per}(E_j,B_R)\leq CR^{n-1}$. Since the perimeter is lower semicontinuous, the same estimate holds by approximation for sets in $\mathcal A'$. Now, this $BV$ estimate leads to the corresponding energy estimate by the exact same argument that we have given for solutions of the fractional Allen-Cahn equation. 

\smallskip
(2) \textbf{Monotonicity formula}.  We claim that if $F\subset \R^n$ is a weakly stable set with $C^2$ boundary, then the quantity 
$$\Phi_F(R)=\frac{1}{ R^{n-s}}\int_{\widetilde B_{R}^+} y^{1-s}|\nabla \bar V(x,y)|^2\,dx\,dy$$
is nondecreasing in $R$,
where $\bar V$ is the $s$-extension of $\bar v :=\chi_F-\chi_{F^c}$. 

This fact follows from the proof of Theorem 8.1 in \cite{CRS}.
Indeed, although Theorem~8.1 in \cite{CRS} is stated for minimizers, its proof only needs that $\bar V$ is a minimizer with respect to sufficiently small Lipschitz perturbations (the size of the perturbations used actually converges to zero). If $\partial E$ is smooth and $E$ is weakly stable, then it is a standard fact that any compact subset of $E$ is strictly stable (since the first eigenvalue of the Jacobi operator is strictly monotone with respect to inclusion of domains). In particular, for any given ball, smooth perturbations of $\partial E$ supported in this ball and with small enough size (depending on the ball) lead to an increased nonlocal perimeter. Since $\partial E$ is smooth, it is easy to see that the same property holds also for Lipchitz perturbations of sufficiently small size. As a consequence, the argument in the proof of Theorem 8.1 in \cite{CRS} (without any modification) also applies to the case of weakly stable sets with smooth boundary.\footnote{We actually expect that ---by a different proof similar to that of Proposition 3.2 in \cite{CC2}--- the monotonicity formula also holds for any critical point of the fractional perimeter with smooth boundary (not necessarily stable). 
The technical details for such proof would become quite more involved and this is why here we treat only the case of stable critical points ---for which the less technical argument in \cite{CRS} can be used.}

Now, with an analogous argument as in the proof of Proposition \ref{good}, if $E$ belongs to $\mathcal A'$, then we can take a sequence of approximating sets $E_j$ with smooth boundary. Since the convergence of $E_j\to E$ in $L^1_{\rm loc}$ gives that  $\bar U_j \to \bar V$   strongly in  $W^{1,2}_{\rm loc}(\R^{n+1}_+, y^{1-s})$,  where  $\bar U_j$  are the $s$-extensions of $\bar u_j :=\chi_{E_j}-\chi_{E_j^c}$, we obtain that  $\Phi_{E}$ must be  monotone since $\Phi_{E_j}$ are.

\smallskip
(3) \textbf{Density estimate}.  The density estimate for sets in the class $\mathcal A'$ follows from the $BV$ estimate with the exact same arguments as those in Section~\ref{sec-5}.

\smallskip
(4) \textbf{Improvement of flatness}. If $F\subset \R^n$ is any stable (or even stationary) set with $C^2$ boundary, the  improvement of flatness result in \cite[Theorem 6.8]{CRS} applies to $F$ (without any change in its proof). This is true because (unlike in the case ``$s=1$'' of classical minimal surfaces) the proof  of \cite[Theorem 6.8]{CRS} applies to any  viscosity solution of the nonlocal minimal surface equation  (in the sense given in \cite[Theorem 5.1]{CRS}).
In \cite{CRS} the minimality assumption is only used to show that the considered surfaces are viscosity solutions of the nonlocal minimal surface equation (this is done in \cite[Theorem 5.1]{CRS}). If one assumes that the boundary of $F$ is smooth then it is easy to see using the computation of the first variation (see for instance \cite{FFFMM}) that $F$ must be a viscosity solution of the nonlocal minimal surface equation. Hence, \cite[Theorem 6.8]{CRS} applies to $F$.

As a consequence (with a similar approximation argument as in the proof of Proposition \ref{good}), we obtain that the improvement of flatness property holds true for sets $E$ in the class $\mathcal A'$.

\smallskip
(5) \textbf{Blow-up}.  The closedness of the class $\mathcal A'$ under blow-up follows by the same argument as in the proof of Proposition \ref{good}.
\end{proof}

\begin{proof} [Proof of Theorem \ref{thmRn}]
Thanks to Proposition \ref{good2}, by the same argument as in the proof of Theorem \ref{thmclas} {(see Remark \ref{thesame})}, we may reduce the classification in $\R^n$ of stable $s$-minimal sets in the class $\mathcal A'$  to the classification of stable  $s$-minimal cones (smooth away from $0$)  in dimensions $3\le m\le n$. 
\end{proof}

\begin{proof} [Proof of Corollary \ref{corR3}]
It  follows from  Theorem \ref{thmRn} analogously as in the proof Proof of Corollary \ref{corclas2}.
\end{proof}

\section{On the control of the potential energy by the Sobolev energy}\label{app-B}

In this section we give a short proof of a weaker version of the estimate in Proposition \ref{DircontrolsPot2}. It is a weaker estimate since it has an additional additive term on its right hand side.

\begin{prop}\label{DircontrolsPot1}
Given $s_0>0$, let $n\ge 2$, $s\in(s_0,2]$,  $W(u)=\frac 1 4 (1-u^2)^2$, and $K$ satisfy \eqref{L0} and \eqref{L2}. Let $u:\R^n\rightarrow (-1,1)$ be a stable solution of   $L_Ku + W'(u)=0$ in $\R^n$  $($meaning $-a_{ij}\partial_{ij} u + W'(u)=0$ when $s=2$$)$.

 Then, 
\begin{equation}\label{estDirPot1} 
\mathcal E^{\rm Pot}_{B_{R}}(u) \le 
\begin{cases}
 \vspace{2mm}
 C\big(\mathcal E^{\rm Sob}_{B_{R+1}}(u) + R^{n-s}\big) \quad& \mbox{if }s\in (0,1) \\
 \vspace{2mm}
 C\big(\mathcal E^{\rm Sob}_{B_{R+1}}(u) + R^{n-1}\log R\big) \quad& \mbox{if }s=1\\
  C\big(\mathcal E^{\rm Sob}_{B_{R+1}}(u) + R^{n-1}\big) \quad& \mbox{if }s\in (1,2],\\  
 \end{cases}
\end{equation}
for all $R\ge 1$, where $C$ is a constant which depends only on $n$, $\lambda$, $\Lambda$, and $s_0$.
\end{prop}

For brevity we give the proof only for the archetypal case $- \Delta u + W'(u) =0$, where $W(u) = \frac 1 4 (1-u^2)^2$. 
The same proof can be modified, with not too much effort, to cover also the operators of fractional order, as well as more general double well potentials as in Remark \ref{otherW}.

\begin{proof}[Proof of Proposition \ref{DircontrolsPot1} in the particular case $- \Delta u -(u-u^3) =0$]
Integrating by parts and using that $|\nabla u| \le C$ in $\R^n$ for some dimensional constant $C$, we obtain
\begin{eqnarray*}
I_R & := &\int_{B_R} u^2(1-u^2)^2\,dx \le \int_{B_R} u^2(1-u^2)\,dx = \int_{B_R} u(-\Delta u)\,dx \\
&\le & \int_{B_R} |\nabla u|^2\,dx + CR^{n-1}.
\end{eqnarray*}

Also, letting $\eta_R = 1-(|x|-R)_+$ (note that $\eta_R= 0$ on $\partial B_{R+1}$) and testing the stability inequality  $\int  (1-3u^2) \xi^2\,dx \le \int |\nabla \xi|^2\,dx$ in $B_{R+1}$ with the function $\xi = (1-u^2)\eta_R$, we obtain
\[
\begin{split}
J_R :& = \int_{B_R} (1-u^2)^3\,dx \le \int_{B_{R+1}} (1-3u^2 +2u^2) \big((1-u^2)\eta_R\big)^2 \,dx
\\
&\le\int_{B_{R+1}} \big|\nabla \big((1-u^2)\eta_R\big) \big|^2\,dx + 2\int_{B_{R+1}}u^2(1-u^2)^2\eta_R^2\,dx
\\
&\le \int_{B_{R+1}} 4u^2 |\nabla u|^2\,dx + C |B_{R+1}\setminus B_R|  + 2 I_{R+1} \le 6\int_{B_{R+1}}  |\nabla u|^2\,dx + CR^{n-1}.
\end{split}
\]

Therefore,
\[
\int_{B_R} (1-u^2)^2\,dx  = \int_{B_R} (1-u^2)^2(u^2 +1-u^2)\,dx = I_R + J_R  \le 7\int_{B_{R+1}} |\nabla u|^2\,dx  + CR^{n-1},
\]
as claimed.
\end{proof}

\section{Regularity of solutions}\label{app-C} 

In this appendix, for the reader's convenience, we prove a local smoothness result for  solutions of  semilinear equations involving the fractional Laplacian. Such result is well-know to experts but it is not easy to find in a clean form in the existing literature.  Our goal is to show that every bounded {\em distributional solution}\footnote{We say that a measurable function  $u:\R^n \to \R$   is a distributional of $L_K u =g$ in a domain $\Omega\subset\R^n$  if $\int (uL_K \xi - g\xi)dx=0$ for all $\xi \in C^\infty_c(\Omega)$.} of the fractional semilinear equation \eqref{equK}, for kernels in the class $\mathcal L_2$, satisfies interior $C^{2,\alpha}$ estimates in compact subsets of $\Omega$. As we see next, this is a consequence of known regularity results for linear equations. 
When $\Omega=\R^n$, the situation is simpler than in the following arguments and one can conclude via a bootstrap argument that $u\in C^2(\R^n)$ even for kernels  $K$ in the class $\mathcal L_0$, i.e., kernels satisfying only \eqref{L0}. This is because the equation is posed in all space and one can differentiate it without introducing errors that come from rough exterior data  (in particular the truncation arguments  given in the proof of Propostion \ref{wjhtiowhoipthwio} are not needed).

We will prove the following.

\begin{prop}\label{wjhtiowhoipthwio} 
Let $s_0\in (0,1)$. Assume that  $u: \R^n \to \R$ be a bounded function which solves $L_K u = f(u)$ in $B_1$ in the sense of distributions, with $f\in C^2$ and $K$ satisfying \eqref{L0} and \eqref{L2} for some positive constants  $\lambda$ and $\Lambda$ and for some $s\in (s_0,1)$. 

Then, 
\begin{equation}\label{gwiohioweh}
\| u\|_{C^{2,\alpha} (B_{1/2})} \le C(n, s_0, \lambda, \Lambda,  f, \|u\|_{L^\infty(\R^n)}),
\end{equation}
where $\alpha = \alpha(n, s_0, \lambda, \Lambda)$ is a positive constant. 
\end{prop}

\begin{proof}
Let us show first that if $v$ is a distributional solution  of $L_K v =g$ in $B_{2r}(x_0)$   belonging to $L^\infty(\R^n)$, then it satisfies 
\begin{equation}\label{hiwuguw1}
r^\alpha [v]_{C^\alpha(B_r(x_0))} \le C(n, s_0, \lambda, \Lambda) \big( \|v\|_{L^\infty(\R^n)} + r^{s}\|g\|_{L^\infty(B_{2r}(x_0))}\big).
\end{equation} 
Indeed,  since $L_K$ is translation invariant we can apply the $C^\alpha$ estimate for solutions to integro-differential equations  in  \cite{CS-reg}  to  $v\ast \phi^\ep$ and $g\ast\phi^\ep$, where $\phi^\ep$ is a smooth mollifier and then send $\ep\to 0$.

Second,  if $0<r_1< r_2<r_3$, $B_{r_3}(x_0)\subset B_1$, and $\eta\in C^\infty_c(B_{r_3}(x_0))$, $0\le \eta\le 1$ is some radial cutoff satisfying $\eta\equiv 1$ in $B_{r_2}(x_0)$,
then thanks to the smoothness of the tails of the kernels assumed in \eqref{L2} we have ---see the proof of  Corollary 1.2 in \cite{Ser} for more details---  
\begin{equation}\label{hiwuguw2}
\begin{split}
\| L_K (v\eta) \|_{C^\beta(B_{r_1}(x_0))}  &\le C\big(\|L_Kv\|_{C^\beta(B_{r_3}(x_0))} + \|v \|_{L^\infty (\R^n)}\big)\quad \mbox{and}
\\
\quad \|v\eta\|_{C^\beta(\R^n)} &\le C\|v\|_{C^\beta(B_{r_3}(x_0))} 
\end{split} 
\end{equation} 
 for all $\beta\le 2$, where $C$ depends only on $n$, $\Lambda$, and $r_i$. 
 
Finally, we show that using  \eqref{hiwuguw1}-\eqref{hiwuguw2} we can adapt the standard local  bootstrap argument for semilinear equations to make it work on our nonlocal equation $L_K u = f(u)$ in $B_1$. 
Observe first that since $u\in L^\infty(\R^n)$, applying \eqref{hiwuguw1} to $v=u$ we obtain (up to a scaling and covering argument) $\|u\|_{C^\alpha(B_{1-\varrho})} \le C_1$, where $\varrho>0$.
Our next goal will be to show that, whenever $k\alpha\le 2$, the following implication holds
\begin{equation}\label{wiothewiothw}
 \|u\|_{C^{k\alpha} (B_{1-k\varrho})} \le C_k \quad \Rightarrow\quad  \|u\|_{C^{(k+1)\alpha} (B_{1-(k+1)\varrho})} \le C_{k+1}.
 \end{equation}
 Here, the constants $C_k$ depend only on $n$, $s_0$, $\lambda$, $\Lambda$, $ f$, $\|u\|_{L^\infty(\R^n)}$, and $\varrho>0$.
 
Indeed, let $r_1 := 1-(k+1/2)\varrho$, $r_2 := 1-(k+1/4)\varrho$, and $r_3:=1-k\varrho$ and choose the cut-off $\eta$ as above.  Define $\bar u := u\eta$. Thanks to the assumption in \eqref{wiothewiothw}, and using $f\in C^2$,  
$Lu = f(u)$, and $u \in C^{k\alpha} (B_{r_3})$  we obtain  $f(u)\in C^{k\alpha} (B_{r_3})$. Hence using \eqref{hiwuguw2} we find
we obtain 
\[
\| L_K \bar u \|_{C^\beta(B_{r_1})}  + \|\bar u \|_{C^\beta(\R^n)} \le CC_{k},\quad \mbox{where } \beta:= \alpha k \le 2.
\]

By the $C^\alpha$ estimate in \eqref{hiwuguw1}, used with  $v$ replaced by the incremental quotients (or incremental quotients of derivatives) of order $\beta= \alpha k \le 2$  of  $\bar u$, we obtain $ \|u\|_{C^{(k+1)\alpha} (B_{1-(k+1)\varrho})} =  \|\bar u\|_{C^{(k+1)\alpha} (B_{1-(k+1)\varrho})} \le C_{k+1}$, proving \eqref{wiothewiothw}. Hence after  $N:= 1/\alpha+1$ iterations (taking $\varrho = \frac{1}{2N}$) we obtain   \eqref{gwiohioweh}.
\end{proof}

\end{document}